\documentclass{article}

\usepackage{subfiles} 

\usepackage{commonpackage}
\usepackage{package_for_list_of_content}
\usepackage{package_to_delete}
\usepackage{macro}

\numberwithin{equation}{section}
\theoremstyle{plain}
\newtheorem{theorem}{Theorem}[section]

\newtheorem{corollary}[theorem]{Corollary}
\newtheorem{lemma}[theorem]{Lemma}
\newtheorem{definition}[theorem]{Definition}
\newtheorem*{remark}{Remark}

\newtheorem{claim}[theorem]{Claim}

\title{The dichotomy of Nikodym sets and local smoothing estimates for wave equations}
\author{Mingfeng Chen, Shaoming Guo}
\date{}

\begin{document}
\maketitle

\begin{abstract}
Let $\gamma: \R^2\to \R$ be an analytic function. For $\epsilon>0$, define the maximal operator 
\begin{equation}
\mc{M}_{\gamma, \epsilon}f(x, y):=\sup_{|v|\le \epsilon}
\anorm{
\int_{-\epsilon}^{\epsilon} f(x-\theta, y-\gamma(v, \theta))d\theta
}.
\end{equation}
\begin{enumerate}
\item We introduce a strongly degenerate condition for $\gamma$ and show that $\gamma$ is not strongly degenerate if and only if there exists $\epsilon>0$ such that 
\begin{equation}
\norm{
\mc{M}_{\gamma, \epsilon} f
}_{L^p(\R^2)}\le C_{\gamma, p} \norm{f}_{L^p(\R^2)}
\end{equation}
for some $p<\infty$ and some constant $C_{\gamma, p}<\infty$ depending only on $\gamma$ and $p$. 
\item 
For every $\gamma$ that is not strongly degenerate, we provide an algorithm that outputs a critical exponent $p_{\gamma}\in [2, \infty)$, and show that $\mc{M}_{\gamma, \epsilon}$ is bounded on $L^p(\R^2)$ whenever $p>p_{\gamma}$. Moreover, this exponent $p_{\gamma}$ is sharp in the sense that $\mc{M}_{\gamma, \epsilon}$ is  unbounded for every $p< p_{\gamma}$. Key ingredients in the proof are local smoothing estimates for linear wave equations in $\R^3$, due to Mockenhaupt, Seeger and Sogge \cite{MSS92}. 
\item For every $\gamma$ that is strongly degenerate, there exists a Nikodym set for $\gamma$, constructed by Chang and Cs\"ornyei \cite{CC16}, that can be used to show that $\mc{M}_{\gamma, \epsilon}$ fails to be bounded on $L^p(\R^2)$ for every $p<\infty$. 
\end{enumerate}
In other words, we show that Nikodym sets and local smoothing estimates for linear wave equations form a dichotomy: If Nikodym sets for $\gamma$ exist, then $\mc{M}_{\gamma, \epsilon}$ is not bounded on $L^p(\R^2)$ for any $p<\infty$; if Nikodym sets for $\gamma$ do not exist, then local smoothing estimates hold, and $\mc{M}_{\gamma, \epsilon}$ is  bounded on $L^p(\R^2)$ for some $p<\infty$. 
\end{abstract}

\setcounter{tocdepth}{2}
\tableofcontents

\section{Introduction}

Let $\gamma$ be an analytic function depending on two variables. We will consider the averaging operator 
\begin{equation}
    T_{\gamma}f(v; x,y):=\int_{\R} f(x-\theta, y-\gamma(v, \theta))a(v, \theta)d\theta
\end{equation}
where $a: \R^2\to \R$ is a smooth function supported in a sufficiently small neighborhood of the origin and the size of the neighborhood is allowed to depend on $\gamma$. Define the corresponding maximal operator:
\begin{align}
    \mathcal{M}_{\gamma}f(x,y):=\sup_{v\in \R}|T_{\gamma} f(v; x, y)|. 
\end{align}
We are interested in the $L^p$-bounds for the maximal operator. More precisely, 
\begin{enumerate}
\item[(1)] We introduce a strongly degenerate condition for $\gamma$ (see Definition \ref{231122defi1_5}), and show that $\mc{M}_{\gamma}$ is bounded on $L^p(\R^2)$ for some $p<\infty$ if and only if $\gamma$ is not strongly degenerate;
\item[(2)] For each $\gamma$ that is not strongly degenerate, we provide an algorithm that outputs a critical exponent $p_{\gamma}\in [2, \infty)$, and show that $\mc{M}_{\gamma}$ is bounded on $L^p(\R^2)$ whenever $p>p_{\gamma}$, and unbounded whenever $p<p_{\gamma}$. 
\end{enumerate}

By translations and shearing transformations, we can without loss of generality assume that 
\begin{equation}\label{231122e1_3kk}
\gamma(0, 0)=0, \ \ \partial_{\theta} \gamma(0, 0)=0.
\end{equation}
To avoid trivial discussions, we also assume that $\gamma$ is neither constant in $v$ nor constant in $\theta$. In other words, we assume that there exists $\mf{p}>0$ and $\mf{q}>0$ such that 
\begin{equation}\label{231122e1_4kk}
c_{\mf{p}, \mf{q}} \neq 0,
\end{equation}
where 
\begin{equation}
c_{\mf{p}, \mf{q}}:=\frac{1}{\mf{p} ! \mf{q} !} \partial_{v}^{\mf{p}} \partial_{\theta}^{\mf{q}} \gamma(0,0).
\end{equation}
The notation $c_{\mf{p}, \mf{q}}$ will be used throughout the paper. \\

We introduce several basic notions, see for instance Varchenko \cite{Var76}. The \underline{Taylor support} of $\gamma(v, \theta)$ at the origin is defined to be 
\begin{equation}
\mathcal{T}(\gamma):=\left\{(\mfp, \mfq) \in \mathbb{N}^2: \frac{1}{\mfp ! \mfq !} \partial_{v}^{\mfp} \partial_{\theta}^{\mfq} \gamma(0,0) \neq 0\right\}.
\end{equation}
The \underline{Newton polyhedron} $\mathcal{N}(\gamma)$ of $\gamma$ at the origin is defined to be the convex hull of the union of all the quadrants
\begin{equation}
(\mfp, \mfq)+\mathbb{R}_{+}^2
\end{equation}
 in $\mathbb{R}^2$, with $(\mfp, \mfq) \in \mathcal{T}(\gamma)$. The associated \underline{Newton diagram} $\mathcal{N}_d(\gamma)$ is the union of all compact edges of the Newton polyhedron. The \underline{reduced Newton polyhedron} $\rn(\gamma)$ of $\gamma$ at the origin is defined to be the convex hull of the union of all the quadrants $$(\mfp, \mfq)+\mathbb{R}_{+}^2$$ in $\mathbb{R}^2$, with $(\mfp, \mfq) \in \mathcal{T}(\gamma)$ and $\mfq\ge 1$. Similarly, we define the \underline{reduced Newton} \underline{diagram} $\rn_d(\gamma)$.  \\

    We define the cinematic curvature of $\gamma$ at $(v, \theta)$ to be 
    \begin{align}
        \cine(\gamma)(v, \theta):=\det \begin{bmatrix}
            \gamma_{\theta\theta} & \gamma_{\theta\theta\theta} \\
            \gamma_{v \theta} & \gamma_{v\theta\theta}
        \end{bmatrix}.
    \end{align}
Cinematic curvatures are defined in Sogge \cite{Sog91} for more general $\gamma$ that is allowed to depend also on $x, y$. Here we are considering a special case where $\gamma$ is translation-invariant, that is, independent of $x, y$.

 \subsection{Strongly degenerate condition}

\begin{definition}\label{231122defi1_5}
Let $\gamma$ be an analytic function satisfying \eqref{231122e1_3kk} and \eqref{231122e1_4kk}. We say that $\gamma$ is strongly degenerate at the origin if at least one of the following three conditions holds: 
\begin{enumerate}
\item[(1)] There exists $\epsilon>0$ such that 
\begin{equation}
\cine(\gamma)(v, \theta)=0, \ \ \forall (v, \theta)\in \B_{\epsilon},
\end{equation}
where $\B_{\epsilon}\subset \R^2$ is the ball of radius $\epsilon$ centered at the origin.
\item[(2)] The reduced Newton diagram of $\gamma$ consists of only one vertex $(\mfp_0, 1)$, with $\mfp_0\ge 1$, and there exists $\epsilon>0$ such that 
\begin{equation}
\cine(\widetilde{\gamma})(v, \theta)=0, \ \ \forall (v, \theta)\in \B_{\epsilon},
\end{equation}
where 
\begin{equation}\label{231122e1_8}
\widetilde{\gamma}(v, \theta):=
\sum_{\mfq\ge 1: (\mfp_0, \mfq)\in \mc{T}(\gamma)} 
c_{\mf{p}, \mf{q}} v^{\mfp_0} \theta^{\mfq}.
\end{equation}
\item[(3)] The function $\gamma(v, \theta)$ is of the form 
\begin{equation}
    \gamma(v, \theta)= c_{\mf{p}, 0} v^{\mfp}+ \widetilde{\gamma}(v, \theta), 
\end{equation}
where $c_{\mf{p}, 0}\neq 0$, $\mfp\ge 1$ and 
\begin{equation}
    |\widetilde{\gamma}(v, \theta)|=O(|v|^{\mfp+1}).
\end{equation} 
\end{enumerate}
\end{definition}

Now we state our first main theorem. To avoid trivial $\gamma$, we let $\gamma$ be an analytic function satisfying the assumptions \eqref{231122e1_3kk} and \eqref{231122e1_4kk}. 

\begin{theorem}\label{240127theorem1_2}
Let $\gamma(v, \theta)$ be an analytic function near the origin. Then $\gamma$ is not strongly degenerate at the origin if and only if there exists $\epsilon>0$ such that 
\begin{equation}\label{1.10}
\norm{
\mc{M}_{\gamma} f
}_{
L^p(\R^2)
} \lesim_{\gamma, a, p} \norm{f}_{L^p(\R^2)},
\end{equation}
for some $p<\infty$, and for all smooth functions $a$ supported in $\B_{\epsilon}$. 
\end{theorem}

If $\gamma$ is a function given by item (3) in Definition \ref{231122defi1_5}, then by a simple scaling argument in the $v$ variable, it is elementary to see that $\mc{M}_{\gamma}$ is not bounded on $L^p(\R^2)$ for every $p<\infty$. For functions in item (1) and item (2), to show that $\mc{M}_{\gamma}$ is not bounded on $L^p(\R^2)$ for every $p<\infty$, 
we will need Nikodym sets constructed in \cite{CC16}. \\

Let us also make a remark regarding the ``only if" part of the theorem. The proof of \eqref{1.10} must respect all the strongly degenerate examples in Definition \ref{231122defi1_5}. Moreover, the more strongly degenerate examples we have, the more complicated the proof of \eqref{1.10} will be; the more classes of strongly degenerate examples we have, the more cases we will need to discuss separately in the proof of Theorem \ref{240127theorem1_2}. This gives an excuse to the case distinctions in the proof of Theorem \ref{240127theorem1_2} (see for instance  Section \ref{240124section4}). \\

In the end of this subsection, let us give several concrete examples that are strongly degenerate. 
\begin{enumerate}
    \item Take 
    \begin{equation}\label{240125e1_15}
        \gamma_1(v, \theta)=v\theta,
    \end{equation}
    and 
    \begin{equation}\label{240125e1_16}
        \gamma_2(v, \theta)=
        v(e^{\theta}-1).
    \end{equation}
    Note that the cinematic curvatures of $\gamma_1$ and $\gamma_2$ both vanish constantly, and therefore they are strongly degenerate. 

        \item 
    Take 
    \begin{align}\label{240125e1_18}
        \gamma_1(v,\theta)=v\theta+v^2\theta^2,
    \end{align}
    and \begin{equation}\label{240125e1_17}
        \gamma_2(v, \theta)=
        v(e^{\theta}-1)+ v^2\theta.
    \end{equation}
    Note that neither the cinematic curvature of $\gamma_1$ nor the cinematic curvature of $\gamma_2$ is constantly zero. However, the reduced Newton diagrams of both $\gamma_1$ and $\gamma_2$ consist of only one vertex $(1, 1)$, the cinematic curvatures for 
    $v\theta$ and 
    $v(e^{\theta}-1)$ both vanish constantly, and therefore $\gamma_1$ and $\gamma_2$ are still strongly degenerate.

\end{enumerate}

We will see later in Lemma \ref{231122lemma2_1} that the examples \eqref{240125e1_15}--\eqref{240125e1_17} are essentially the only examples for item (2) in Definition \ref{231122defi1_5}. Characterizing item (1) in Definition \ref{231122defi1_5} is much more complicated, and this is partially done in Lemma \ref{231123lemma2_4}.

\subsection{Vertical Newton distance}\label{240129subsection1_2}

The goal of this subsection is to introduce an algorithm that outputs, for each $\gamma$ that is not strongly degenerate, a sharp exponent  $p_{\gamma}\in [2, \infty)$ for the $L^p$ bounds of $\mc{M}_{\gamma}$.  Let $\gamma$ be an analytic function satisfying the assumptions \eqref{231122e1_3kk} and \eqref{231122e1_4kk}. We start by introducing a few relevant notions. \\

Let $(\mfp, \mfq)$ be a vertex in $\mc{R}\mc{N}_{d}(\gamma)$. 
A line on $\R^2$ is said to be \underline{vertex-tangent} to $\gamma$ at $(\mfp, \mfq)$ if it passes through  $(\mfp, \mfq)$ and does not pass through any other point in the reduced Newton polyhedron.  

We use $\mf{L}_{(
\mf{p}, \mf{q}
)}(\gamma)$ to denote the collection of all these lines. If it is clear from the context which $\gamma$ is involved, we often abbreviate $\mf{L}_{(
\mf{p}, \mf{q}
)}(\gamma)$ to $\mf{L}_{(
\mf{p}, \mf{q}
)}$.

Given a line 
\begin{equation}
\mf{L}=\{(x, y): a x+ by=c\}, \text{ with } a>0, b>0, c>0,
\end{equation}
and a point $(\mf{p}, 0)$ with $\mf{p}>0$, we define their \underline{vertical distance} to be 
\begin{equation}
\mf{d}(
\mf{L}, (\mf{p}, 0)
):=
\begin{cases}
0 \hfill & \hfill \text{ if }  a\mf{p}\ge c;\\
\frac{c-a\mf{p}}{b} \hfill
& \hfill \text{ if } a\mf{p}< c.
\end{cases}
\end{equation}
In particular, if $\mf{p}=+\infty$, then we make a convention that 
\begin{equation}
        \mf{d}(
    \mf{L}, (\mf{p}, 0)
    )=0,
\end{equation}
for every $\mc{L}$. 

\begin{definition}[Vertical Newton distance]\label{vertical_newton}
Assume $\gamma$ is not strongly degenerate. 
 Let $\mfp_0$ be the smallest $\mfp$ such that $c_{\mfp, 0}\neq 0$. If such $\mfp$ does not exist, that is, $c_{\mfp, 0}=0$ for all $\mfp$, then we write $\mfp_0=\infty$. Define 
\begin{equation}
\mf{d}(\gamma):=
\sup_{
(\mf{p}, \mf{q}): \mathrm{vertex} \mathrm{\ of\ } \mc{R}\mc{N}_d(\gamma)
}
\sup_{
\mf{L}\in \mf{L}_{(\mf{p}, \mf{q})}
}
\mf{d}(
\mf{L}, (
\mf{p}_0, 0
)
).
\end{equation}
Take $(\mf{p}, \mf{q})\in \mc{R}\mc{N}_d(\gamma)$. Define  
\begin{equation}\label{240124e1_18}
\mf{d}_{>(\mf{p}, \mf{q})}(\gamma):=
\sup_{
\substack{
(\mf{p}', \mf{q}'): \mathrm{vertex} \mathrm{\ of\ } \mc{R}\mc{N}_d(\gamma)\\
\mf{p}'> \mf{p}
}
}
\sup_{
\mf{L}\in \mf{L}_{(\mf{p}', \mf{q}')}
}
\mf{d}(
\mf{L}, (
\mf{p}_0, 0
)
).
\end{equation}
If the sup in \eqref{240124e1_18} is taken over an empty set, that is, there is no vertex $(\mf{p}', \mf{q}')\in \mc{R}\mc{N}_d(\gamma)$ with $\mf{p}'> \mf{p}$, then we define\footnote{This convention will be important in Subsection \ref{240125subsection5_3}.}
\begin{equation}\label{240125e1_19}
    \mf{d}_{>(\mf{p}, \mf{q})}(\gamma):=
    \begin{cases}
        \mf{q}, & \text{ if } \mf{p}_0< \infty;\\
        0, & \text{ if } \mf{p}_0=\infty. 
    \end{cases}
\end{equation}
\end{definition}

The definition of the vertical Newton distance can be better explained in a picture. 

\includegraphics[scale=0.6]{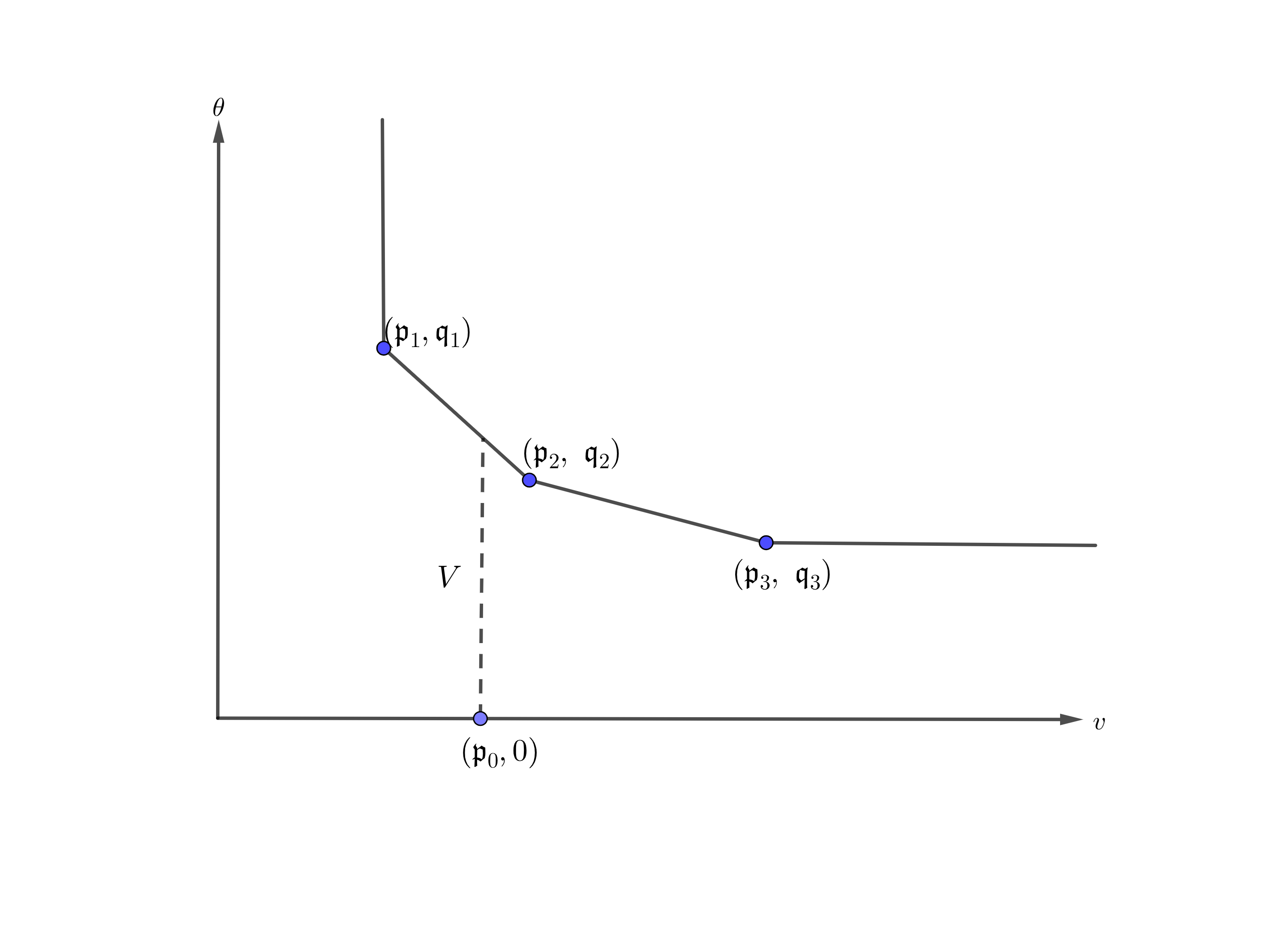}

In the picture above, the reduced Newton diagram has three vertices $(\mf{p}_i, \mf{q}_i), i=1, 2, 3$. The dashed vertical line segment is called $V$, and its length $|V|$ is the vertical Newton distance $\mf{d}(\gamma)$. The reason for introducing the form in Definition \ref{vertical_newton} will become clear later when we try to prove sharp $L^p(\R^2)$ bounds for maximal operators. \\

Now we are ready to determine the sharp exponent $p_{\gamma}\ge 2$. We initialize 
\begin{equation}\label{240129e1_24kk}
    \mf{D}_{\gamma}:= \mf{d}(\gamma).
\end{equation}
The value of $\mf{D}_{\gamma}$ will keep being updated as we run our algorithm. We label all the vertices of the reduced Newton diagram $\mc{R}\mc{N}_d(\gamma)$ by 
\begin{equation}
    (
    \mf{p}_1, \mf{q}_1
    ), (
    \mf{p}_2, \mf{q}_2
    ), \dots
\end{equation}
with 
\begin{equation}
    \mf{p}_1< \mf{p}_2< \dots
\end{equation}
Let $\mf{E}_i$ denote the edge of $\mc{R}\mc{N}_d(\gamma)$ connecting the vertices $(
    \mf{p}_i, \mf{q}_i
    ), (
    \mf{p}_{i+1}, \mf{q}_{i+1}
    )$. 
For later use, we express the line containing $\mf{E}_i$ as 
\begin{equation}\label{240129e1_27kk}
    \{(x, y):
    a_{\mf{E}_i} x+ b_{\mf{E}_i} y= c_{\mf{E}_i}
    \},
\end{equation}
where 
\begin{equation}
    a_{\mf{E}_i}>0, \ b_{\mf{E}_i}>0, \ c_{\mf{E}_i}>0.
\end{equation}
We consider all the edges $\mf{E}_1, \mf{E}_2, \dots$ one after another. Suppose we are currently considering the edge $\mf{E}_i$. Denote 
\begin{equation}
    \mf{m}_i:=
    \frac{
    \mf{p}_i-\mf{p}_{i+1}
    }{
    \mf{q}_{i+1}-\mf{q}_i
    }.
\end{equation}
Denote 
\begin{equation}
    \kappa_{
    \mf{E}_i
    }(v, \theta):=
    \sum_{
    (\mf{p}, \mf{q})\in \mf{E}_i, \mf{q}\ge 2
    }c_{
    \mf{p}, \mf{q}
    }
    \mf{q}(\mf{q}-1) v^{\mf{p}}\theta^{\mf{q}-2},
\end{equation}
which is a mixed-homogeneous polynomial. 
Note that $ \kappa_{
    \mf{E}_i
    }(v, \theta)$ is the second derivative in the $\theta$ variable of the mixed-homogeneous polynomial 
    \begin{equation}
    \gamma_{
    \mf{E}_i
    }(v, \theta):=
    \sum_{
    (\mf{p}, \mf{q})\in \mf{E}_i
    }c_{
    \mf{p}, \mf{q}
    }
     v^{\mf{p}}\theta^{\mf{q}}.
\end{equation}
We introduce a variable $r$, and write 
\begin{equation}
    \kappa_{\mf{E}_i}(v, 
    r v^{
    \mf{m}_i
    }
    )
    =:
    v^{\mf{e}_i}
    \kappa_{\mf{E}_i}(r),
\end{equation}
where 
\begin{equation}\label{240129e1_30pp}
    \mf{e}_i:= \mf{p}+ (\mf{q}-2) \mf{m}_i,
\end{equation}
and in \eqref{240129e1_30pp}, the point $(\mf{p}, \mf{q})$ can be taken to be an arbitrary point from $\mf{E}_i$. We label all the non-zero distinct real roots of the polynomial $\kappa_{\mf{E}_i}(r)$ by 
\begin{equation}\label{240129e1_31pp}
    \mf{r}_{i, 1}, \mf{r}_{i, 2}, \dots
\end{equation}
Consider all the roots in \eqref{240129e1_31pp} one after another. Suppose we are currently considering the root $\mf{r}_{i, j}$. We apply the change of variable
\begin{equation}
    v\to v, \ \ \theta\to \theta+ \mf{r}_{i, j} v^{
    \mf{m}_i
    },
\end{equation}
and obtain 
\begin{equation}
    \gamma_{\star}(v, \theta):=
    \gamma(v, 
    \theta+
    \mf{r}_{i, j} v^{
    \mf{m}_i
    }
    ). 
\end{equation}
Consider the reduced Newton diagram for $\gamma_{\star}$. To make it easier to think of the reduced Newton diagram $\mc{R}\mc{N}_d(\gamma_{\star})$, we mention several of its simple properties. For instance, the vertex $(\mf{p}_i, \mf{q}_i)$ is still a vertex in $\mc{R}\mc{N}_d(\gamma_{\star})$; for every vertex $(\mf{p}, \mf{q})\in \mc{R}\mc{N}_d(\gamma_{\star})$, we always have 
\begin{equation}
    a_{
    \mf{E}_i
    } \mf{p}
    + b_{
    \mf{E}_i
    }
    \mf{q}\ge c_{
    \mf{E}_i
    },
\end{equation}
that is, $(\mf{p}, \mf{q})$ never lies to the lower left of the line containing $\mf{E}_i$. The proofs for these two properties are simple and elementary, and details can be found in Lemma \ref{240124lemma5_1} and Lemma \ref{240124lemma5_2}. 

Now update 
\begin{equation}\label{240129e1_37kk}
    \mf{D}_{\gamma}:=
    \max\{
    \mf{D}_{\gamma}, 
    \mf{d}_{
    >(\mf{p}_i, \mf{q}_i)
    }(\gamma_{\star})
    \}.
\end{equation}
Next, we label all the vertices $(\mf{p}, \mf{q})\in \mc{R}\mc{N}_d(\gamma_{\star})$ with $\mf{p}> \mf{p}_i$ by 
\begin{equation}
    (
    \mf{p}_{\star+1}, \mf{q}_{\star+1}
    ), (
    \mf{p}_{\star+2}, \mf{q}_{\star+2}
    ), \dots 
\end{equation}
where 
\begin{equation}
    \mf{p}_{\star+1}< \mf{p}_{\star+2}< \dots
\end{equation}
Let $\mf{E}_{\star+\iota}$ be the edge connecting the vertices 
\begin{equation}\label{240129e1_40kk}
    (
    \mf{p}_{\star+\iota}, \mf{q}_{\star+\iota}
    ), (
    \mf{p}_{\star+\iota+1}, \mf{q}_{\star+\iota+1}
    ).
\end{equation}
Moreover, if the vertex $(
    \mf{p}_{\star+1}, \mf{q}_{\star+1}
    )$ does not lie on the line containing $\mf{E}_i$, then we let $\mf{E}_{\star+0}$ be the edge connecting the vertices 
    \begin{equation}
        (
        \mf{p}_{i}, \mf{q}_i
        ), (
    \mf{p}_{\star+1}, \mf{q}_{\star+1}
    ).
    \end{equation}
Note that the slope of the edge $\mf{E}_{\star+0}$ is always strictly smaller than the slope of the edge $\mf{E}_i$. For later use, let us denote 
\begin{equation}\label{240129e1_38kk}
    \gamma_{\star+0}(v, \theta):=
    \sum_{
    \substack{
    (\mf{p}, \mf{q})\in \mf{E}_{\star+0}
    }
    }
    \frac{1}{\mf{p}!\mf{q}!}
    \partial_v^{\mf{p}}
    \partial_{\theta}^{\mf{q}}\gamma_{\star}(0, 0)
    v^{\mf{p}}\theta^{\mf{q}}.
\end{equation}
For every edge $\mf{E}_{\star+\iota}, \iota=0, 1, 2, \dots$, we repeat the steps in \eqref{240129e1_27kk}-\eqref{240129e1_38kk}. The algorithm terminates if we run out of edges to consider.   \\

Readers familiar with resolutions of singularities (for instance Varchenko's paper \cite{Var76}) must have long realized that this is exactly what we are doing here. The precise argument above is taken from Xiao's paper \cite{Xia17}, and it is a variant of the algorithm in Greenblatt \cite{Gre04}. We will remark on the two slightly different versions of resolutions of singularities in \cite{Gre04} and \cite{Xia17} later; let us first finish the above algorithm. \\

It is proven in Xiao \cite{Xia17} (see Theorem 4.11 there) that we have exactly one of the following two scenarios. 
\begin{enumerate}
    \item[Scenario One.] The algorithm in \eqref{240129e1_24kk}-\eqref{240129e1_38kk} terminates after finitely many steps;
    \item[Scenario Two.] The algorithm does not terminate after finitely many steps. In this scenario, we know that starting from some stage we must always have
    \begin{equation}\label{240129e1_42kkk}
        \gamma_{\star+0}(v, \theta)=c_{\star} v^{b_{\star}}
        (\theta+ \mf{r}_n v^{\mf{w}_n})^{W}
        -
        c_{\star} v^{b_{\star}}
        (\mf{r}_n v^{\mf{w}_n})^{W}
        ,
    \end{equation}
    where $n=1, 2, \dots,$ $b_{\star}\ge 0, c_{\star}\neq 0, \ \mf{r}_{n}\neq 0, n=1, 2, \dots,$ 
    the exponents $\mf{w}_1, \mf{w}_2, \dots$ are non-zero positive rational numbers that can be written as fractions sharing a common denominator, and $W$ is a positive integer. 
\end{enumerate}
If we are in Scenario One, then $\mf{D}_{\gamma}$ is clearly defined. If we are in Scenario Two, instead of doing the change of variable
\begin{equation}
    \theta+ \mf{r}_n v^{
    \mf{w}_n
    }\to \theta, \ n=1, 2, \dots,
\end{equation}
we should do the above change of variables all at once, that is, 
\begin{equation}
    \theta+ \mf{r}_1v^{
    \mf{w}_1
    }+\mf{r}_2v^{
    \mf{w}_2
    }+\dots\to \theta,
\end{equation}
which forces the algorithm to terminate immediately. 
This finishes defining the critical exponent $p_{\gamma}$. Examples for the above algorithm are postponed to Section \ref{240222section3}.

\begin{remark}
    We remark on the two algorithms of resolutions of singularities in \cite{Xia17}
    and \cite{Gre04}. As mentioned above, the one in \cite{Xia17} is a variant of the one in \cite{Gre04}. A big advantage of the algorithm in \cite{Gre04} is that it always terminates after finitely many steps. If we use \eqref{240129e1_42kkk} to explain, then what \cite{Gre04} does is to do the change of variable 
    \begin{equation}\label{240129e1_45hhh}
        v\to v, \ \ \theta\to \theta+ \mf{r}_1 v^{\mf{w}_1}+ \mf{r}_2 v^{\mf{w}_2}+\dots
    \end{equation}
    directly, which terminates the algorithm immediately. However, \eqref{240129e1_45hhh} involves functions in $v$ that are not homogeneous, and the algorithm in \cite{Gre04} applies the implicit function theorems to detect such functions. The reason we apply Xiao's algorithm is that it involves only monomials in $v$, which makes it slightly easier to compute $p_{\gamma}$ for concrete examples. 

    Later when proving $L^p$ bounds for the maximal operator $\mc{M}_{\gamma}$, we will also apply resolutions of singularities to the cinematic curvature of $\gamma$. There we do not need to compute any exponent like $p_{\gamma}$ explicitly, and it is much more convenient to apply the algorithm in Greenblatt \cite{Gre04}, which is reviewed in Section \ref{240129section4kk}. 
\end{remark}

\begin{theorem}\label{240127theorem1_6}
    For every $\gamma$ that is not strongly degenerate, we let $p_{\gamma}\in [2, \infty)$ be the exponent given by the above algorithm. Then there exists $\epsilon>0$ such that 
    \begin{equation}\label{240128e1_47ppp}
        \norm{
\mc{M}_{\gamma} f
}_{
L^p(\R^2)
} \lesim_{\gamma, a, p} \norm{f}_{L^p(\R^2)},
    \end{equation}
    holds for all $p>p_{\gamma}$ and all smooth function $a$ supported on $\B_{\epsilon}$. Moreover, the exponent $p_{\gamma}$ is sharp. 
\end{theorem}

A few special (but important) cases of Theorem \ref{240127theorem1_6} were already obtained in the literature. The most well-known case is 
\begin{equation}\label{240127e1_46pp}
    \gamma(v, \theta)= v(1+\theta^2).
\end{equation}
In the breakthrough works \cite{Bou85} and \cite{Bou86}, Bourgain proved that $p_{\gamma}=2$. Bourgain used both geometrical and Fourier analytic methods. Later, Mockenhaupt, Seeger and Sogge \cite{MSS92} discovered local smoothing estimates for linear wave equations and gave a simple proof of Bourgain's result which uses only Fourier analysis. Moreover, Schlag \cite{Sch98} gave a purely geometric proof for Bourgain's theorem, and his method also yielded sharp $L^p\to L^q$ bounds for the maximal operator, see Schlag \cite{Sch97}. The results in \cite{Sch97} were recovered by Schlag and Sogge in a later paper \cite{SS97}. 

We would like to  emphasize that 
all the proofs of Bourgain \cite{Bou86}, Mockenhaupt, Seeger and Sogge \cite{MSS92} and Schlag \cite{Sch98} work well for perturbations of \eqref{240127e1_46pp}. More precisely, if we take 
\begin{equation}
    \gamma(v, \theta)=v(1+h_2(\theta)),
\end{equation}
and assume that the second order derivative of $h_2$ at the origin does not vanish, then we still have $p_{\gamma}=2$.

A very natural problem to consider after Bourgain's work is what happens if we allow the second order derivative of $h_2$ to vanish near the origin. This problem was resolved by Iosevich \cite{Ios94}, where he considered functions of the form similar to 
\begin{equation}\label{240127e1_48pp}
    \gamma(v, \theta)=v(1+\theta^k), \ k\in \N, k\ge 2,
\end{equation}
and proved that $p_{\gamma}=k$. Note that for $\gamma$ in \eqref{240127e1_48pp}, its cinematic curvature vanishes at $\theta=0$. 

More recently, Li \cite{Li18} studied the problem where the cinematic curvature is even more degenerate than that of Iosevich \cite{Ios94}.  One typical example considered in \cite{Li18} can be written as 
\begin{equation}
    \gamma(v, \theta)=
    \theta^2+ v \theta^{k}, \ \ k\ge 3,
\end{equation}
and Li \cite{Li18} showed that $p_{\gamma}=2$. Let us also emphasize that both the result in \cite{Ios94} and the result in \cite{Li18} are stable under perturbations, similarly to the result of Bourgain \cite{Bou86}. \\

Bourgain's results in \cite{Bou85} and \cite{Bou86} were also generalized to the variable coefficient setting, where the function $\gamma$ starts to depend also on $x$ and $y$. This setting includes maximal operators on Riemannian manifolds as special cases. We refer to Sogge \cite{Sog91} where the notion of cinematic curvatures first appeared, and Mockenhaupt, Seeger and Sogge \cite{MSS93} for a general theory of related Fourier integral operators. 

It may be very interesting to work out how our two main theorems, Theorem \ref{240127theorem1_2} and Theorem \ref{240127theorem1_6}, look like in the variable coefficient setting. Seeger \cite{See93} and \cite{See98} proved sharp Sobolev estimates for fixed time (in our notation, for fixed $v$ parameter) for general planar Fourier integral operators, and these will be crucial in establishing a version Theorem \ref{240127theorem1_6} in the variable coefficient setting. \\

Bourgain's results in \cite{Bou85} and \cite{Bou86} proved the non-existence of certain Nikodym sets. Recall that $E\subset [0, 1]^2$ is called a (classical) Nikodym set if the Lebesgue measure of $E$ is zero, and for every $\bfx\in [0, 1]^2\setminus E$, there exists a line $\mf{L}_{\bfx}$ satisfying 
\begin{equation}
    (\mf{L}_{\bfx}\setminus \{\bfx\})\cap E^c=\emptyset. 
\end{equation}
If we replace lines by circles in the definition of Nikodym sets, then Bourgain \cite{Bou85} and \cite{Bou86} and Marstrand \cite{Mar87} independently proved that Nikodym sets for circles do not exists.  Let $E\subset\R^2$ be a measurable set. Define 
\begin{equation}
    E_{\circ}:=\{\bfx\in \R^2: E \text{ contains a circle centered at } \bfx\}.
\end{equation}
If the Lebesgue measure of $E_{\circ}$ is positive, then $E$ itself must also have positive Lebesgue measure, in a strong contrast to the case of the classical Nikodym sets. 

We next discuss Nikodym sets for general $\gamma$. We say that $E\subset \R^2$ is a $\gamma$-Nikodym set if $E$ has Lebesgue measure zero and there exists $\epsilon>0$ depending on $\gamma$ such that $E_{\gamma, \epsilon}$ has a positive Lebesgue measure, where 
\begin{multline}
    E_{\gamma, \epsilon}:= \{
    \bfx\in \R^2: \exists v\in (-\epsilon, \epsilon) \text{ such that } \\
    \mc{L}^1(
    \{
    |\theta|< \epsilon: \bfx+ (\theta, \gamma(v, \theta))\in E
    \}
    )>0
    \}.
\end{multline}
Here $\mc{L}^1$ refers to the one-dimensional Lebesgue measure.  A direct corollary of Theorem \ref{240127theorem1_6} is as follows. 
\begin{corollary}
    If $\gamma$ is not strongly degenerate, then $\gamma$-Nikodym sets do not exist.
\end{corollary}
Moreover, if $\gamma$ is strongly degenerate, then we will see in Section \ref{240128section2} that $\gamma$-Nikodym sets always exist and have been constructed in \cite{CC16}. 

When proving Theorem \ref{240127theorem1_6},  key ingredients are local smoothing estimates that generalize the ones in \cite{MSS92}. In this sense, we say that Nikodym sets and local smoothing estimates form a dichotomy.

\vspace{1cm}

\noindent {\bf Notation.} We collect some of the  notations that were used in the introduction and will be used in the rest of the paper. 
\begin{enumerate}
    \item We often use boldface letters to refer to vectors. For instance, we often write $\bfx=(x, y)$ and  $\bxi=(\xi, \eta)$.   
    \item When considering $v, \theta$ at the same time, the order we write is always $(v, \theta)$. We will look at dyadic decompositions, and we always use $2^{-j_1}$ for the scale of $v$ and $2^{-j_2}$ for the scale of $\theta$. We write $\bfj=(j_1, j_2)$. 
    \item We use $P_k$ to denote a Littlewood-Paley projection in the second variable to frequencies $\{(\xi, \eta): |\eta|\simeq 2^k\}$, and $P'_k$ in the first variable. The reason behind this choice of notation is the parametrization of the curves by $(\theta, \gamma(v, \theta))$, which makes the second variable a primary frequency direction to consider, and the first variable a secondary frequency direction to consider. We use $\bfP_{\bfk}$ for $P'_{k_1}P_{k_2}$, where $\bfk=(k_1, k_2)$. 
    \item For the function $\gamma(v, \theta)$, we are only concerned with $(v, \theta)$ in a small neighborhood of $(0, 0)$, unless otherwise specified. 
    \item For $\epsilon>0$, we let $\B_{\epsilon}\subset \R^2$ be the ball of radius $\epsilon$ centered at the origin. 
    \item Let $C>1$ and $I\subset \R$ be an interval. We use $C I$ to mean the interval of length $C|I|$ that has the same center as $I$.
    \item Through the paper, $c_{\mf{p}, \mf{q}}$ is always reserved for 
    \begin{equation}
    c_{\mfp, \mfq}:=
\frac{1}{\mfp!\mfq!}
\partial_v^{\mfp} \partial_{\theta}^{\mfq} 
\gamma(0, 0),
    \end{equation}
    where $\mf{p}, \mf{q}\in \N$. 
    \item For a function $a(x, y)$, we use $\supp(a)$ to denote its support. We will use $x\in \supp(a)$ to mean that there exists some $y$ such that $(x, y)\in \supp(a)$.  
    \item We use $A\lesim B$ to mean that $\exists C>0$, such that $A\leq CB$. $A\simeq B$ means $A\lesim B$ and $B\lesim A$.
    \item Unless otherwise specified, every implicit constant in this paper is allowed to depend on $\gamma$, which we often suppress from the notation. 
    \item For an analytic function $\gamma(v, \theta)$, we use $\gamma\equiv 0$ to mean that $\gamma$ is identically zero. Similarly, we use $\gamma\not\equiv 0$ to mean that $\gamma$ is not identically zero. Moreover, $\gamma(v, \cdot)\equiv_{\theta} 0$ means that for the given $v$, the function $\gamma(v, \theta)$ vanishes identically as a function of $\theta$. Similarly, we define $\gamma(v, \cdot)\not\equiv_{\theta} 0$. 
\end{enumerate}

\bigskip

\noindent {\bf Acknowledgements.} 
S.G. is partly supported by NSF-2044828. The authors would like to thank Jacob Denson for discussions related to cinematic curvatures, thank Shaozhen Xu for many helpful discussions related to resolutions of singularities, and thank Alan Chang for explaining to them the Nikodym sets constructed in \cite{CC16}. They also thank Andreas Seeger and Josh Zahl for sharing their thinkings on possible forms of the dichotomy of Nikodym sets and local smoothing estimates.

\section{Strongly degenerate conditions and Nikodym sets}\label{240128section2}

In this section, we will first characterize strongly degenerate functions, and then apply a result in \cite{CC16} to show that $\gamma$-Nikodym sets exist whenever $\gamma$ is strongly degenerate. In particular, this proves the ``if" part of Theorem \ref{240127theorem1_2}, that is, if $\gamma$ is strongly degenerate, then for every $\epsilon<0$ and every $p<\infty$, there exists a smooth amplitude function $a$ supported on $\B_{\epsilon}$ such that the estimate \eqref{1.10} fails.

\subsection{Strongly degenerate functions}

\begin{lemma}\label{231122lemma2_1}
Let $\gamma$ be an analytic function satisfying \eqref{231122e1_3kk} and  \eqref{231122e1_4kk}. If $\gamma$ satisfies (2) in Definition \ref{231122defi1_5}, then we either have 
 \begin{align}\label{231123e2_1}
        \widetilde{\gamma}(v,\theta)=c_1v^a\theta
    \end{align}
   for some $a\ge 1, a\in \N$ and $c_1\in \R$, or
    \begin{align}\label{231123e2_2}
    \widetilde{\gamma}(v,\theta)=c_2v^a(e^{c_3\theta}-1)
    \end{align}
    for some $a\ge 1, a\in \N$ and $c_2, c_3\in \R$. Here $\widetilde{\gamma}$ is defined in \eqref{231122e1_8}. 
\end{lemma}
\begin{proof}[Proof of Lemma \ref{231122lemma2_1}]
Let us write 
\begin{equation}
\widetilde{\gamma}(v, \theta)=
v^a \gamma_2(\theta). 
\end{equation}
By a direct computation, 
\begin{equation}
\begin{split}
\cine(\widetilde{\gamma})= & 
\det
\begin{bmatrix}
v^a \gamma''_2(\theta), & v^a \gamma'''_2(\theta)\\
a v^{a-1} \gamma'_2(\theta), & a v^{a-1} \gamma''_2(\theta)
\end{bmatrix}\\
=& av^{2a-1} 
\begin{bmatrix}
\gamma''_2(\theta), & \gamma'''_2(\theta)\\
 \gamma'_2(\theta), & \gamma''_2(\theta)
\end{bmatrix}
\end{split}
\end{equation}
We therefore must have 
\begin{equation}\label{231123e2_5}
\begin{bmatrix}
\gamma''_2(\theta), & \gamma'''_2(\theta)\\
 \gamma'_2(\theta), & \gamma''_2(\theta)
\end{bmatrix}=0, \forall \theta.
\end{equation}
We write 
\begin{equation}
\gamma_2(\theta)=\sum_{k=1}^{\infty} \frac{c_k}{k!} \theta^k.
\end{equation}
By \eqref{231123e2_5}, we obtain 
\begin{equation}\label{231123e2_7}
\pnorm{
c_2+ \frac{c_3}{1!}\theta+ \frac{c_4}{2!}\theta^2+\dots
}^2= \pnorm{
c_1+ \frac{c_2}{1!}\theta+ \frac{c_3}{2!}\theta^2+\dots
}
\pnorm{
c_3+ \frac{c_4}{1!}\theta+ \frac{c_5}{2!}\theta^2+\dots
}.
\end{equation}
If $c_1=0$, then by checking the coefficients of each $\theta^k$ on both sides of \eqref{231123e2_7}, we obtain that $c_{i}=0$ for every $i\ge 2$.  This violates the assumption that the Newton diagram of $\gamma$ consists of the vertex $(a, 1)$. We therefore assume that $c_1\neq 0$. By multiplying both sides of \eqref{231123e2_7} by a non-zero constant, we can without loss of generality assume that $c_1=1$. By checking the coefficients of each $\theta^k$ on both sides of \eqref{231123e2_7}, we obtain 
\begin{equation}
c_3=c_2^2, \ c_4=c_2^3, \dots
\end{equation}
Therefore, 
\begin{equation}
\gamma_2(\theta)=
\frac{c_2^0}{1!}\theta+ \frac{c_2^1}{2!}\theta^2+ \frac{c_2^2}{3!}\theta^3+\dots
\end{equation}
If $c_2=0$, then we have the form \eqref{231123e2_1}. If $c_2\neq 0$, then we have the form \eqref{231123e2_2}. This finishes the proof of the lemma. 
\end{proof}

\begin{lemma}\label{231123lemma2_4}
Let $\gamma$ be an analytic function satisfying \eqref{231122e1_3kk} and  \eqref{231122e1_4kk}. Assume that the cinematic curvature of $\gamma$ vanishes constantly. Then $\gamma(v, \theta)$  is of one of the following forms: 
\begin{enumerate}
\item[(1)] There exist $\epsilon>0$, analytic functions $h_1, h_0: (-\epsilon, \epsilon)\to \R$ with $h_1(0)=0$ and $h_1\not\equiv 0$ such that 
\begin{equation}\label{231124e2_10uu}
\gamma(v, \theta)=
h_1(v)\theta+ h_0(v);
\end{equation}
\item[(2)] There exist $\epsilon>0$, a small interval $V_{\epsilon}\subset (-\epsilon, \epsilon)$ containing the origin, analytic functions $h_1, h_2: (-\epsilon, \epsilon)\to \R$ with $h_2(0)=h'_2(0)=0, h_2\not\equiv 0$, and an analytic function $w: V_{\epsilon}\to (-\epsilon/2, \epsilon/2)$ with $w(0)=0$ such that 
\begin{equation}\label{231124e2_11uu}
\gamma(v, \theta)=
h_1(v)+ h_2(\theta+w(v)),
\end{equation}
for all $v\in V_{\epsilon}$ and $|\theta|< \epsilon/2$. 
\item[(3)] For every $\epsilon>0$, there exist $V_{\epsilon}\subset (-\epsilon, \epsilon)$, analytic functions $h_1, h_2: (-\epsilon, \epsilon)\to \R$, and an analytic function $w: V_{\epsilon}\to (-\epsilon/2, \epsilon/2)$ satisfying 
\begin{equation}
w'(v)\neq 0, \ \forall v\in V_{\epsilon},
\end{equation}
 such that 
\begin{equation}\label{231124e2_12uu}
\gamma(v, \theta)=
h_1(v)+ h_2(\theta+w(v)),
\end{equation}
for all $v\in V_{\epsilon}$ and $|\theta|< \epsilon/2$. 
\end{enumerate}

\begin{remark}
A key difference between (2) and (3) in Lemma \ref{231123lemma2_4} is that in (3) we do not necessarily know that the interval $V_{\epsilon}$ contains the origin. As a consequence, in \eqref{231124e2_12uu}, we do not have a characterization of $\gamma(v, \theta)$ for $(v, \theta)$ in an open neighborhood of the origin. To make item (3) equally applicable, it is crucial to have the quantifier ``for every $\epsilon>0$" there. 
\end{remark}

\end{lemma}
\begin{proof}[Proof of Lemma \ref{231123lemma2_4}]
We discuss three cases separately. The first case is that $\partial_{\theta\theta}\gamma\equiv 0$, the second case is that $\partial_{v\theta}\gamma\equiv 0$, and the third case is that 
\begin{equation}\label{231123e2_19}
\partial_{\theta\theta}\gamma\not\equiv 0, \ \ \partial_{v\theta}\gamma\not\equiv 0.
\end{equation}
Let us start with the first case. In this case, we can solve $\gamma$ directly, and write it as 
\begin{equation}
\gamma(v, \theta)=
h_1(v)\theta+ h_0(v),
\end{equation}
where $h_1(v), h_0(v)$ are analytic functions in $v$ and $h_1\not\equiv 0$. This is of the form \eqref{231124e2_10uu}.\\

In the second case, we have 
\begin{equation}
\gamma(v, \theta)= h_1(\theta)+ h_2(v),
\end{equation}
which violates the assumption  \eqref{231122e1_4kk}. Therefore this case can not happen. \\

Let us consider the last case \eqref{231123e2_19}. 
\begin{claim}\label{231123claim2_5}
Assume \eqref{231123e2_19}. 
There exists an analytic function $m(v)$ such that 
\begin{equation}\label{231123e2_22}
\gamma_{\theta\theta}(v, \theta)-m(v)\gamma_{v\theta}(v, \theta)\equiv 0
\end{equation}
or 
\begin{equation}\label{231123e2_23}
\gamma_{v\theta}(v, \theta)-m(v)\gamma_{\theta\theta}(v, \theta)\equiv 0.
\end{equation}
\end{claim}
\begin{proof}[Proof of Claim \ref{231123claim2_5}]
By the assumption that $\gamma$ is analytic, we know that neither of the sets
\begin{equation}
\{v: \partial_{\theta\theta}\gamma(v, \cdot)\equiv_{\theta}0 \}, \ \ \{v: \partial_{v\theta}\gamma(v, \cdot)\equiv_{\theta}0 \}
\end{equation}
has cluster points. Therefore, if $\epsilon>0$ is chosen to be sufficiently small, then 
\begin{equation}
\begin{split}
& \#\{v\in (-\epsilon, \epsilon): \partial_{\theta\theta}\gamma(v, \cdot)\equiv_{\theta}0 \}\le 1,\\
& \#\{v\in (-\epsilon, \epsilon): \partial_{v\theta}\gamma(v, \cdot)\equiv_{\theta}0 \}\le 1.
\end{split}
\end{equation}
We consider three cases separately: The first case is when $\partial_{\theta\theta}\gamma(0, \cdot)\not\equiv_{\theta}0$; the second case is when $\partial_{v\theta}\gamma(0, \cdot)\not\equiv_{\theta}0$; the last case is when 
\begin{equation}\label{231123e2_26}
\partial_{\theta\theta}\gamma(0, \cdot)\equiv_{\theta}0, \ \ \partial_{v\theta}\gamma(0, \cdot)\equiv_{\theta}0.
\end{equation}
We start with the first case. Fix $v$. The Vandermonde determinant 
\begin{equation}
\det
\begin{bmatrix}
\gamma_{\theta\theta}, & \gamma_{v\theta}\\
\partial_{\theta}\gamma_{\theta\theta}, & \partial_{\theta}\gamma_{v\theta}
\end{bmatrix}\equiv_{\theta} 0.
\end{equation}
By Bocher's result \cite{Boc00} on linear dependence of analytic functions, we are able to find a function $m(v)$ such that 
\begin{equation}\label{231123e2_28}
\gamma_{v\theta}(v, \theta)=m(v)\gamma_{\theta\theta}(v, \theta), \ \ \forall |v|\le \epsilon, |\theta|\le \epsilon. 
\end{equation}
It remains to show that $m(v)$ is analytic. We write 
\begin{equation}
\gamma(v, \theta)=
\sum_{k=0}^{\infty} \frac{g_k(v)}{k!}\theta^k.
\end{equation}
The equation \eqref{231123e2_28} says that 
\begin{equation}\label{231123e2_30}
g'_{k+1}(v)= m(v) g_{k+2}(v), \ \ \forall k\in \N.
\end{equation}
Recall that we are in the case 
\begin{equation}
\partial_{\theta\theta}\gamma(0, \cdot)\not\equiv_{\theta}0.
\end{equation}
This guarantees that there exists $k_0\in \N$ such that $g_{k_0+2}(0)\neq 0$. We apply \eqref{231123e2_28} to $k=k_0$ and obtain 
\begin{equation}
m(v)=\frac{g'_{k_0+1}(v)}{g_{k_0+2}(v)},
\end{equation}
which is an analytic function. This finishes the proof of the first case. \\

The second case where 
\begin{equation}
\partial_{v\theta}\gamma(0, \cdot)\not\equiv_{\theta}0
\end{equation}
can be argued in exactly the same way. \\

We handle the last case \eqref{231123e2_26}. In the case, we know that 
\begin{equation}
\partial_{\theta\theta}\gamma(v, \cdot)\not\equiv 0, 
\end{equation}
whenever $|v|\le \epsilon, v\neq 0$. Therefore, we can apply Bocher's result \cite{Boc00} and obtain that there exists a function $\widetilde{m}$ such that 
\begin{equation}\label{231123e2_35}
\gamma_{v\theta}(v, \theta)=\widetilde{m}(v)\gamma_{\theta\theta}(v, \theta), \ \ \forall |v|\le \epsilon, v\neq 0,  |\theta|\le \epsilon. 
\end{equation}
Instead of \eqref{231123e2_30}, we have 
\begin{equation}\label{231123e2_36}
g'_{k+1}(v)= \widetilde{m}(v) g_{k+2}(v), \ \ \forall k\in \N, \forall v\neq 0.
\end{equation}
In the current case, 
\begin{equation}
g'_{k+1}(0)=g_{k+2}(0)=0, \forall k\in \N.
\end{equation}
However, recall $\gamma$ still satisfies \eqref{231123e2_19}, that is,
\begin{equation}
\partial_{\theta\theta}\gamma\not\equiv 0, \ \ \partial_{v\theta}\gamma\not\equiv 0.
\end{equation}
This guarantees that there exists $k_0\in \N$ such that 
\begin{equation}
g_{k_0+2}(v)\not\equiv 0.
\end{equation}
We apply \eqref{231123e2_36} to $k=k_0$ and see that either $\widetilde{m}$ or $1/\widetilde{m}$ is an analytic  function on $v\in (-\epsilon, \epsilon)$. This finishes the proof of the last case, and the proof of Claim \ref{231123claim2_5}. 
\end{proof}

We continue the discussion on the case \eqref{231123e2_19}. Claim \ref{231123claim2_5} says that there are two cases to consider. We first consider the case \eqref{231123e2_23}. We can find an analytic function $n(v)$ such that 
\begin{equation}\label{231124e2_33}
\gamma_{v}(v, \theta)-m(v) \gamma_{\theta}(v, \theta)=n(v). 
\end{equation}
Let $w: V_{\epsilon}\to \R$ be the analytic function satisfying 
\begin{equation}
w'(v)= m(v), \ \ w(0)=0.
\end{equation}
Here $V_{\epsilon}\subset (-\epsilon, \epsilon)$ is a small interval containing the origin. 
Consider the function 
\begin{equation}
\widetilde{\gamma}(v, \theta):=
\gamma(
v, \theta-w(v)
).
\end{equation}
By the chain rule, 
\begin{equation}
\partial_v \widetilde{\gamma}(v, \theta)= \partial_v \gamma(v, \theta-w(v))- m(v)
\partial_{\theta}
\gamma(v, \theta-w(v))= n(v). 
\end{equation}
Therefore, there exist analytic functions $h_1, h_2$ such that 
\begin{equation}
\gamma(v, \theta-w(v))=
h_1(v)+ h_2(\theta), \ v\in V_{\epsilon}, \theta\in (-\epsilon/2, \epsilon/2),
\end{equation}
and 
\begin{equation}
\gamma(v, \theta)=
h_1(v)+ h_2(\theta+w(v)), \ v\in V_{\epsilon}, \theta\in (-\epsilon/2, \epsilon/2).
\end{equation}
Recall that $\gamma$ satisfies \eqref{231122e1_3kk} and  \eqref{231122e1_4kk}. These imply that $h_2(0)=h'_2(0)=0$ and $h_2\not\equiv 0$, which means $\gamma$ is of the form \eqref{231124e2_11uu}. 
This finishes the discussion of the case \eqref{231123e2_23}.\\

We next consider \eqref{231123e2_22}. We can find $\epsilon>0$ and an analytic function $n(v): (-\epsilon, \epsilon)\to \R$ such that 
\begin{equation}\label{231124e2_39}
\gamma_{\theta}(v, \theta)-m(v)\gamma_{v}(v, \theta)=n(v), 
\end{equation}
for all $|v|< \epsilon, |\theta|< \epsilon$.  We without loss of generality assume that $m(0)=0$, as otherwise \eqref{231124e2_39} can be written in the form \eqref{231124e2_33}. Note that $m\not\equiv 0$, which means we are able to find $v_0\in (-\epsilon/2, \epsilon/2)$ such that 
\begin{equation}
m(v_0)\neq 0.
\end{equation}
Let $w(v): I_{\epsilon}\to (-\epsilon/2, \epsilon/2)$ be the analytic function satisfying 
\begin{equation}\label{231124e2_41}
w'(v)= m(w(v)), \ \ w(0)=v_0.
\end{equation}
Here $I_{\epsilon}\subset (-\epsilon/2, \epsilon/2)$ is a small interval containing the origin. The existence of $I_{\epsilon}$ and the existence and uniqueness of $w(v)$
is  guaranteed by  the Cauchy–Kovalevskaya Theorem. Note that 
\begin{equation}\label{231124e2_42}
w'(0)=m(w(0))=m(v_0)\neq 0,
\end{equation}
which means that $w\not\equiv 0$. \footnote{If we solve \eqref{231124e2_41} with $v_0=0$, then we will obtain $w\equiv 0$, and the function in \eqref{231124e2_43} would not be useful. }
Consider the function 
\begin{equation}\label{231124e2_43}
\widetilde{\gamma}(v, \theta):=\gamma(w(v), \theta-v), \ \ v\in I_{\epsilon}, \theta\in (-\epsilon/2, \epsilon/2). 
\end{equation}
By the chain rule, 
\begin{equation}
\partial_v \widetilde{\gamma}(v, \theta)= m(w(v))\partial_v \gamma(w(v), \theta-v)- \partial_{\theta} \gamma(w(v), \theta-v)=n(w(v)).
\end{equation}
Therefore, there exist analytic functions $h_1, h_2$ such that 
\begin{equation}
\gamma(w(v), \theta-v)=h_1(v)+h_2(\theta), \ v\in I_{\epsilon}, \theta\in (-\epsilon/2, \epsilon/2). 
\end{equation}
and 
\begin{equation}
\gamma(w(v), \theta)=h_1(v)+h_2(\theta+v), \ v\in I_{\epsilon}, \theta\in (-\epsilon/2, \epsilon/2). 
\end{equation}
Recall from \eqref{231124e2_42} that $w'(0)\neq 0$. By the implicit function theorem, we are able to find a small interval $V_{\epsilon}\subset (-\epsilon, \epsilon)$ such that 
\begin{equation}
\gamma(v, \theta)=
\widetilde{h}_1(v)+ h_2(\theta+
\widetilde{w}(v)), \ v\in V_{\epsilon}, \theta\in (-\epsilon/2, \epsilon/2),
\end{equation}
for some new analytic functions $\widetilde{h}_1$ and $\widetilde{w}$. Because of \eqref{231124e2_42}, we see immediately that
\begin{equation}
(\widetilde{w})'(v)\neq 0, \ \forall v\in V_{\epsilon},
\end{equation}
if the length of $V_{\epsilon}$ is chosen to be sufficiently small, 
 and therefore $\gamma$ is of the form \eqref{231124e2_12uu}.  This finishes the discussion on the case \eqref{231123e2_22}, and therefore the proof of Lemma \ref{231123lemma2_4}. \end{proof}

\subsection{Nikodym sets}

Recall the maximal operator 
\begin{align}
    \mathcal{M}_{\gamma}f(x,y)=\sup_{v\in \R}|\int_{\R} f(x-\theta, y-\gamma(v, \theta))a(v, \theta)d\theta| .
\end{align}
For a fixed $\gamma$, in order to show that $\mc{M}_{\gamma}$ is not bounded on $L^p(\R^2)$ for any $p<\infty$, one routine idea is to construct Nikodym sets. For instance, if $\gamma(v, \theta)=v\theta$, then the classical Nikodym set says that 
\begin{equation}
    \sup_{v\in \R}|\int_{\R} f(x-\theta, y-\gamma(v, \theta))a(v, \theta)d\theta|
\end{equation}
is unbounded on $L^p$ for every $p<\infty$. Quite surprisingly, for every single $\gamma$ that is strongly degenerate, the Nikodym sets we need have already been constructed by  Chang and Cs\"ornyei \cite{CC16}. 

\begin{theorem}[Chang and Cs\"ornyei \cite{CC16}]\label{240126theorem2_5}
    Let $E\subset \R^2$ be a strictly convex curve and let $\Gamma\subset \R^2$ be a rectifiable curve. Then there exists $A\subset \R^2$ with Lebesgue measure zero such that for every $\bfx\in \R^2$, we can find $p_{\bfx}\in \R^2$ and 
    \begin{equation}\label{240126e2_53pp}
        E_{\bfx}\subset E, \ \ \mc{H}^1(E\setminus  E_{\bfx})=0,
    \end{equation}
    such that 
    \begin{equation}
        \bfx\in p_{\bfx}+ \Gamma, \ \ p_{\bfx}+E_{\bfx}\subset A. 
    \end{equation}
    Here $\mc{H}^1$ refers to the one-dimensional Hausdorff measure. 
\end{theorem}

Chang and Cs\"ornyei \cite{CC16} indeed proved a stronger version of Theorem \ref{240126theorem2_5}, see Theorem 6.10 in their paper. For instance, they worked with $E$ being rectifiable. However, for such a general $E$, they need an extra assumption that for every direction $\theta\in S^1$, the set 
\begin{equation}\label{240126e2_55z}
    \{\bfx\in E: \theta_{\bfx}=\theta\}
\end{equation}
is $\mc{H}^1$-null. Here roughly speaking $\theta_{\bfx}$ can be understood as the tangent direction of $E$ at $\bfx$. By assuming that $E$ is strictly convex, we automatically have \eqref{240126e2_55z}.

Note that in Theorem \ref{240126theorem2_5}, the sets $E$ and $\Gamma$ do not need to be related, and this makes Theorem \ref{240126theorem2_5} an extremely general theorem. It turns out that this level of generality is precisely what we need in our setting. \\

In the rest of this subsection, we will use the classical Nikodym set and the Nikodym sets constructed by Chang and Cs\"ornyei to show that $\mc{M}_{\gamma}$ is unbounded on $L^p(\R^2)$ for every $p<\infty$, whenever $\gamma$ is strongly degenerate.

\begin{lemma}\label{231123lemma2_2}
Assume that 
\begin{equation}
\gamma(v, \theta)= v^a \theta+h_1(v)+ O(|v|^{a+1}|\theta|),
\end{equation}
for some $a\in \N, a\ge 1$. Then for every smooth function $a(v, \theta)$ with $a(0, 0)\neq 0$, the maximal operator $\mc{M}_{\gamma}$ fails to be bounded on $L^p(\R^2)$ for every $p<\infty$. 
\end{lemma}
\begin{proof}[Proof of Lemma \ref{231123lemma2_2}]
We consider two cases 
\begin{equation}
    h_1(v)=O(v^a), \text{ or } h_1(v)\neq O(v^a).
\end{equation}
The latter case is trivial, as can be seen via a direct scaling argument in the $v$ variable. We focus on the former case. 
We argue by contradiction. Let us assume that there exists $p<\infty$, such that 
\begin{equation}\label{231123e2_12}
\Norm{
\mc{M}_{\gamma} f
}_{L^p(\R^2)}\lesim_{\gamma, a} \norm{f}_{L^p(\R^2)}. 
\end{equation}
Let $j_1\in \N$ be a large integer. \eqref{231123e2_12} trivially implies that 
\begin{equation}
\Norm{
\sup_{0\le v\le 2^{-j_1}}\anorm{\int_{\R} f(x-\theta, y-\gamma(v, \theta))a(v, \theta)d\theta}
}_{L^p(\R^2)}\lesim_{\gamma, a} \norm{f}_{L^p(\R^2)}. 
\end{equation}
We rename the parameter $v\mapsto 2^{-j_1}v$,  scale the $y$ variable, and obtain 
\begin{multline}
\Norm{
\sup_{0\le v\le 1}\anorm{\int_{\R} f(x-\theta, y-
v^a\theta+
2^{aj_1}h_1(2^{-j_1}v)
+ 2^{-j_1} O(|v|^{a+1}|\theta|)
)a(v, \theta)d\theta}
}_{L^p(\R^2)}\\
\lesim_{\gamma, a} \norm{f}_{L^p(\R^2)},
\end{multline}
uniformly in $j_1\in \N$. 
However, by taking a limit in $j_1$, this is a contradiction to the existence of the (classical) Nikodym sets.
 The proof of Lemma \ref{231123lemma2_2} is finished. \end{proof}

\begin{lemma}\label{231123lemma2_3}
Assume that 
\begin{equation}
\gamma(v, \theta)= v^a(e^{\theta}-1)+ h_1(v)+O(|v|^{a+1}|\theta|),
\end{equation}
for some $a\in \N, a\ge 1$. Then for every smooth function $a(v, \theta)$ with $a(0, 0)\neq 0$, the maximal operator $\mc{M}_{\gamma}$ fails to be bounded on $L^p(\R^2)$ for every $p<\infty$. 
\end{lemma}
\begin{proof}[Proof of Lemma \ref{231123lemma2_3}] Without loss of generality, let us assume that $a$ is non-negative.  We argue by contradiction, and assume that there exists $p<\infty$, such that 
\begin{equation}\label{231123e2_12zz}
\Norm{
\mc{M}_{\gamma} f
}_{L^p(\R^2)}\lesim_{\gamma, a} \norm{f}_{L^p(\R^2)}. 
\end{equation}
By the same scaling argument as in the proof of Lemma \ref{231123lemma2_2}, we obtain that 
\begin{multline}
\Norm{
\sup_{0\le v\le 1}\anorm{\int_{\R} f(x+\theta, y+
v(e^{\theta}-1))a(\theta)d\theta}
}_{L^p(\R^2)}
\lesim_{\gamma, a} \norm{f}_{L^p(\R^2)},
\end{multline}
where $a: \R\to \R$ is a non-negative smooth function supported in a small neighborhood of the origin, with $a(0)\neq 0$. Let $w_1(v)$ be a $C^1$ function. We show that the maximal operator 
\begin{equation}\label{240126e2_64pp}
    \sup_{0\le v\le 1}
    \anorm{
    \int_{-1}^1
    f(x+\theta, 
    y+ ve^{\theta}+w_1(v)
    ) d\theta
    }
\end{equation}
is not bounded on $L^p(\R^2)$ for every $p<\infty$. To apply Theorem \ref{240126theorem2_5}, we take 
\begin{equation}
    E:=\{
    (\theta, e^{\theta}-1):
    |\theta|\le 1
    \},
\end{equation}
and 
\begin{equation}
    \Gamma:=
    \{
    (\ln v, -w_1(v)-1): 1/e\le v\le 1
    \},
\end{equation}
where $e$ is the base of the natural log.  For every $\bfx=(x, y)\in \R^2$, we are able to find $p_{\bfx}$ such that 
\begin{equation}
    \bfx\in p_{\bfx}+\Gamma.
\end{equation}
In other words, there exists $v_{\bfx}\in [1/e, 1]$ such that 
\begin{equation}
    p_{\bfx}= 
    (
    x-\ln v_{\bfx}, 
    y+ w_1(v_{\bfx})+1
    ).
\end{equation}
Moreover, 
\begin{equation}
\begin{split}
        & p_{\bfx}+ (\theta, e^{\theta}-1)\\
        & =
        (
    x-\ln v_{\bfx}+\theta, 
    y+e^{\theta}+w_1(v_{\bfx})
    )
        \in A,
\end{split}
\end{equation}
for almost every $\theta\in [-1, 1]$. Note that $v_{\bfx}\in [1/e, 1]$, and therefore 
\begin{equation}
    (x+\theta', 
    y+ 
    e^{\theta'+\ln v_{\bfx}}+w_1(v_{\bfx})
    )\in A,
\end{equation}
for almost every $\theta'\in [0, 1]$. By taking $f$ to be the indicator function of $A$, we see that  the maximal operator \eqref{240126e2_64pp} fails to be bounded on $L^p(\R^2)$ for every $p<\infty$. 
\end{proof}

So far we have finished the two types of functions given by Lemma \ref{231122lemma2_1}. Next, we will handle the three types of functions given by Lemma \ref{231123lemma2_4}. 

The first type, given by item (1) in Lemma \ref{231123lemma2_4} can be handled in exactly the same way as Lemma \ref{231123lemma2_2}. 

The second and third types, given by item (2) and item (3) in Lemma \ref{231123lemma2_4}, will require Theorem \ref{240126theorem2_5}. Let us take the second type as an example; the other one is exactly the same. Recall that there exists $\epsilon>0$ such that 
\begin{equation}
    \gamma(v, \theta)=
    h_2(\theta+w(v))+h_1(v),
\end{equation}
where $w: V_{\epsilon}\to (-\epsilon/2, \epsilon/2)$ is an analytic function with $w(0)=0$, $V_{\epsilon}\subset (-\epsilon, \epsilon)$ is an interval containing the origin, $h_1, h_2$ are analytic functions on $(-\epsilon, \epsilon)$, and
\begin{equation}
    h_2(0)=h'_2(0)=0, \ h_2\not\equiv 0.
\end{equation}
By applying Theorem \ref{240126theorem2_5} to 
\begin{equation}
    \Gamma:=  
    \{
    (w(v), -h_1(v)): v\in V_{\epsilon}
    \},
\end{equation}
and 
\begin{equation}
    E:=
    \{
    (\theta, h_2(\theta)): |\theta|< \epsilon/2
    \},
\end{equation}
we finish the discussion for the second type of functions.

\section{Examples}\label{240222section3}

In this section, we will first give examples for the algorithm of determining the critical exponent $p_{\gamma}$, and then give examples for the sharpness of $p_{\gamma}$.

\subsection{Examples for the algorithm of determining the critical exponent \texorpdfstring{$p_{\gamma}$}{}
}

The first example is 
\begin{equation}
   \gamma(v, \theta)=
   (\theta-v)^3 v+v^3.
\end{equation}
The reduced Newton diagram $\large \mathcal{R}\mathcal{N}_d(\gamma)$ consists of two points $(1, 3), (3, 1)$ and one edge connecting the points $(1, 3), (3, 1)$, called $\mf{E}_1$. The vertical Newton distance in this case is 
\begin{equation}
   \mf{d}(\gamma)=1.
\end{equation}
Next, we work on the edge $\mf{E}_1$. Denote 
\begin{equation}
   \gamma_{
   \mf{E}_1
   }(v, \theta):= 
   (\theta-v)^3 v+v^4.
\end{equation}
We take the second order derivative in $\theta$, and obtain 
\begin{equation}
    \kappa_{
    \mf{E}_1
    }(v, \theta)=
    6v(\theta-v).
\end{equation}
Write 
\begin{equation}
   \theta=r v,
\end{equation}
and 
\begin{equation}
   \kappa_{\mf{E}_1}(v, \theta)=
   \kappa_{\mf{E}_1}(v, r v)=
   6v^2(r-1)=: v^2 \kappa_{\mf{E}_1}(r). 
\end{equation}
Note that 
\begin{equation}
   \kappa_{\mf{E}_1}(r)=
   6(r-1),
\end{equation}
and it has only one non-zero real root. Denote $\mf{r}_{1, 1}:= 1$. We do the change of variables 
\begin{equation}
   v\to v, \ \ \theta\to \theta+ \mf{r}_{1, 1}v, 
\end{equation}
and obtain 
\begin{equation}
   \gamma_{\star}(v, \theta):=\gamma(v, \theta+v)=
   \theta^3 v+ v^4.
\end{equation}
The reduced Newton diagram $\mc{R}\mc{N}_d(\gamma_{\star})$ consists of only one vertex $(1, 3)$. Note that 
\begin{equation}
    \mf{d}_{
    >(1, 3)
    }(\gamma_{\star})=3,
\end{equation}
where we applied the convention made in \eqref{240125e1_19}. Moreover, we run out of edges to consider, and therefore the algorithm terminates. We define 
\begin{equation}
   \mf{D}_{\gamma}:=\max\{
   \mf{d}(\gamma), \mf{d}_{>(1, 3)}(\gamma_{\star})
   \}=3. 
\end{equation}
Moreover, define the critical exponent
\begin{equation}
   p_{\gamma}:=\max\{2, 
   \mf{D}_{\gamma}
   \}=3.
\end{equation}
This finishes the discussion of the first example. \\

The second example is 
\begin{equation}
    \gamma(v, \theta)=
    (
    \theta+\frac{v}{1!}+\frac{v^2}{2!}+\dots
    )^3 v+ v^2.
\end{equation}
The reduced Newton diagram consists of two vertices $(1, 3), (3, 1)$ and the edge $\mf{E}_1$ connecting them. The vertical Newton distance 
\begin{equation}
    \mf{d}(\gamma)=2. 
\end{equation}
Let us work on $\mf{E}_1$. Denote 
\begin{equation}
    \gamma_{\mf{E}_1}(v, \theta)=
    (\theta+v)^3v-v^4.
\end{equation}
We take the second derivative in $\theta$ and obtain 
\begin{equation}
    \kappa_{\mf{E}_1}(v, \theta)=
    6(\theta+v)v.
\end{equation}
Note that 
\begin{equation}
    \kappa_{\mf{E}_1}(r)=
    6(r+1),
\end{equation}
which admits only one non-zero root $\mf{r}_{1, 1}=-1$. We apply the change of variable 
\begin{equation}
    \theta\to \theta+\mf{r}_{1, 1}v,
\end{equation}
and obtain 
\begin{equation}
    \gamma_{\star}(v, \theta)=
    \pnorm{
    \theta+\frac{v^2}{2!}+\frac{v^3}{3!}+\dots
    }^3 v+ v^2.
\end{equation}
Note that 
\begin{equation}
    \mf{d}_{
    >(1, 3)
    }(\gamma_{\star})=5/2.
\end{equation}
In the reduced Newton diagram $\mc{R}\mc{N}_d(\gamma_{\star})$, we consider all the vertices $(\mf{p}, \mf{q})$ with $\mf{p}>1$. There is only one such vertex $(5, 1)$, which we call $(\mf{p}_{\star+1}, \mf{q}_{\star+1})$. The vertex $(\mf{p}_{\star+1}, \mf{q}_{\star+1})$ does not lie on the line containing $\mf{E}_1$, and therefore we let $\mf{E}_{\star+0}$ denote the edge in the $\mc{R}\mc{N}_d(\gamma_{\star})$ connecting the vertices 
\begin{equation}
    (1, 3),  \ (5, 1).
\end{equation}
Note that the slope of $\mf{E}_{\star+0}$ is strictly smaller than that of $\mf{E}_1$. Denote 
\begin{equation}\label{240222e3_22}
    \gamma_{\star+0}(v, \theta)=
    \pnorm{
    \theta+\frac{v^2}{2!}
    }^3 v- 
    \pnorm{
    \frac{v^2}{2!}
    }^3 v.
\end{equation}
The reduced Newton diagram $\mc{R}\mc{N}_d(\gamma_{\star})$ consists of a new edge $\mf{E}_{\star+0}$ we need to consider next. We repeat the above process. The changes of variables we will do in the rest of the algorithm will be 
\begin{equation}\label{240222e3_23}
    \theta\to \theta-\frac{v^n}{n!}, \  \ n=2, 3, \dots
\end{equation}
and the counterparts of \eqref{240222e3_22} are 
\begin{equation}
    \pnorm{
    \theta+\frac{v^n}{n!}
    }^3 v- 
    \pnorm{
    \frac{v^n}{n!}
    }^3 v, \ n=3, 4, \dots
\end{equation}
Note that they are all of the form \eqref{240129e1_42kkk}. Instead of doing the changes of variables \eqref{240222e3_23} separately, we apply 
\begin{equation}
    \theta\to \theta-\frac{v}{1!}-\frac{v^2}{2!}-\dots,
\end{equation}
and obtain 
\begin{equation}
    \gamma_{\star}(v, \theta)=\theta^3 v+v^2,
\end{equation}
and the algorithm terminates immediately. Note that 
\begin{equation}
    \mf{d}_{
    >(1, 3)
    }(\gamma_{\star})=3,
\end{equation}
where we applied the convention made in \eqref{240125e1_19}. In the end, put 
\begin{equation}
   \mf{D}_{\gamma}:=\max\{
   \mf{d}(\gamma), \mf{d}_{>(1, 3)}(\gamma_{\star})
   \}=3. 
\end{equation}
Moreover, define the critical exponent
\begin{equation}
   p_{\gamma}:=\max\{2, 
   \mf{D}_{\gamma}
   \}=3.
\end{equation}
This finishes the discussion of the second example.

\subsection{Examples for the sharpness of \texorpdfstring{$p_{\gamma}$}{}}

Our goal in this subsection is to show that the exponent $p_{\gamma}$ in Theorem \ref{240127theorem1_6} is sharp, in the sense that $\mc{M}_{\gamma}$ is unbounded on $L^p(\R^2)$ for every $p< p_{\gamma}$. \\

Let us start by using the standard Knapp example to show that the constraint $p\ge 2$ is always necessary. Assume that 
\begin{equation}\label{240129e3_1}
    \sup_{
    |v|\le \epsilon
    }
    \anorm{
    \int_{|\theta|\le \epsilon}
    f(x-\theta, 
    y-
    \gamma(v, \theta)
    )d\theta
    }
\end{equation}
is bounded on $L^p(\R^2)$. The only ``enemy" here is the case where $\gamma(v, \theta)$ depends purely on $\theta$, as in this case \eqref{240129e3_1} is bounded on $L^p(\R^2)$ for every $p>1$. 

Let $\delta>0$.  Take $f$ to be the indicator function of the ball $\B_{\delta}$. Our goal is to show that 
\begin{equation}\label{240129e3_2}
    \eqref{240129e3_1}\gtrsim \delta,
\end{equation}
for $(x, y)$ from a set of Lebesgue measure $\simeq_{\gamma} 1$, and the constrain $p\ge 2$ follows immediately. To show \eqref{240129e3_2}, it suffices to show that the Jacobian of the map 
\begin{equation}\label{240129e3_2zz}
    (v, \theta)\mapsto (\theta, \gamma(v, \theta))
\end{equation}
does not vanish constantly. Simple calculation shows that the Jacobian is equal to $\partial_v \gamma$, which  vanishes constantly if and only if $\gamma$ depends purely on $\theta$.  This finishes showing that the constraint $p\ge 2$ is necessary. \\

We next show that the constraint 
\begin{equation}\label{240129e3_4}
    p\ge \mf{d}(\gamma)
\end{equation}
is necessary in order for $\mc{M}_{\gamma}$ to be bounded on $L^p$. Without loss of generality, we assume that $\mf{d}(\gamma)>2$.  Denote 
\begin{equation}\label{240128e3_2pp}
    h_1(v):=\sum_{\mf{p}} c_{\mf{p}, 0} v^{\mf{p}}.
\end{equation}
Denote 
\begin{equation}
    \mf{p}_0:=\inf \{\mf{p}: c_{\mf{p}, 0}\neq 0\}.
\end{equation}
If the inf is taken over an empty set, then we write $\mf{p}_0=\infty$. By the assumption that $\mf{d}(\gamma)>2$, we know that $\mf{p}_0< \infty$. We draw the vertical line passing through the point $(\mf{p}_0, 0)$, and by the definition of $\mf{d}(\gamma)$, this line enters the  reduced Newton polyhedron $\mc{R}\mc{N}(\gamma)$ at $(\mf{p}_0, \mf{d}(\gamma))$. (This line may not intersect the reduced Newton diagram $\mc{R}\mc{N}_d(\gamma)$.) Let $\mf{E}_i$ be an (possibly non-compact) edge in $\mc{R}\mc{N}(\gamma)$ that contains the point  $(\mf{p}_0, \mf{d}(\gamma))$. \footnote{There may be two such edges, we just take an arbitrary one. } Write the two endpoints of the edge $\mf{E}_i$ as 
\begin{equation}
    (\mf{p}_i, \mf{q}_i),  \ \ (\mf{p}_{i+1}, \mf{q}_{i+1}),
\end{equation}
with $\mf{p}_i< \mf{p}_{i+1}$. If $\mf{E}_i$ is non-compact, then $\mf{p}_{i+1}=\infty$. For the sake of simplicity, we only consider the case 
\begin{equation}
    \mf{p}_i < \mf{p}_0< \mf{p}_{i+1};
\end{equation}
the other cases can be handled similarly. Denote 
\begin{equation}
    \mf{m}_i:=
    \frac{
    \mf{p}_i-\mf{p}_{i+1}
    }{
    \mf{q}_{i+1}-\mf{q}_{i}
    }.
\end{equation}
Let $\zeta>0$ be a small positive real number; $\zeta$ will be sent to zero eventually. Take two large positive integers $j_1, j_2$ satisfying 
\begin{equation}
    j_1= (\mf{m}_i+\zeta) j_2.
\end{equation}
We will send $j_1, j_2$ to $\infty$, for fixed $\zeta$. Assume that 
\begin{equation}
    \sup_{
    |v|\le \epsilon
    }
    \anorm{
    \int_{|\theta|\le \epsilon}
    f(x-\theta, 
    y-
    \gamma(v, \theta)
    )d\theta
    }
\end{equation}
is bounded on $L^p(\R^2)$. We restrict the values of $v$ and $\theta$, and trivially obtain that 
\begin{equation}
    \Norm{
    \sup_{
    v\in [2^{-j_1}, 2^{-j_1+1}]
    }
    \anorm{
    \int_{2^{-j_2}}^{
    2^{-j_2+1}
    }
    f(x-\theta, 
    y-
    \gamma(v, \theta)
    )d\theta
    }
    }_{
    L^p(\R^2)
    }\lesim \norm{f}_{L^p(\R^2)},
\end{equation}
uniformly in $j_1, j_2$. By scaling, this implies 
\begin{equation}
    \Norm{
    \sup_{
    v\in [1, 2]
    }
    \anorm{
    \int_{1}^{
    2
    }
    f(x-\theta, 
    y-
    \gamma_{\bfj}(v, \theta)
    )d\theta
    }
    }_{
    L^p(\R^2)
    }\lesim 
    2^{j_2} \norm{f}_{L^p(\R^2)},
\end{equation}
where 
\begin{equation}
    \gamma_{\bfj}(v, \theta):=
    2^{
    \mf{p}_i j_1+\mf{q}_i j_2
    }
    \gamma(2^{-j_1}v, 2^{-j_2}\theta), \ \ \bfj:=(j_1, j_2). 
\end{equation}
By taking $f$ to be the indicator function of the unit ball, we obtain
\begin{equation}
    2^{
    \frac{
    \mf{p}_i j_1+\mf{q}_i j_2-\mf{p}_0 j_1
    }{
    p
    }
    }\lesim 
    2^{
    j_2
    },
\end{equation}
and therefore 
\begin{equation}
    p\ge 
    \mf{q}_i+
    (\mf{p}_i-\mf{p}_0) \frac{j_1}{j_2}\ge \mf{q}_i+
    (\mf{p}_i-\mf{p}_0) (\mf{m}_i+\delta),
\end{equation}
for every $\delta>0$. By sending $\delta$ to zero, we obtain 
\begin{equation}\label{240129e3_17}
    p\ge \mf{q}_i+
    (\mf{p}_i-\mf{p}_0) \mf{m}_i=\mf{d}(\gamma). 
\end{equation}
This finishes showing that the constraint $p\ge \mf{d}(\gamma)$ is necessary. \\

At this point, it should be clear that by following the algorithm in Subsection \ref{240129subsection1_2}, and repeating the argument in \eqref{240129e3_4}-\eqref{240129e3_17}, we will see that the constraint 
\begin{equation}
    p\ge \max\{2, \mf{D}_{\gamma}\}=p_{\gamma}
\end{equation}
is necessary.

\section{Resolution of singularities}\label{240129section4kk}

The material in this section is a preparation for the proof of the ``only if" part of Theorem \ref{240127theorem1_2} and the proof of Theorem \ref{240127theorem1_6}, and can be  stated independently from the rest of the (already long) proof.  \\

 For $\epsilon>0$, let $\B_{\epsilon}\subset \R^2$ be the ball of radius $\epsilon$ centered at the origin. Let $\B_{\epsilon}^+$ be the intersection of $\B_{\epsilon}$ with the first quadrant. Let $h(v): (-\epsilon, \epsilon)\to \R$ be an analytic function that is not constantly zero. Define 
 \begin{equation}
 \order(h):=\min \{
 \alpha\in \N: h^{(\alpha)}(0)\neq 0
 \}.
 \end{equation}
 Let $M\in \N$ be a positive integer. Define 
 \begin{equation}
 \order(h(v^{\frac{1}{M}})):= 
 \frac{
 \order(h)
 }{
 M}.
 \end{equation}

\begin{theorem}[Greenblatt, \cite{Gre04}]\label{231120theorem3_3}
    Let $P(v, \theta): \R^2\to \R$ be a real analytic function that is not constantly zero. There exist a small constant $\epsilon>0$, large positive integers $L, N\in \N$ and $K_1, \dots, K_L\in \N$, a positive integer $M\in \N$, a collection of functions 
    \begin{equation}\label{231122e3_2zz}
    \begin{split}
    & r_{k_1}(v),r_{k_1,k_2}(v),\cdots,r_{k_1,k_2,\cdots,k_L}(v), s_{k_1}(v),s_{k_1,k_2}(v),\cdots,s_{k_1,k_2,\cdots,k_L}(v),\\
    & t_{k_1}(v),t_{k_1,k_2}(v),\cdots,t_{k_1,k_2,\cdots,k_L}(v), \ 1\le k_l\le K_l, \ 1\le l\le L,
    \end{split}
    \end{equation}
each of which     is an analytic function of $v^{\frac{1}{M}}$ for $0\le v< \epsilon$, such that 
\begin{enumerate}
\item[(1)] for all $1\le l\le L$ and all  $k_1, \dots, k_L$, it holds that 
\begin{align}
& \order(r_{k_1})<\order(r_{k_1,k_2})<\cdots<\order(r_{k_1,\cdots,k_l})\\
 & \le \min\{\order(s_{k_1,\cdots,k_l}),\order(t_{k_1,\cdots,k_l})\};
\end{align}
\item[(2)] the set $\B_{\epsilon}^+$ can be written as a disjoin union 
\begin{equation}
\B_{\epsilon}^+=
U_0\bigcup 
 \pnorm{
\bigcup_{n=1}^N U_n}
\bigcup U_{N+1},
\end{equation}
where 
\begin{equation}\label{231122e3_6zz}
\begin{split}
& U_0= \{(v,\theta)\in \B^+_{\epsilon}: 0\le \theta<  r_{1}(v)\}, \\
& U_{N+1}=
\{(v,\theta)\in \B^+_{\epsilon}:\theta\ge r_{K_1}(v)\},
\end{split}
\end{equation}
and each $U_n, 1\le n\le N$ can be written in the form 
 \begin{equation}\label{231122e3_7zz}
 \begin{split}
        U_n=\Big\{(v,\theta)\in \B^+_{\epsilon}: & \sum_{l'=1}^{l}r_{k_1,k_2,\cdots,k_{l'}}(v)+s_{k_1,k_2,\cdots,k_l}(v)\le 
        \theta \\
        & 
        <\sum_{l'=1}^{l}r_{k_1,k_2,\cdots,k_{l'}}(v)+t_{k_1,k_2,\cdots,k_{l'}}(v)\Big\}
        \end{split}
    \end{equation}
for some $1\le l\le L$ and some $k_1, \dots, k_l$; 
\item[(3)] on $U_0$ and $U_{N+1}$, it holds that 
\begin{equation}
|P(v, \theta)|\simeq |v|^{a_0}|\theta|^{b_0}, \ \ |P(v, \theta)|\simeq |v|^{a_{N+1}}|\theta|^{b_{N+1}},
\end{equation}
respectively, for some $a_0, b_0, a_{N+1}, b_{N+1}\ge 0$;  on the set $U_n$ given by \eqref{231122e3_7zz}, it holds that 
\begin{align}\label{231122e3_9zz}
    |P(v,\theta)|\simeq |v|^{a_n}\anorm{
    \theta-\sum_{l'=1}^{l}r_{k_1,\cdots,k_{l'}}(v)}^{b_n},
\end{align}  
for some $a_n\ge 0, b_n\ge 0$, where the implicit constant depends only on the analytic function $P$.  
\end{enumerate}
\end{theorem}

Theorem \ref{231120theorem3_3} is a simplified version of \cite[Theorem 2.1]{Gre04}. 
The statement of Theorem \ref{231120theorem3_3} is long and involves complicated notation. After seeing the proof, we find it much easier to accept the functions in \eqref{231122e3_2zz} and their properties. For this reason, we decide to include the part of the proof of Theorem \ref{231120theorem3_3} in \cite{Gre04} that is related to the construction of the functions in \eqref{231122e3_2zz}.

\subsection{A standard reduction}\label{231121subsection3_1}

If the Newton diagram $\mc{N}_d(P)$ contains only one vertex, then the theorem is immediate. We therefore assume that the Newton diagram $\mc{N}_d(P)$ contains more than one vertices. 

We label all vertices of $\mc{N}_d(P)$ by $(p_1, q_1), (p_2, q_2), \dots, (p_{I}, q_{I})$ with 
\begin{equation}
    p_1< p_2< \dots, \ \ q_1>q_2>\dots
\end{equation}
Denote 
\begin{equation}
    m_{\iota}:=
    \frac{p_{\iota+1}-p_{\iota}}{q_{\iota}-q_{\iota+1}}.
\end{equation}
We have 
\begin{equation}
    m_1< m_2< \dots
\end{equation}
Consider the regions 
\begin{equation}\label{231121e3_8}
    \begin{split}
        & \bomega_1:=U_+\cap \{
    (v, \theta): \theta\ge C_P v^{m_1}
    \}, \\
    & 
    \bomega_{\iota}:= U_+\cap \{
    (v, \theta): 
    (C_P)^{-1} v^{m_{\iota-1}}\ge \theta
    \ge 
    C_P v^{
    m_{\iota}
    }
    \}, 1< \iota\le I, \\
    & \bomega_{I+1}:=U_+\cap 
    \{
    (v, \theta): \theta\le (C_P)^{-1} v^{m_I}
    \}.
    \end{split}
\end{equation}
Here and below, we use $C_P$ to denote a large constant that is allowed to depend on $P$. Its precise value is not important, and may vary from line to line. For each $1\le \iota\le I+1$, that $C_P$ is large enough guarantees that the monomial $v^{p_{\iota}}\theta^{q_{\iota}}$ dominates  on $\bomega_{\iota}$. More precisely, we have 
\begin{equation}
    |P(v, \theta)|\simeq 
v^{p_{\iota}}\theta^{q_{\iota}},
\end{equation}
for $(\theta, v)\in \bomega_{\iota}$. This shows that $\bomega_{\iota}$ is of the form \eqref{231122e3_6zz} or \eqref{231122e3_7zz} for every $\iota$.  Therefore, to prove the theorem, we only need to consider regions outside the union of $\bomega_{\iota}$. \\

For $1\le \iota\le I$, denote 
\begin{equation}
    \bomega'_{\iota}:=
    U_+\cap 
    \{
    (v, \theta): 
    (C_P)^{-1}v^{m_{\iota}}\le \theta\le C_P v^{m_{\iota}}
    \},
\end{equation}
where $C_P$ is the same as that in \eqref{231121e3_8} and we therefore have 
\begin{equation}\label{231121e3_11}
    \pnorm{\bigcup_{1\le \iota\le I+1} \bomega_{\iota}}
    \bigcup 
    \pnorm{\bigcup_{1\le \iota\le I} \bomega'_{\iota}}
\end{equation}
is the first quadrant. We denote by $E_{\iota}$ the edge of the Newton diagram that connects the vertices $(p_{\iota}, q_{\iota})$ and $(p_{\iota+1}, q_{\iota+1})$. To make it easier to remember notation, we denote 
\begin{equation}
    (p_{E_{\iota}, l}, q_{E_{\iota}, l}):=(p_{\iota}, q_{\iota}), \ \ (p_{E_{\iota}, r}, q_{E_{\iota}, r}):=(p_{\iota+1}, q_{\iota+1}).
\end{equation}
Denote 
\begin{equation}
    c_{p, q}:= \frac{1}{p ! q !} \partial_{v}^p \partial_{\theta}^q P(0,0).
\end{equation}
Moreover, denote
\begin{equation}
P_{E_{\iota}}(v,\theta)=\sum_{(p,q)\in E_{\iota}}c_{p,q}v^p\theta^q.
\end{equation}
Note that for $(v, \theta)\in \bomega'_{\iota}$, it holds that 
\begin{align}
    |v^{p_{E_{\iota},l}}\theta^{q_{E_{\iota},l}}|\simeq |v^{p_{E_{\iota},r}}\theta^{q_{E_{\iota},r}}|,
\end{align}
that is, 
\begin{align}
    |\theta|\simeq |v|^{m_{\iota}}.
\end{align}
Consider the curve 
$\theta=rv^{m_{\iota}}$, where $r\in [(C_P)^{-1}, C_P]$. On this curve, 
\begin{align}\label{231121e3_17}
    P_{E_{\iota}}(v, \theta)
    =P_{E_{\iota}}(v,rv^{m_{\iota}})=v^{e_{\iota}}\sum_{(p,q)\in E_{\iota}}c_{p,q} r^q:=v^{e_{\iota}}P_{E_{\iota}}(r), 
\end{align}
where
\begin{equation}
e_{\iota}=p+qm_{\iota}
\end{equation}
 and $(p, q)$ can be taken to be an arbitrary point in $E_{\iota}$.

The simple calculation in \eqref{231121e3_17} suggests that the zeros of the single variable polynomial $P_{E_{\iota}}(r)$ plays particular roles.  Let
\begin{equation}
\{r_{\iota, j}\}_{j=1}^{J_{\iota}}\subset \R
\end{equation}
 be the collection of distinct zeros of $P_{E_{\iota}}(r)$. For each $\iota$ and $j$, let $s_{\iota, j}$ be the multiplicity of $r_{\iota, j}$.
 Denote 
 \begin{equation}\label{231122e3_26pp}
 S:= \max_{\iota, j} s_{\iota, j}. 
 \end{equation}
  Recall that $r\in [(C_P)^{-1}, C_P]$. Denote 
\begin{equation}
    R_{\iota}:=[(C_P)^{-1}, C_P].
\end{equation}
Moreover, 
\begin{equation}
    R_{\iota, \mathrm{bad}}:= 
    R_{\iota} \bigcap \pnorm{
    \bigcup_{j=1}^{J_{\iota}}
    (r_{\iota, j}-(C_P)^{-2}, r_{\iota, j}+ (C_P)^{-2})
    }, \ \ R_{\iota, \mathrm{good}}:=R_{\iota}\setminus R_{\iota, \mathrm{bad}}. 
\end{equation}
We decompose $\bomega'_{\iota}$ into two parts: 
\begin{equation}\label{231121e3_20}
\begin{split}
    & \bomega'_{\iota, \mathrm{good}}:= 
    \{
    (v, \theta): \theta=rv^{m_{\iota}}, r\in R_{\iota, \mathrm{good}}
    \}, \\
    &  \bomega'_{\iota, \mathrm{bad}}:= 
    \{
    (v, \theta): \theta=rv^{m_{\iota}}, r\in R_{\iota, \mathrm{bad}}
    \}.
    \end{split}
\end{equation}
On $\bomega'_{\iota, \mathrm{good}}$, it is elementary to see that 
\begin{equation}
    |P(v, \theta)|\simeq v^p \theta^q,
\end{equation}
for every $(p, q)\in E_{\iota}$, and therefore $\bomega'_{\iota, \mathrm{good}}$ can be written in the form \eqref{231122e3_7zz} and satisfies \eqref{231122e3_9zz}. 

This finishes the reduction, and from now on we will only focus on the regions $\bomega'_{\iota, \mathrm{bad}}$.

\subsection{Structures of bad regions}\label{231121subsection3_2}

In this subsection, we will focus on $(v, \theta)$ in the bad regions $\bomega'_{\iota, \mathrm{bad}}$. Denote 
\begin{equation}
    \bomega'_{\iota, j}:=
    \{
    (v, \theta): 
    \theta= rv^{m_{\iota}}, r\in R_{\iota}, |r-r_{\iota, j}|\le (C_P)^{-2}
    \},
\end{equation}
and we have 
\begin{equation}\label{231121e3_23}
    \bomega'_{\iota, \mathrm{bad}}=\bigcup_{j=1}^{J_{\iota}} \bomega'_{\iota, j}. 
\end{equation}
Recall the polynomial $P_{E_{\iota}}(r)$, its roots $r_{\iota, j}$ and multiplicity $s_{\iota, j}$. 

\begin{lemma}[\cite{Gre04}, Lemma 2.3]\label{231121lemma3_2}
    For every $1\le \iota\le I$ and every $1\le j\le J_{\iota}$, it holds that 
    \begin{equation}\label{231121e3_24}
        s_{\iota, j}\le q_{E_{\iota}, l}- q_{E_{\iota}, r},
    \end{equation}
    and\footnote{This is a very crude bound, and it is stated here to say that the exponent on the right hand side of \eqref{231121e3_26} is non-negative. }
    \begin{equation}\label{231121e3_25}
        s_{\iota, j} m_{\iota}\le e_{\iota}. 
    \end{equation}
    On the bad region $\bomega'_{\iota, j}$, we have
    \begin{align}\label{231121e3_26}
        |\partial_{\theta}^{s_{\iota, j}}P(v,\theta)|\simeq v^{e_{\iota}-s_{\iota, j}m_{\iota}}.
    \end{align}
    Furthermore, $s_{\iota, j}$ is the minimal value of $s$ for which an identity of the form
    \begin{align}\label{231121e3_27}
        |\partial_{\theta}^sP(v,\theta)|\simeq v^{f}
    \end{align}
    holds on $\bomega'_{\iota, j}$ for some $f\ge 0$.
\end{lemma}
\begin{proof}[Proof of Lemma \ref{231121lemma3_2}]
Recall that 
\begin{equation}
    P_{E_{\iota}}(v, \theta)=\sum_{
    (p, q)\in E_{\iota}
    }
    v^p\theta^q,
\end{equation}
and 
\begin{equation}
    P_{E_{\iota}}(v, r v^{m_{\iota}})
    =
    \sum_{(p, q)\in E_{\iota}}
    c_{p, q}
    v^p r^q v^{qm_{\iota}}=
    v^{e_{\iota}} \sum_{(p, q)\in E_{\iota}}
    c_{p, q}
    r^q. 
\end{equation}
    The highest order and lowest power of $r$ appearing in $P_{E_{\iota}}(r)$ is $r^{q_{E_{\iota}, l}}$ and $r^{q_{E_{\iota}, r}}$. From this, the bound \eqref{231121e3_24} follows immediately. We also have
    \begin{align}
        m_{
        \iota
        }s_{
        \iota, j}
        \leq m_{\iota}
        (q_{E_{\iota}, l}- q_{E_{\iota}, r})\leq m_{\iota}
        q_{E_{\iota}, l}\leq p_{E_{\iota}, l}+m_{\iota}q_{E_{\iota}, l}=e_{\iota}.
    \end{align}
    This finishes the proof of \eqref{231121e3_25}.\\

Next we make the coordinate changes $(v,\theta)=(v,v^{m_{\iota}}\theta')$. Then on the bad set $\bomega'_{\iota, j}$ we have
\begin{align}               
P(v,\theta)=P(v,v^{m_{\iota}}\theta')=v^{e_{\iota}}
\bnorm{
\sigma_{\iota, j}
(\theta'-r_{\iota, j})^{s_{\iota, j}}+O
\pnorm{
(\theta'-r_{\iota, j})^{s_{\iota, j}+1}
}
}+o(v^{e_{\iota}})
\end{align}
for some $\sigma_{\iota, j} \neq 0$.
Assume $s\leq s_{\iota, j}$. By the chain rule, we obtain 
\begin{align}\label{231121e3_32}
\partial_{\theta}^sP(v,\theta)=v^{e_{\iota}-m_{\iota} s}
\bnorm{
\rho_{\iota, j, s}(\theta'-r_{\iota, j})^{s_{\iota, j}-s}+O\pnorm{
(\theta'-r_{\iota, j})^{s_{\iota, j}+1-s}
}
}+o(v^{e_{\iota}-sm_{\iota}}),
\end{align}
for some $\rho_{\iota, j, s}\neq 0$. 
In particular, if we choose $s=s_{\iota, j}$, then we obtain 
\begin{align}
    \partial_{\theta}^{s_{\iota, j}}P(v,\theta)=v^{e_{\iota}-m_{\iota} s_{\iota, j}}
    \bnorm{
    \rho_{s_{\iota, j}}+O(\theta'-r_{\iota, j})
    }+o(v^{e_{\iota}-m_{\iota} s_{\iota, j}})
\end{align}
This finishes the proof of \eqref{231121e3_26}. The last statement \eqref{231121e3_27} can be proven similarly, and we leave out the proof. This finishes the proof of Lemma \ref{231121lemma3_2}.

\end{proof}

\begin{lemma}[\cite{Gre04}, Lemma 2.3]\label{231121lemma3_3}
For every $1\le \iota\le I$ and every $1\le j\le J_{\iota}$,  if we take $v$ to be sufficiently small, then there is a unique $h_{\iota, j}(v)$ with
\begin{equation}
\partial_{\theta}^{s_{\iota, j}-1}P(v,h_{\iota, j}(v))=0.
\end{equation}
 The function $h_{\iota, j}(v)$ is a real-analytic function of
 \begin{equation}
 v^{\frac{1}{q_{E_{\iota}, l}-q_{E_{\iota}, r}}}
 \end{equation}
  whose leading term is given by
  \begin{equation}
  r_{\iota, j}v^{m_{\iota}}= r_{\iota, j}
  v^{
\frac{
p_{E_{\iota}, r}-p_{E_{\iota}, l}
}{q_{E_{\iota}, l}-q_{E_{\iota}, r}}
  }.
  \end{equation}
   Moreover, on $\bomega'_{\iota, j}$, we have
    \begin{align}
    |\partial_{\theta}^{s_{\iota, j}-1}P(v,\theta)|\simeq v^{e_{\iota}-s_{\iota, j}m_{\iota}}|\theta-h_{\iota, j}(v)|   
    \end{align}
    
\end{lemma}
\begin{proof}[Proof of Lemma \ref{231121lemma3_3}]
The proof of this lemma follows from applying \eqref{231121e3_32} with $s=s_{\iota, j}-1$, and then applying the implicit function theorem. 
\end{proof}

If we are in the case where 
\begin{equation}\label{231121e3_38}
s_{\iota, j}=1, \ \ \forall \iota, j,
\end{equation}
then Lemma \ref{231121lemma3_3} tells us that on each $\bomega'_{\iota, j}$, we have
    \begin{align}
    |P(v,\theta)|\simeq v^{e_{\iota}-s_{\iota, j}m_{\iota}}|\theta-h_{\iota, j}(v)|,
    \end{align}
which is exactly what Theorem \ref{231120theorem3_3} says. Therefore, we only need to consider cases where \eqref{231121e3_38} fails, that is, the multiplicity of some root is bigger than one. 

Let us summarize this subsection by writing down the decomposition we have done so far. Recall \eqref{231121e3_11}, \eqref{231121e3_20} and \eqref{231121e3_23}. We have decomposed the first quadrant into 
\begin{equation}\label{231121e3_40}
\begin{split}
&     \pnorm{\bigcup_{1\le \iota\le I+1} \bomega_{\iota}}
    \bigcup 
    \pnorm{\bigcup_{1\le \iota\le I} \bomega'_{\iota}}\\
    &     =\pnorm{\bigcup_{1\le \iota\le I+1} \bomega_{\iota}}
    \bigcup 
    \pnorm{\bigcup_{1\le \iota\le I} \bomega'_{\iota, \mathrm{good}}}  \bigcup 
    \pnorm{\bigcup_{1\le \iota\le I} \bomega'_{\iota, \mathrm{bad}}}\\
    &     =\pnorm{\bigcup_{1\le \iota\le I+1} \bomega_{\iota}}
    \bigcup 
    \pnorm{\bigcup_{1\le \iota\le I} \bomega'_{\iota, \mathrm{good}}}  \bigcup 
    \pnorm{\bigcup_{1\le \iota\le I} \bigcup_{j=1}^{J_{\iota}}
    \bomega'_{\iota, j}}.
\end{split}
\end{equation}
Moreover, there exists a positive integer $W$ such that each set from the most right hand side of \eqref{231121e3_40} can be written in the form 
\begin{equation}\label{231121e3_41}
\{(v, \theta): 
c_{w} v^{
\frac{w}{W}
}
<  \theta< 
C_{w'} v^{
\frac{w'}{W}
}
\},
\end{equation}
where $w\le w'$ are positive integers, and $c_{w}, C_{w'}\in [0, \infty]$. In particular, each $\bomega'_{\iota, j}$ is of the form \eqref{231121e3_41} with
\begin{equation}
w=w', \ \ 0< c_{w}< C_{w'}< \infty. 
\end{equation}
In the next subsection, we will decompose $\bomega'_{\iota, j}$ for those $\iota, j$ for which \eqref{231121e3_38} fails.  

\subsection{An iteration algorithm}\label{231122subsection3_3}

Take $\iota, j$ with $s_{\iota, j}>1$. We will decompose the bad region $\bomega'_{\iota, j}$ into more smaller regions. We perform the change of variables 
\begin{equation}
(v_1,\theta_1)=(v,\theta-h_{\iota, j}(v)),
\end{equation}
and consider 
\begin{equation}
P_1(v_1, \theta_1):= P(v_1, \theta_1+h_{\iota, j}(v_1)).
\end{equation}
We use $\bomega^1_{\iota, j}$ to denote the image of $\bomega'_{\iota, j}$ under the above transformation. We know that $\bomega^1_{\iota, j}$ is of the form 
\begin{equation}\label{231121e3_45}
-h^{1, -}_{\iota, j}(v_1)< \theta_1< h^{1, +}_{\iota, j}(v_1),
\end{equation}
where both $h^{1, -}_{\iota, j}(v_1)$ and $h^{1, +}_{\iota, j}(v_1)$ are non-negative analytic functions (for non-negative and small $v_1$) in 
 \begin{equation}
 v_1^{\frac{1}{q_{E_{\iota}, l}-q_{E_{\iota}, r}}}
 \end{equation}
  whose leading term is of the same order as 
  \begin{equation}
  v_1^{m_{\iota}}= 
  v_1^{
\frac{
p_{E_{\iota}, r}-p_{E_{\iota}, l}
}{q_{E_{\iota}, l}-q_{E_{\iota}, r}}
  }.
  \end{equation}
 We will break \eqref{231121e3_45} into two parts, and consider 
 \begin{equation}
 -h^{1, -}_{\iota, j}(v_1)< \theta_1<0, \ \ 0< \theta_1< h^{1, +}_{\iota, j}(v_1),
 \end{equation}
 separately. Without loss of generality, we consider the latter case. 
For the new analytic function $P_1(v_1, \theta_1)$, we repeat the same argument as in Subsection \ref{231121subsection3_1} and Subsection \ref{231121subsection3_2}, and arrive at a similar decomposition as in \eqref{231121e3_40}. The only difference is that we will intersect each of the resulting sets with $\bomega^1_{\iota, j}$. Each of the resulting set is of the form 
\begin{equation}\label{231122e3_49}
\{(v_1, \theta_1): 
c_{w_1} v_1^{
\frac{w_1}{W_1}
}
<  \theta_1< 
C_{w'_1} v_1^{
\frac{w'_1}{W_1}
}
\},
\end{equation}
where $w_1, w'_1, W_1$ are positive integers, and $c_{w_1}, C_{w'_1}\in [0, \infty)$. Moreover, we can always without loss of generality assume that 
\begin{equation}
\frac{w'_1}{W_1}\ge m_{\iota}.
\end{equation}
If we change variables back to $(v, \theta)$, then \eqref{231122e3_49} can be written as 
\begin{equation}
\{(v, \theta): 
h_{\iota, j}(v)+
c_{w_1} v^{
\frac{w_1}{W_1}
}
<  \theta< 
h_{\iota, j}(v)+
C_{w'_1} v^{
\frac{w'_1}{W_1}
}
\},
\end{equation}
which is of the form \eqref{231122e3_7zz}. 

\begin{lemma}[\cite{Gre04}, Lemma 2.4]\label{231122lemma3_4}
Recall the multiplicity parameter $S$ defined in \eqref{231122e3_26pp}. The above algorithm will terminate after repeating the argument in Subsection \ref{231121subsection3_1} and Subsection \ref{231121subsection3_2} at most $S+1$ times. 
\end{lemma}

Roughly speaking, Lemma \ref{231122lemma3_4} says that every time when we repeat the argument in Subsection \ref{231121subsection3_1} and Subsection \ref{231121subsection3_2}, the multiplicity parameter $s_{\iota, j}$ that is fixed at the beginning of Subsection \ref{231122subsection3_3} will drop by at least one, and therefore the algorithm will terminate after finitely many steps. In the end, we collect all the resulting functions and sets and write them in the form as in Theorem \ref{231120theorem3_3}.

\section{Sharp \texorpdfstring{$L^p$}{} bounds: Vertices-dominating}\label{240124section4}

Let $\gamma$ be an analytic function satisfying the assumptions \eqref{231122e1_3kk} and \eqref{231122e1_4kk}. Moreover, we assume that $\gamma$ is not strongly degenerate. In the rest of the paper we will prove that there exists $\epsilon>0$ such that 
\begin{equation}\label{240102e4_1}
\norm{
\mc{M}_{\gamma} f
}_{
L^p(\R^2)
} \lesim_{\gamma, a, p} \norm{f}_{L^p(\R^2)},
\end{equation}
for all
\begin{equation}
p> \max\{2, \mathfrak{D}_{\gamma}\},
\end{equation}
and for all smooth functions $a$ supported in $\B_{\epsilon}$, the ball of radius $\epsilon$ centered at the origin.  \\

Consider the reduced Newton diagram $\mc{R}\mc{N}_d(\gamma)$. 
Denote 
\begin{equation}\label{240106e4_3}
\gamma_{\mf{r}\mf{e}}(v, \theta):=
\sum_{
(\mf{p}, \mf{q}): \mf{q}\neq 0
}
c_{\mf{p}, \mf{q}} v^{\mf{p}} \theta^{\mf{q}},
\end{equation}
where 
\begin{equation}\label{240102e4_10}
c_{\mfp, \mfq}:=
\frac{1}{\mfp!\mfq!}
\partial_v^{\mfp} \partial_{\theta}^{\mfq} 
\gamma(0, 0),
\end{equation}
and $\mf{r}\mf{e}$ is short for ``reduced". 
We label all the vertices of $\mc{R}\mc{N}_d(\gamma)$ by 
\begin{equation}\label{240108e4_5}
(\mf{p}_1, \mf{q}_1), (\mf{p}_2, \mf{q}_2), \dots, (\mf{p}_{I_1}, \mf{q}_{I_1}),
\end{equation}
for some positive integer $I_1$, 
with 
\begin{equation}
\mf{p}_1<\dots< \mf{p}_{I_1}.
\end{equation}
Denote 
\begin{equation}
\mf{m}_{i}=
\frac{
\mf{p}_{i+1}-\mf{p}_i
}{
\mf{q}_{i}-\mf{q}_{i+1}
}.
\end{equation}
Consider the regions 
\begin{equation}\label{240124e4_8}
\begin{split}
& \mf{O}_1:=
\B_{\epsilon}\cap 
\{
(v, \theta):
|\theta|\ge C_{\gamma}|v|^{\mf{m}_1}
\}, \\
& 
\mf{O}_i:=
\B_{\epsilon}
\cap 
\{
(v, \theta):
(C_{\gamma})^{-1} |v|^{\mf{m}_{i-1}}\ge |\theta|\ge C_{\gamma}
|v|^{\mf{m}_i}
\}, \ \ 1< i\le I_1,\\
& \mf{O}_{I_1+1}:=
\B_{\epsilon}\cap 
\{
(v, \theta): 
|\theta|\le (C_{\gamma})^{-1} |v|^{\mf{m}_{I_1}}
\}.
\end{split}
\end{equation}
Here we use $C_{\gamma}$ to denote a large constant that is allowed to depend on $\gamma$. If $C_{\gamma}$ is chosen to be sufficiently large, then for every $1\le i\le I_1+1$, we always have that 
\begin{equation}
|\gamma_{
\mf{r}\mf{e}
}(v, \theta)|\simeq 
|v|^{
\mf{p}_i
} 
|\theta|^{
\mf{q}_i
}, \ \ (\theta, v)\in \mf{O}_i.
\end{equation}
For $1\le i\le I_{1}$, denote 
\begin{equation}\label{240220e5_10}
\mf{O}'_i:=
\B_{\epsilon}\cap 
\{
(v, \theta):
(C_{\gamma})^{-1}|v|^{\mf{m}_i}\le |\theta|\le 
C_{\gamma} |v|^{\mf{m}_i}
\}.
\end{equation}
We have 
\begin{equation}
    \B_{\epsilon}=
    \pnorm{
    \bigcup_{i=1}^{I_1+1}\mf{O}_i
    }\bigcup \pnorm{
    \bigcup_{i=1}^{I_1} \mf{O}'_i
    }.
\end{equation}
We use $\mf{E}_i$ to denote the edge of the Newton diagram that connects the vertices $(\mf{p}_i, \mf{q}_i)$ and $(\mf{p}_{i+1}, \mf{q}_{i+1})$. Moreover, denote 
\begin{equation}
\gamma_{\mf{E}_i}(v, \theta):=
\sum_{
(\mfp, \mfq)\in \mf{E}_i
}
c_{\mfp, \mfq} v^{\mfp} \theta^{\mfq}.
\end{equation}
For a given 
\begin{equation}
r\in [
(C_{\gamma})^{-1}, C_{\gamma}
],
\end{equation}
consider the curve 
\begin{equation}
\{(v, \theta):
 \theta= rv^{\mf{m}_i}
\}.
\end{equation}
On this curve, we have 
\begin{equation}
\gamma_{
\mf{E}_i
}(v, \theta)=
v^{\mf{e}_i}
\sum_{(\mfp, \mfq)\in \mf{E}_i}
c_{\mfp, \mfq}
r^{\mfq}:= 
v^{\mf{e}_i} \gamma_{\mf{E}_i}(r),
\end{equation}
where 
\begin{equation}\label{240102e4_ab}
\mf{e}_i:= \mfp+ \mfq\mf{m}_i,
\end{equation}
and in \eqref{240102e4_ab}, the point $(\mfp, \mfq)$ can be taken to be an arbitrary point in $\mf{E}_i$. \\

We will handle the regions $\mf{O}_i$ in the current section; the regions $\mf{O}'_{
i}$ will require a different argument, and will be discussed in the next section. More precisely, we will prove in this section that 
\begin{equation}\label{240102e4_20}
\Norm{
\sup_{v\in \R}
\anorm{
\int_{
\mf{O}_i(v)
}
f(x-\theta, y-\gamma(v, \theta)) a(v, \theta)d\theta
}
}_{L^p(\R^2)}
\lesim_{
\gamma, p, a
}
\norm{f}_{L^p(\R^2)},
\end{equation}
for all 
\begin{equation}
p> \max\{2, \mf{D}_{\gamma}\},
\end{equation}
where 
\begin{equation}
\mf{O}_i(v):=\{\theta: (v, \theta)\in \mf{O}_i\}.
\end{equation}
For each $|v|\le \epsilon$, let 
\begin{equation}
a_i(v, \cdot): \R\to \R
\end{equation}
 be a non-negative smooth bump function supported on $(1+c_{\gamma}) \mf{O}_i(v)$, and equal to one on $\mf{O}_i(v)$, where $c_{\gamma}>0$ is a small constant depending only on $\gamma$. If $c_{\gamma}$ is chosen to be sufficiently small, then on the support of $a_i(v, \theta)$, we still have 
\begin{equation}
|\gamma_{
\mf{r}\mf{e}
}(v, \theta)|\simeq 
|v|^{
\mf{p}_i
}
|\theta|^{
\mf{q}_i
}.
\end{equation}
Under these new notations, \eqref{240102e4_20} follows from 
\begin{equation}\label{240102e4_24}
\Norm{
\sup_{|v|\le \epsilon}
\anorm{
\int_{
\R
}
f(x-\theta, y-\gamma(v, \theta)) a_i(v, \theta)d\theta
}
}_{L^p(\R^2)}
\lesim_{
\gamma, p, a
}
\norm{f}_{L^p(\R^2)}.
\end{equation}
We consider several cases separately: The first case is when 
\begin{equation}\label{240107e4_26}
\mf{p}_i\ge 1, \ \ \mf{q}_i\ge 2,
\end{equation}
the second case is when 
\begin{equation}\label{240107e4_27}
\mf{p}_i= 0, \ \ \mf{q}_i\ge 2,
\end{equation}
 the third case is when 
\begin{equation}\label{240107e4_28}
\mf{p}_i\ge 1, \ \ \mf{q}_i=1, \ \ i=1,
\end{equation}
and the last case is when 
\begin{equation}\label{240108e4_29}
\mf{p}_i\ge 1, \ \ \mf{q}_i=1, \ \ i\ge 1.
\end{equation}
The case \eqref{240107e4_28} is the same as saying that the vertex $(\mf{p}_i, 1)$ is the only vertex in the reduced Newton diagram $\rn_d(\gamma)$, while the case \eqref{240108e4_29} says that in the reduced Newton diagram $\rn_d(\gamma)$, there exists a vertex that is different from the vertex $(\mf{p}_1, 1)$.\\

Let us explain the motivation behind the case distinctions \eqref{240107e4_26}-\eqref{240108e4_29}. In the case \eqref{240107e4_26}, the dominating monomial $v^{\mf{p}_i}\theta^{\mf{q}_i}$ has a ``good" cinematic curvature bound. In the other three cases, the dominating monomials $\theta^{\mf{q}_i}$ and $v^{\mf{p}_i}\theta$ both have zero cinematic curvatures. The difference between \eqref{240107e4_27} and \eqref{240107e4_28}-\eqref{240108e4_29} is that in the case \eqref{240107e4_27} the second order derivative (in the $\theta$ variable) of the dominating monomial $\theta^{\mf{q}_i}$ always has a ``favorable" lower bound (rotational curvatures), which allows us to use Van der Corput's lemma in the proof of \eqref{240102e4_24}. The cases \eqref{240107e4_28}-\eqref{240108e4_29} are particularly interesting because of the strongly degenerate examples given by item (2) in Definition \ref{231122defi1_5}. The existence of these strongly degenerate examples must be reflected in the proof, and we will see later that the case distinction between \eqref{240107e4_28} and \eqref{240108e4_29} is how we detect these examples. \\

\subsection{The dominating monomial 
\texorpdfstring{$v^{\mf{p}_i} \theta^{\mf{q}_i}$}{}
satisfies \texorpdfstring{$\mf{p}_i\ge 1, \mf{q}_i\ge 2$}{}. }\label{240106subsection4_1}

Denote 
\begin{equation}\label{240103e4_28}
h_1(v):=
\sum_{\mfp}
c_{\mfp, 0} v^{\mfp},
\end{equation}
where $c_{\mfp, 0}$ is as in \eqref{240102e4_10}. Let $\mfp_0$ be the smallest $\mfp$ such that $c_{\mfp, 0}\neq 0$. If such $\mfp$ does not exist, that is, $c_{\mfp, 0}=0$ for all $\mfp$, then we write $\mfp_0=\infty$. We consider three subcases. The first subcase is when there exists a vertex-tangent line $\mf{L}\in \mf{L}_{(\mf{p}_i, \mf{q}_i)}$ passing through $(\mfp_0, 0)$; the second case is when $(\mfp_0, 0)$ lies on the lower left of every $\mf{L}\in \mf{L}_{(\mf{p}_i, \mf{q}_i)}$, that is, 
if we write the line $\mf{L}\in \mf{L}_{(\mf{p}_i, \mf{q}_i)}$ as 
\begin{equation}
ax+ by=c, \ \ a>0, b>0, c>0,
\end{equation}
then 
\begin{equation}
a \mf{p}_0< c;
\end{equation}
the third case is when $(\mfp_0, 0)$ lies on the upper right of every $\mf{L}\in \mf{L}_{(\mf{p}_i, \mf{q}_i)}$.

\subsubsection{The third subcase}\label{240103subsubsection4_1_1}

We start with the third subcase, which is the simplest case among all the three subcases. What makes this subcase the simplest is that 
\begin{equation}
\gamma(v, \theta)
=
c_{
\mf{p}_i, \mf{q}_i
} v^{\mf{p}_i} \theta^{\mf{q}_i}+ o(
|v|^{\mf{p}_i} |\theta|^{\mf{q}_i}
), \ \ c_{
\mf{p}_i, \mf{q}_i
}\neq 0,
\end{equation}
whenever $(v, \theta)$ lies in the support of $a_i(v, \theta)$ and $|v|\le \epsilon$. Without loss of generality, we assume that $c_{
\mf{p}_i, \mf{q}_i
}=1$. Our goal is to prove that 
\begin{equation}\label{240102e4_30}
\Norm{
\sup_{|v|\le \epsilon}
\anorm{
\int_{
\R
}
f(x-\theta, y-\gamma(v, \theta)) a_i(v, \theta)d\theta
}
}_{L^p(\R^2)}
\lesim_{
\gamma, p, a
}
\norm{f}_{L^p(\R^2)},
\end{equation}
for all $p>2$. Note that in this case we do not see the vertical Newton distance $\mf{D}_{\gamma}$.\\

Let $v(x, y)$ be a measurable function mapping from $\R^2$ to $(-\epsilon, \epsilon)$. To prove \eqref{240102e4_30}, it suffices to prove that 
\begin{equation}\label{240102e4_30zz}
\Norm{
\int_{
\R
}
f(x-\theta, y-\gamma(v(x, y), \theta)) a_i(v(x, y), \theta)d\theta
}_{L^p(\R^2)}
\lesim_{
\gamma, p, a
}
\norm{f}_{L^p(\R^2)},
\end{equation}
where the implicit constant on the right hand side is independent of the measurable function $v(x, y)$. We can without loss of generality assume that 
\begin{equation}
v(x, y)\neq 0, \ \ \forall (x, y).
\end{equation}
Let $\chi: \R\to \R$ be a smooth bump function supported on $(-\epsilon, \epsilon)$. Let $\chi_{+}: \R\to \R$ be a smooth bump function supported on $(-3, -1)\cup (1, 3)$. Denote 
\begin{equation}
    \chi_{j_2}(\theta):=\chi_+(2^{j_2}\theta) \chi(\theta). 
\end{equation}
To prove \eqref{240102e4_30zz}, by the triangle inequality in $j_2$, it suffices to prove that 
\begin{equation}\label{240102e4_34}
\Norm{
\int_{
\R
}
f(x-\theta, y-\gamma(v(x, y), \theta)) a_i(v(x, y), \theta)
\chi_{j_2}(\theta)
d\theta
}_{L^p(\R^2)}
\lesim_{
\gamma, p, a
}
2^{-\delta_p j_2}
\norm{f}_{L^p(\R^2)},
\end{equation}
for every $p>2, j_2\in \N$ and some $\delta_p>0$ depending on $p$. From now on, we will always consider a fixed $j_2$. Moreover, we without loss of generality always assume that 
\begin{equation}\label{240102e4_35}
\mf{O}_i(v(x, y))\cap 
\supp(\chi_{j_2})\neq \emptyset, \ \ \forall (x, y),
\end{equation}
as otherwise the integral on the left hand side of \eqref{240102e4_34} simply vanishes. Denote 
\begin{equation}
\mf{j}_1(x, y):=
-[
\log |v(x, y)|
],
\end{equation}
where $[c]$ means the integral part of $c\in \R$. Recall that on the support of $a_i(v, \theta)$, we always have 
\begin{equation}\label{240103e4_37}
|\gamma(v, \theta)|\simeq 
|v|^{\mf{p}_i}|\theta|^{\mf{q}_i}.
\end{equation}
We apply a Littlewood-Paley decomposition in the second variable, and write \eqref{240102e4_34} as 
\begin{equation}\label{240103e4_38}
\begin{split}
& \Norm{
\sum_{k\in \Z}
\int_{
\R
}
P_k f(x-\theta, y-\gamma(v(x, y), \theta)) a_i(v(x, y), \theta)
\chi_{j_2}(\theta)
d\theta
}_{L^p(\R^2)}\\
& 
\lesim_{
\gamma, p, a
}
2^{-\delta_p j_2}
\norm{f}_{L^p(\R^2)}.
\end{split}
\end{equation}
For the frequencies $k
\le 
\mf{p}_i\mf{j}_1(x, y)
+ \mf{q}_i j_2$, by \eqref{240103e4_37}, the left hand side of \eqref{240103e4_38} can be controlled by the strong maximal operator. More precisely, we have 
\begin{equation}\label{240103e4_39}
\begin{split}
& \Norm{
\sum_{k
\le 
\mf{p}_i\mf{j}_1(x, y)
+ \mf{q}_i j_2
}
\int_{
\R
}
P_k f(x-\theta, y-\gamma(v(x, y), \theta)) a_i(v(x, y), \theta)
\chi_{j_2}(\theta)
d\theta
}_{L^p(\R^2)}\\
& 
\lesim_{
\gamma, p, a
}
2^{-\delta_p j_2}
\norm{f}_{L^p(\R^2)}.
\end{split}
\end{equation}
Therefore, by the triangle inequality, it suffices to prove that 
\begin{equation}\label{240103e4_40}
\begin{split}
& \Norm{
\int_{
\R
}
P_{
\mf{p}_i\mf{j}_1(x, y)
+ \mf{q}_i j_2
+k
} f(x-\theta, y-\gamma(v(x, y), \theta)) a_i(v(x, y), \theta)
\chi_{j_2}(\theta)
d\theta
}_{L^p(\R^2)}\\
& 
\lesim_{
\gamma, p, a
}
2^{-\delta_p j_2}
2^{-\delta_p k}
\norm{f}_{L^p(\R^2)}, \ \ \forall k\in \N. 
\end{split}
\end{equation}
Define 
\begin{equation}
v_{j_1}(x, y):=2^{-j_1} \frac{v(x, y)}{2^{-\mf{j}_1(x, y)}}.
\end{equation}
By Littlewood-Palay inequalities, and by a change of variable 
\begin{equation}
\theta\mapsto 2^{-j_2}\theta,
\end{equation}
and renaming the parameter $v\to 2^{-j_1}v$, 
\eqref{240103e4_40} follows from 
\begin{equation}\label{240103e4_42}
\begin{split}
& \Norm{
\int_{
\R
}
P_{
k
} f(x-\theta, y-\gamma_{\bfj}(v_{0}(x, y), \theta)) a_i(2^{-j_1} v_{0}(x, y), 2^{-j_2} \theta)
\chi_{+}(\theta)
d\theta
}_{L^p(\R^2)}\\
& 
\lesim_{
\gamma, p, a
}
2^{j_2-\delta_p j_2}
2^{-\delta_p k}
\norm{f}_{L^p(\R^2)}, \ \ \forall k\in \N,
\end{split}
\end{equation}
where 
\begin{equation}
\gamma_{\bfj}(v, \theta):=
2^{
\mf{p}_i j_1+\mf{q}_i j_2
}
\gamma(2^{-j_1}v, 2^{-j_2}\theta), \ \ \bfj=(j_1, j_2). 
\end{equation}
Note that
\begin{equation}
|\gamma_{\bfj}(v, \theta)|\simeq |v|^{\mf{p}_i}|\theta|^{\mf{q}_i},
\end{equation}
for all $|v|\simeq 1$ and $|\theta|\simeq 1$, and $\mf{p}_i\ge 1, \mf{q}_i\ge 2$. The desired estimate \eqref{240103e4_42} follows from the local smoothing estimates of Mockenhaupt, Seeger and Sogge \cite[Theorem 1]{MSS92}.

\subsubsection{The second subcase}\label{240104subsubsection4_1_2}

We consider the subcase where $(\mfp_0, 0)$ lies on the lower left of every $\mf{L}\in \mf{L}_{(\mfp_i, \mfq_i)}$. We follow the first few steps in the proof of the third subcase in Subsection \ref{240103subsubsection4_1_1}. More precisely, we will prove 
\begin{equation}\label{240103e4_43}
\Norm{
\int_{
\R
}
f(x-\theta, y-\gamma(v(x, y), \theta)) a_i(v(x, y), \theta)
\chi_{j_2}(\theta)
d\theta
}_{L^p(\R^2)}
\lesim_{
\gamma, p, a
}
2^{-\delta_p j_2}
\norm{f}_{L^p(\R^2)},
\end{equation}
for all $j_2\in \N$, and all 
\begin{equation}
p> \max\{2, \mf{D}_{\gamma}\},
\end{equation}
where we still without loss of generality assume that \eqref{240102e4_35} holds, that is, 
\begin{equation}\label{240103e4_45}
\mf{O}_i(v(x, y))\cap 
\supp(\chi_{j_2})\neq \emptyset, \ \ \forall (x, y).
\end{equation}
We remark here that in Subsection \ref{240103subsubsection4_1_1} the estimate \eqref{240102e4_30} was proven for all $p>2$. In this subsection, we will start to see the exponent $\mf{D}_{\gamma}$. \\

We apply a Littlewood-Paley decomposition in the second variable, and write \eqref{240103e4_43} as 
\begin{equation}\label{240103e4_46}
\begin{split}
& \Norm{
\sum_{k\in \Z}
\int_{
\R
}
P_k f(x-\theta, y-\gamma(v(x, y), \theta)) a_i(v(x, y), \theta)
\chi_{j_2}(\theta)
d\theta
}_{L^p(\R^2)}\\
& 
\lesim_{
\gamma, p, a
}
2^{-\delta_p j_2}
\norm{f}_{L^p(\R^2)}.
\end{split}
\end{equation}
The assumption that $(\mfp_0, 0)$ lies on the lower left of every $\mf{L}\in \mf{L}_{(\mfp_i, \mfq_i)}$ implies that 
\begin{equation}
|\gamma(v, \theta)|\simeq |v|^{\mf{p}_0},
\end{equation}
for every $(v, \theta)$ in the support of $a_i(v, \theta)$. As a consequence, we obtain 
\begin{equation}\label{240103e4_48}
\begin{split}
& \Norm{
\sum_{k\le 
\mf{p}_0 \mf{j}_1(x, y)
}
\int_{
\R
}
P_k f(x-\theta, y-\gamma(v(x, y), \theta)) a_i(v(x, y), \theta)
\chi_{j_2}(\theta)
d\theta
}_{L^p(\R^2)}\\
& 
\lesim_{
\gamma, p, a
}
2^{-\delta_p j_2}
\norm{f}_{L^p(\R^2)}.
\end{split}
\end{equation}
It remains to prove 
\begin{equation}\label{240103e4_49}
\begin{split}
& \Norm{
\sum_{
\mf{p}_i \mf{j}_1(x, y)+ \mf{q}_i j_2\ge 
k\ge 
\mf{p}_0 \mf{j}_1(x, y)
}
\int_{
\R
}
P_k f(x-\theta, y-\gamma(v(x, y), \theta)) a_i(v(x, y), \theta)
\chi_{j_2}(\theta)
d\theta
}_{L^p(\R^2)}\\
& 
\lesim_{
\gamma, p, a
}
2^{-\delta_p j_2}
\norm{f}_{L^p(\R^2)},
\end{split}
\end{equation}
and 
\begin{equation}\label{240103e4_50}
\begin{split}
& \Norm{
\sum_{k\ge 
\mf{p}_i \mf{j}_1(x, y)+ \mf{q}_i j_2
}
\int_{
\R
}
P_k f(x-\theta, y-\gamma(v(x, y), \theta)) a_i(v(x, y), \theta)
\chi_{j_2}(\theta)
d\theta
}_{L^p(\R^2)}\\
& 
\lesim_{
\gamma, p, a
}
2^{-\delta_p j_2}
\norm{f}_{L^p(\R^2)}.
\end{split}
\end{equation}
Here we remark that the assumption that $(\mfp_0, 0)$ lies on the lower left of every $\mf{L}\in \mf{L}_{(\mfp_i, \mfq_i)}$  guarantees that 
\begin{equation}
\mf{p}_i \mf{j}_1(x, y)+ \mf{q}_i j_2\ge 
\mf{p}_0 \mf{j}_1(x, y)
\end{equation}
for all $(x, y)$. \\

We first prove \eqref{240103e4_49}. Let $j_1$ be a possible value of the function $\mf{j}_1(x, y)$. The assumption \eqref{240103e4_45} says that 
\begin{equation}\label{240103e4_52}
\begin{split}
& 
j_1 \mf{p}_i+ j_2 \mf{q}_i\le j_1 \mf{p}_{i-1}+ j_2 \mf{q}_{i-1}, \\
& j_1 \mf{p}_i+ j_2 \mf{q}_i\le j_1 \mf{p}_{i+1}+ j_2 \mf{q}_{i+1}. 
\end{split}
\end{equation}
Let $\mf{J}_1$ be the sup of all the possible values of the function $\mf{j}_1(x, y)$. We bound the left hand side of \eqref{240103e4_49} by 
\begin{equation}
\pnorm{
\sum_{j_1=0}^{
\mf{J}_1
}
\anorm{
\sum_{
\mf{p}_i j_1+ \mf{q}_i j_2\ge 
k\ge 
\mf{p}_0 j_1
}
\int_{
\R
}
P_k f(x-\theta, y-\gamma(v_{j_1}(x, y), \theta)) a_i(v_{j_1}(x, y), \theta)
\chi_{j_2}(\theta)
d\theta
}^p
}^{
\frac{1}{p}
}
\end{equation}
and will prove that 
\begin{equation}\label{240103e4_54}
\begin{split}
& \Norm{
\sum_{
\mf{p}_i j_1+ \mf{q}_i j_2\ge 
k\ge 
\mf{p}_0 j_1
}
\int_{
\R
}
P_k f(x-\theta, y-\gamma(v_{j_1}(x, y), \theta)) a_i(v_{j_1}(x, y), \theta)
\chi_{j_2}(\theta)
d\theta
}_{L^p(\R^2)}\\
& 
\lesim_{
\gamma, p, a
}
2^{-\delta_p j_2}
\norm{f}_{L^p(\R^2)},
\end{split}
\end{equation}
for all allowable $j_1$. The estimate \eqref{240103e4_54}, combined with the following Claim \ref{240103claim4_1}, immediately implies \eqref{240103e4_49}.
\begin{claim}\label{240103claim4_1} For every $j_1=0, 1, \dots, \mf{J}_1$, we always have 
\begin{equation}
\#\{k: 
\mf{p}_i j_1+ \mf{q}_i j_2\ge 
k\ge 
\mf{p}_0 j_1
\}\lesim_{\gamma} j_2.
\end{equation}
\end{claim}
\begin{proof}[Proof of Claim \ref{240103claim4_1}]
If $\mf{p}_0\ge \mf{p}_i$, then the claim is immediate. Let us assume $\mf{p}_0< \mf{p}_i$. In this case, we will need to make use of \eqref{240103e4_52}. Recall item (3) in Definition \ref{231122defi1_5}, and recall that we are assuming that $\gamma$ is not strongly degenerate. As a consequence of the assumption  $\mf{p}_0< \mf{p}_i$, we can conclude that  $i\ge 2$, that is, $(\mf{p}_{i-1}, \mf{q}_{i-1})$ is also a vertex of the Newton diagram of $\gamma$. \footnote{This is not necessarily true for $(\mf{p}_{i+1}, \mf{q}_{i+1})$. } By the first equation of \eqref{240103e4_52}, we obtain 
\begin{equation}
j_1 (\mf{p}_i-\mf{p}_{i-1})\le j_2 (\mf{q}_{i-1}-\mf{q}_i).
\end{equation}
Therefore, 
\begin{equation}
\#\{k: 
\mf{p}_i j_1+ \mf{q}_i j_2\ge 
k\ge 
\mf{p}_0 j_1
\}\le 
(\mf{p}_i-\mf{p}_0)j_1+ \mf{q}_i j_2 \lesim_{\gamma} j_2.
\end{equation}
This finishes the proof of the claim. 
\end{proof}
As noted above, to prove \eqref{240103e4_49}, it remains to prove \eqref{240103e4_54}. By applying Claim \ref{240103claim4_1} again and by the triangle inequality, it suffices to prove that 
\begin{equation}\label{240103e4_58}
\begin{split}
& \Norm{
\int_{
\R
}
P_k f(x-\theta, y-\gamma(v_{j_1}(x, y), \theta)) a_i(v_{j_1}(x, y), \theta)
\chi_{j_2}(\theta)
d\theta
}_{L^p(\R^2)}\\
& 
\lesim_{
\gamma, p, a
}
2^{-\delta_p j_2}
\norm{f}_{L^p(\R^2)},
\end{split}
\end{equation}
for all 
\begin{equation}\label{240103e4_59}
j_1\le \mf{J}_1, \ \ \mf{p}_i j_1+ \mf{q}_i j_2\ge 
k\ge 
\mf{p}_0 j_1.
\end{equation}
Note that the function $P_k f$ is essentially constant in the vertical direction at the scale 
\begin{equation}
2^{-k}\ge 2^{
-(
\mf{p}_i j_1+ \mf{q}_i j_2
)
},
\end{equation}
and therefore \eqref{240103e4_58} is equivalent to proving that 
\begin{equation}\label{240103e4_61}
\begin{split}
& \Norm{
\int_{
\R
}
P_k f(x-\theta, y-h_1(v_{j_1}(x, y))) a_i(v_{j_1}(x, y), \theta)
\chi_{j_2}(\theta)
d\theta
}_{L^p(\R^2)}\\
& 
\lesim_{
\gamma, p, a
}
2^{-\delta_p j_2}
\norm{f}_{L^p(\R^2)},
\end{split}
\end{equation}
where $h_1$ was introduced in \eqref{240103e4_28}. We write the above estimate in maximal operator form, and need to prove 
\begin{equation}\label{240103e4_62}
\begin{split}
& \Norm{
\sup_{
|v|\simeq 2^{-j_1}
}
\anorm{
\int_{
\R
}
P_k f(x-\theta, y-h_1(v))
\chi_{j_2}(\theta)
d\theta
}
}_{L^p(\R^2)}
\lesim
2^{-\delta_p j_2}
\norm{f}_{L^p(\R^2)}.
\end{split}
\end{equation}
Here to simplify the presentation, we are not carrying the term $a_i(v_{j_1}(x, y), \theta)$ in \eqref{240103e4_61}; it can be removed by doing a Fourier series expansion as both $v$ and $\theta$ are localized in dyadic intervals.  By the change of variable 
\begin{equation}
\theta\mapsto 2^{-j_2}\theta,
\end{equation}
and by renaming the parameter $v\mapsto 2^{-j_1}v$, the desired bound \eqref{240103e4_62} follows from 
\begin{equation}\label{240103e4_64}
\begin{split}
& \Norm{
\sup_{
|v|\simeq 1
}
\anorm{
\int_{
\R
}
P_{
k-
\mf{p}_0 j_1
} f(x-\theta, y-\widetilde{h}_1(v))
\chi_{+}(\theta)
d\theta
}
}_{L^p(\R^2)}
\lesim
2^{j_2-\delta_p j_2}
\norm{f}_{L^p(\R^2)},
\end{split}
\end{equation}
where 
\begin{equation}
\widetilde{h}_1(v):=2^{
\mf{p}_0
 j_1} h_1\left(2^{-j_1} v\right).
\end{equation}
By Sobolev embedding, the left hand side of \eqref{240103e4_64} is bounded by 
\begin{equation}
\begin{split}
& 2^{
\frac{
k-\mf{p}_0 j_1
}{p}
}
\Norm{
\int_{
\R
}
P_{
k-
\mf{p}_0 j_1
} f(x-\theta, y-\widetilde{h}_1(v))
\chi_{+}(\theta)
d\theta
}_{
L^p_{x, y, v}
}\\
& \lesim 
2^{
\frac{
k-\mf{p}_0 j_1
}{p}
}
\norm{f}_{L^p(\R^2)}
\lesim 
2^{
\frac{
\mf{p}_i j_1+ \mf{q}_i j_2- \mf{p}_0 j_1
}{p}
} 
\norm{f}_{L^p(\R^2)},
\end{split}
\end{equation}
where in the last step we used \eqref{240103e4_59}. The constraints \eqref{240103e4_52} tells us that 
\begin{equation}\label{240104e4_69}
\frac{
\mf{q}_{i}-\mf{q}_{i+1}
}{
\mf{p}_{i+1}-\mf{p}_{i}
}
\le 
\frac{j_1}{j_2} \le \frac{
\mf{q}_{i-1}-\mf{q}_i
}{
\mf{p}_{i}-\mf{p}_{i-1}
}.
\end{equation}
By letting 
\begin{equation}
2^{
\frac{
\mf{p}_i j_1+ \mf{q}_i j_2- \mf{p}_0 j_1
}{p}
} 
\le 2^{j_2},
\end{equation}
we obtain 
\begin{equation}
p\ge 
\mf{q}_i+(\mf{p}_i-\mf{p}_0)j_1/j_2.
\end{equation}
This, combined with \eqref{240104e4_69}, leads to the constraint 
\begin{equation}\label{240104e4_72}
p\ge 
\sup_{
\mf{L}\in \mf{L}_{
(\mf{p}_i, \mf{q}_i)
}
}
\mf{d}(
\mf{L}, (\mf{p}_0, 0)
). 
\end{equation}
Note that \eqref{240103e4_64} holds at $p=\infty$ with $\delta_p=1$. Interpolation show that \eqref{240103e4_64} holds for all
\begin{equation}
p> \max\{
2, \mf{D}_{\gamma}
\}.
\end{equation}
This finishes estimating the term \eqref{240103e4_49}.\\

Next we prove \eqref{240103e4_50}. By the triangle inequality, it suffices to prove 
\begin{equation}\label{240103e4_69}
\begin{split}
& \Norm{
\int_{
\R
}
P_{
\mf{p}_i \mf{j}_1(x, y)+ \mf{q}_i j_2+k
} f(x-\theta, y-\gamma(v(x, y), \theta)) a_i(v(x, y), \theta)
\chi_{j_2}(\theta)
d\theta
}_{L^p(\R^2)}\\
& 
\lesim_{
\gamma, p, a
}
2^{-\delta_p j_2
-\delta_p k
}
\norm{f}_{L^p(\R^2)}, \ \ \forall k\in \N. 
\end{split}
\end{equation}
By the Littlewood-Paley inequality, we just need to prove 
\begin{equation}\label{240103e4_70}
\begin{split}
& \Norm{
\int_{
\R
}
P_{
\mf{p}_i j_1+ \mf{q}_i j_2+k
} f(x-\theta, y-\gamma(v_{j_1}(x, y), \theta)) a_i(v_{j_1}(x, y), \theta)
\chi_{j_2}(\theta)
d\theta
}_{L^p(\R^2)}\\
& 
\lesim_{
\gamma, p, a
}
2^{-\delta_p j_2
-\delta_p k
}
\norm{f}_{L^p(\R^2)},
\end{split}
\end{equation}
for all $(j_1, j_2)$ satisfying \eqref{240103e4_52}. We write \eqref{240103e4_70} in maximal operator form, and need to prove 
\begin{equation}\label{240103e4_71}
\begin{split}
& \Norm{
\sup_{
|v|\simeq 2^{-j_1}
}
\anorm{
\int_{
\R
}
P_{
\mf{p}_i j_1+ \mf{q}_i j_2+k
} f(x-\theta, y-\gamma(v, \theta))
\chi_{j_2}(\theta)
d\theta
}
}_{L^p(\R^2)}\\
& 
\lesim
2^{-\delta_p j_2
-\delta_p k
}
\norm{f}_{L^p(\R^2)}.
\end{split}
\end{equation}
Here to simplify the presentation, we again removed the amplitude function $a_i$, and this can be done via Fourier series expansion.  By the change of variable 
\begin{equation}
\theta\mapsto 2^{-j_2}\theta,
\end{equation}
and by renaming the parameter $v\to 2^{-j_1}v$, it is equivalent to prove 
\begin{equation}\label{240104e4_73}
\begin{split}
& \Norm{
\sup_{
|v|\simeq 1
}
\anorm{
\int_{
\R
}
P_{
k
} f(x-\theta, y-\gamma_{\bfj}(v, \theta))
\chi_{+}(\theta)
d\theta
}
}_{L^p(\R^2)}
\lesim
2^{j_2-\delta_p j_2
-\delta_p k
}
\norm{f}_{L^p(\R^2)},
\end{split}
\end{equation}
where 
\begin{equation}
\gamma_{\mathbf{j}}(v, \theta):=2^{
\mf{p}_i j_1+\mf{q}_i j_2
} \gamma\left(2^{-j_1} v, 2^{-j_2} \theta\right).
\end{equation}
Note that in the current case, the function $\gamma_{\bfj}$ contains terms with very large coefficients. More precisely, 
\begin{equation}
\gamma_{\bfj}(v, \theta)
=
v^{\mf{p}_i}\theta^{\mf{q}_i}+
o(|v|^{\mf{p}_i}|\theta|^{\mf{q}_i})+
2^{
\mf{p}_i j_1+\mf{q}_i j_2
}
h_1(2^{-j_1}v),
\end{equation}
and the leading term of $2^{
\mf{p}_i j_1+\mf{q}_i j_2
}
h_1(2^{-j_1}v)$ is 
\begin{equation}
2^{
\mf{p}_i j_1+\mf{q}_i j_2
}
2^{-\mf{p}_0 j_1} v^{\mf{p}_0}. 
\end{equation}
By Sobolev embedding, to prove \eqref{240104e4_73}, it suffices to prove 
\begin{equation}\label{240104e4_82}
\begin{split}
& \Norm{
\int_{
\R
}
P_{
k
} f(x-\theta, y-\gamma_{\bfj}(v, \theta))
\chi_{+}(\theta)
d\theta
}_{L^p_{x, y, v}}\\
&
\lesim
2^{-\frac{\mf{p}_i j_1+\mf{q}_i j_2-\mf{p}_0 j_1+k}{p}}
2^{j_2-\delta_p j_2
-\delta_p k
}
\norm{f}_{L^p(\R^2)},
\end{split}
\end{equation}
where on the left hand side, when integrating in $v$, we have $|v|\simeq 1$. By local smoothing estimates of Mockenhaupt, Seeger and Sogge \cite[Theorem 1]{MSS92},  we have 
\begin{equation}
\Norm{
\int_{
\R
}
P_{
k
} f(x-\theta, y-\gamma_{\bfj}(v, \theta))
\chi_{+}(\theta)
d\theta
}_{L^p_{x, y, v}}
\lesim 
2^{
-\frac{k}{p}-\delta_p k
}\norm{f}_p,
\end{equation}
for some $\delta_p>0$. Exactly the same computations as in \eqref{240104e4_69}-\eqref{240104e4_72} show that \eqref{240104e4_82} holds for all 
\begin{equation}
p> \max\{
2, \mf{D}_{\gamma}
\}.
\end{equation}
This finishes estimating the term \eqref{240103e4_50}, and therefore finishes the whole discussion of this subsection.

\subsubsection{The first subcase}\label{240107subsubsection4_1_3}

In this subsection we consider the subcase where there exists $\mf{L}\in \mf{L}_{(\mf{p}_i, \mf{q}_i)}$ that passes through the point $(\mf{p}_0, 0)$. This subcase is the most complicated among all the three subcases. However, we will see that the proof of the current subcase is just a combination of the proofs of the previous two subcases. \\

Let us first write down what we need to prove. We copy \eqref{240103e4_45} and \eqref{240103e4_46}, and will prove 
\begin{equation}\label{240104e4_85}
\begin{split}
& \Norm{
\sum_{k\in \Z}
\int_{
\R
}
P_k f(x-\theta, y-\gamma(v(x, y), \theta)) a_i(v(x, y), \theta)
\chi_{j_2}(\theta)
d\theta
}_{L^p(\R^2)}\\
& 
\lesim_{
\gamma, p, a
}
2^{-\delta_p j_2}
\norm{f}_{L^p(\R^2)},
\end{split}
\end{equation}
for all $j_2\in \N$, 
\begin{equation}
p> \max\{
2, \mf{D}_{\gamma}
\},
\end{equation}
and some $\delta_p>0$, where $v(x, y)$ takes values in $(-\epsilon, \epsilon)$ and satisfies 
\begin{equation}
\mf{O}_i(v(x, y))\cap 
\supp(\chi_{j_2})\neq \emptyset, \ \ \forall (x, y).
\end{equation}
By splitting the left hand side of \eqref{240104e4_85} into two pieces, it suffices to prove \eqref{240104e4_85} under the assumption 
\begin{equation}\label{240104e4_88}
2^{
-\mf{j}_1(x, y) \mf{p}_i- j_2\mf{q}_i
}\le 2^{
-\mf{p}_0 j_1
}, \ \ \forall (x, y),
\end{equation}
and under the assumption 
\begin{equation}\label{240104e4_89}
2^{
-\mf{j}_1(x, y) \mf{p}_i- j_2\mf{q}_i
}\ge 2^{
-\mf{p}_0 j_1
}, \ \ \forall (x, y),
\end{equation}
respectively. Under the assumption \eqref{240104e4_88}, we have 
\begin{equation}
|\gamma(v(x, y), \theta)|\simeq |v(x, y)|^{\mf{p}_0},
\end{equation}
and under the assumption \eqref{240104e4_89}, we have 
\begin{equation}
|\gamma(v(x, y), \theta)|\simeq |v(x, y)|^{\mf{p}_i}|\theta|^{\mf{q}_i}.
\end{equation}
The proof of the case \eqref{240104e4_89} is precisely the same as the proof in Subsection \ref{240103subsubsection4_1_1}, and the proof of the case \eqref{240104e4_88} is essentially the same as the proof in Subsection \ref{240104subsubsection4_1_2}. The details are left out.

\subsection{The dominating monomial is \texorpdfstring{$\theta^{\mf{q}_i}$}{} with \texorpdfstring{$\mf{q}_i\ge 2$}{}. }\label{240107subsection4_2}

Recall from \eqref{240102e4_24} that we need to prove 
\begin{equation}\label{240104e4_91}
\Norm{
\sup_{|v|\le \epsilon}
\anorm{
\int_{
\R
}
f(x-\theta, y-\gamma(v, \theta)) a_i(v, \theta)d\theta
}
}_{L^p(\R^2)}
\lesim_{
\gamma, p, a
}
\norm{f}_{L^p(\R^2)},
\end{equation}
for all 
\begin{equation}
p> \max\{
2, \mf{D}_{\gamma}
\}.
\end{equation}
The main difficulties in this case is that the cinematic curvature of the monomial $\theta^{\mf{q}_i}$ vanishes constantly, as $\theta^{\mf{q}_i}$ does not depend on the $v$ variable.

Recall the definition of $h_1(v)$ in \eqref{240103e4_28}, and the notation $\mf{p}_0$. We will consider two subcases separately. The first subcase is when there exists a vertex-tangent line $\mf{L}\in \mf{L}_{(
\mf{p}_i, \mf{q}_i
)}$ passing through $(\mf{p}_0, 0)$, and the second subcase is when $(\mf{p}_0, 0)$ lies on the upper right of  every vertex-tangent line $\mf{L}\in \mf{L}_{(
\mf{p}_i, \mf{q}_i
)}$. The proofs for these two subcases are similar, and the proof for the second subcase is slightly less technical as it involves fewer case distinctions. In particular, in the second subcase we will prove \eqref{240104e4_91} for all $p>2$, and in the first subcase we will see the exponent $\mf{D}_{\gamma}$.  We will discuss the proof of the second subcase first, and then sketch the proof of the first subcase.

\subsubsection{The second subcase}\label{240106subsection4_2_1}

In this subcase, we always have 
\begin{equation}
|\gamma(v, \theta)|\simeq 
|\theta|^{\mf{q}_i}, \ \ \forall (v, \theta)\in \supp(a_i).
\end{equation}
To prove \eqref{240104e4_91}, we without loss of generality will only prove 
\begin{equation}\label{240104e4_95}
\Norm{
\sup_{0\le v\le \epsilon}
\anorm{
\int_{
\R
}
f(x-\theta, y-\gamma(v, \theta)) a_i(v, \theta)d\theta
}
}_{L^p(\R^2)}
\lesim_{
\gamma, p, a
}
\norm{f}_{L^p(\R^2)}.
\end{equation} 
Moreover, as mentioned above, we will prove \eqref{240104e4_95} for all $p>2$.  Denote 
\begin{equation}
    P(v, \theta):= \cine(\gamma)(v, \theta),
\end{equation}
the cinematic curvature of $\gamma$. 
We apply Theorem \ref{231120theorem3_3} to the analytic function $P(v, \theta)$, and write $\B_{\epsilon}^+$
  as a disjoint union 
\begin{equation}
U_0\bigcup \pnorm{
\bigcup_{n=1}^N
U_n
}\bigcup U_{N+1}.
\end{equation}
We recall some properties of this decomposition. We have  a sequence of functions
\begin{equation}
\Theta_0(v)\le \Theta_1(v)\le \dots\le \Theta_N(v)
\end{equation}
that are positive when $v>0$, and 
\begin{equation}
\begin{split}
& U_0(v):=U_0\cap \{(v, \theta): \theta\in \R\}=[0, \Theta_0(v)), \\
& U_n(v):=U_n\cap \{(v, \theta): \theta\in \R\}=
[\Theta_{n-1}(v), \Theta_{n}(v)), \ n=1, 2, \dots, N, \\
& U_{N+1}(v):=U_{N+1}\cap \{(v, \theta): \theta< \epsilon\}=[\Theta_{N}(v), \epsilon). 
\end{split}
\end{equation}
For every $1\le n\le N-1$, there exist a positive integer $\ell_n$ and a sequence of functions 
\begin{equation}
\Theta_{n, 1}, \dots, \Theta_{n, \ell_n}, \Theta_{n, \ell_n}^+, \Theta_{n, \ell_n}^-,
\end{equation}
with 
\begin{equation}
\order(\Theta_{n, 1})< \dots< \order(\Theta_{n, \ell_n})\le 
\min\{\order(\Theta_{n, \ell_n}^-),
\order(\Theta_{n, \ell_n}^+)
\},
\end{equation}
such that 
\begin{equation}
\begin{split}
& 
\Theta_{n-1}=\Theta_{n, 1}+ \dots+ \Theta_{n, \ell_n}+ \Theta_{n, \ell_n}^-, \\
& \Theta_{n}=\Theta_{n, 1}+ \dots+ \Theta_{n, \ell_n}+ \Theta_{n, \ell_n}^+.
\end{split}
\end{equation}
There exists a positive integer $M$ such that all the functions that appear above are analytic functions in $v^{\frac{1}{M}}$ for $0\le v< \epsilon$. On $U_0$ and $U_{N+1}$, it holds that 
\begin{equation}
|P(v, \theta)|\simeq |v|^{a_0} |\theta|^{b_0}, \ \ |P(v, \theta)|\simeq |v|^{a_{N+1}} |\theta|^{b_{N+1}},
\end{equation}
respectively, for some $a_0, b_0, a_{N+1}, b_{N+1}\in \N$, and on $U_n, n=1, \dots, N$, it holds that 
\begin{equation}
|P(v, \theta)|\simeq |v|^{a_n} \anorm{
\theta- \Theta^{\circ}_n(v)
}^{b_n},
\end{equation}
for some $a_n, b_n\in \N$, where
\begin{equation}
\Theta_n^{\circ}=\Theta_{n, 1}+ \dots+ \Theta_{n, \ell_n}.
\end{equation}
This finishes recalling Theorem \ref{231120theorem3_3}. \\

To prove \eqref{240104e4_95}, it suffices to prove 
\begin{equation}\label{231130e4_82}
    \Norm{
    \sup_{0\le v\le \epsilon} 
    \anorm{
    \int_{U_n(v)} f(x-\theta, y-\gamma(v, \theta))
    a_i(v, \theta)
    d\theta
    }
    }_{
    L^p(\R^2)
    }
    \lesim_{\gamma, p} \norm{f}_{L^p(\R^2)},
\end{equation}
for every $n=0, 1, \dots, N+1$. We consider the case $n=0$, the case $n=N+1$ and the case $1\le n\le N$ separately. \\

Let us start with the case $n=0$. Assume that 
\begin{equation}\label{231130e4_83}
\Theta_0(v)= c_0 v^{\alpha_0}+ o(v^{\alpha_0})
\end{equation}
for some $c_0>0$ and $\alpha_0$ is an integer multiple of $1/M$. Recall that for $0\le \theta\le \Theta_0(v)$, it holds that 
\begin{equation}\label{231130e4_84}
|P(v, \theta)|\simeq |v|^{a_0} |\theta|^{b_0}.
\end{equation}
If $c_{\gamma}>0$ is chosen to be sufficiently small, depending on $\gamma$, then we can guarantee that \eqref{231130e4_84} still holds for all 
\begin{equation}\label{231130e4_85}
0\le \theta\le (1+c_{\gamma}) \Theta_0(v).
\end{equation}
Next, for each $v$ we create a smooth partition of unity for the interval 
\begin{equation}\label{240104e4_110}
\supp (a_i(v, \cdot))\cap (
0, (1+c_{\gamma})\Theta_0(v)
).
\end{equation}
Let $\chi_+: \R\to \R$ be a smooth bump function supported on $(1, 3)$. Denote 
\begin{equation}\label{231130e4_86}
\chi_{j_1}(v):=\chi_+(2^{j_1}v), \ \ \chi_{j_2}(\theta):=\chi_+(2^{j_2}\theta),
\end{equation}
and 
\begin{equation}\label{231130e4_87}
\chi_{\bfj}(v, \theta):=\chi_+(2^{j_1}v)\chi_+(2^{j_2}\theta), \ \ \bfj=(j_1, j_2). 
\end{equation}
For each $j_1$, we define $J^-_2(j_1), J^+_2(j_1)\in \N\cup{\infty}$ to be two integers such that 
\begin{equation}
\sum_{
J^-_2(j_1)\le 
j_2\le J^+_2(j_1)
}
\chi_{\bfj}(v, \theta)
\end{equation}
is supported on \eqref{240104e4_110} and equals one on 
\begin{equation}\label{240104e4_114}
\supp (a_i(v, \cdot))\cap (
0, \Theta_0(v)
),
\end{equation} 
for all $v$ in the support of $\chi_{j_1}(v)$. With the new notations, what we need to prove can be written as 
\begin{equation}\label{231130e4_89}
\Norm{
\sup_{j_1} \sup_v 
\anorm{
\sum_{
J^-_2(j_1)\le 
j_2\le J^+_2(j_1)
}
\int_{\R} 
f(x-\theta, y-\gamma(v, \theta))
\chi_{\bfj}(v, \theta)d\theta
}
}_{
L^p(\R^2)
}
\lesim \norm{f}_{L^p(\R^2)}.
\end{equation}
Let $v(x, y)$ be a measurable function from $\R^2$ to $(-\epsilon, \epsilon)$. Without loss of generality, we assume that $v(x, y)\neq 0$ for every $(x, y)$. Denote 
\begin{equation}\label{231130e4_90}
\mathfrak{j}_1(x, y):=-[\log 
|v(x, y)|
],
\end{equation}
and 
\begin{equation}\label{231130e4_91}
v_j(x, y):= 2^{-j}
\frac{
v(x, y)
}{
2^{-
\mathfrak{j}_1(x, y)
}
}. 
\end{equation}
To prove \eqref{231130e4_89}, it suffices to prove that 
\begin{equation}\label{231130e4_92}
\Norm{
\sum_{
J^+_2(
\mathfrak{j}_1(x, y)
)\ge 
j_2\ge J^-_2(
\mathfrak{j}_1(x, y)
)
}
\int_{\R} 
f(x-\theta, y-\gamma(v(x, y), \theta))
\chi_{j_2}(\theta)d\theta
}_{
L^p(\R^2)
}
\lesim \norm{f}_{L^p(\R^2)}.
\end{equation}
By applying the triangle inequality to the sum over $j_2$, it suffices to prove that there exists $\delta_p>0$ such that 
\begin{equation}\label{240104e4_119}
\Norm{
\int_{\R} 
f(x-\theta, y-\gamma(v(x, y), \theta))
\chi_{j_2}(\theta)d\theta
}_{
L^p(\R^2)
}
\lesim
2^{-\delta_p j_2}
 \norm{f}_{L^p(\R^2)},
\end{equation}
where we can without loss of generality assume that 
\begin{equation}
J^+_2(
\mathfrak{j}_1(x, y)
)\ge 
j_2\ge J^-_2(
\mathfrak{j}_1(x, y)
), \ \ \forall (x, y).
\end{equation}
This, combined with \eqref{231130e4_83}, tells us that 
\begin{equation}
2^{
-\alpha_0 \mathfrak{j}_1(x, y)
}\gtrsim 2^{-j_2},
\end{equation}
which further implies the upper bound 
\begin{equation}
\mathfrak{j}_1(x, y)\le j_2/\alpha_0.
\end{equation}
By the triangle inequality, to prove \eqref{240104e4_119}, it suffices to prove that 
\begin{equation}\label{231130e4_97}
\Norm{
\int_{\R} 
f(x-\theta, y-\gamma(v_{j_1}(x, y), \theta))
\chi_{j_2}(\theta)d\theta
}_{
L^p(\R^2)
}
\lesim 
2^{-\delta_p j_2}
\norm{f}_{L^p(\R^2)},
\end{equation}
for every $j_1\le j_2/\alpha_0$. By a change of variable in $\theta$, \eqref{231130e4_97} is equivalent to 
\begin{equation}\label{240104e4_124}
\Norm{
\int_{\R} 
f(x-\theta, y-\widetilde{\gamma}(v_{j_1}(x, y), \theta))
\chi_{+}(\theta)d\theta
}_{
L^p(\R^2)
}
\lesim 
2^{j_2-\delta_p j_2}
\norm{f}_{L^p(\R^2)},
\end{equation}
where 
\begin{equation}
\widetilde{\gamma}(v, \theta):=
2^{\mf{q}_i j_2} \gamma(v, 2^{-j_2}\theta).
\end{equation}
We apply a Littlewood-Paley decomposition in the second variable, and need to prove 
\begin{equation}\label{240104e4_126}
\Norm{
\int_{\R} 
\sum_{k\in \Z} P_k
f(x-\theta, y-\widetilde{\gamma}(v_{j_1}(x, y), \theta))
\chi_{+}(\theta)d\theta
}_{
L^p(\R^2)
}
\lesim 
2^{j_2-\delta_p j_2}
\norm{f}_{L^p(\R^2)}.
\end{equation}
Note that the low frequency part $k\le 0$ can be bounded by the strong maximal operator. 
To prove \eqref{240104e4_126}, it therefore remains to prove that 
\begin{equation}\label{240104e4_128}
\Norm{
\int_{\R} 
\sum_{k\ge 0} P_k
f(x-\theta, y-\widetilde{\gamma}(v_{j_1}(x, y), \theta))
\chi_{+}(\theta)d\theta
}_{
L^p(\R^2)
}
\lesim 
2^{j_2-\delta_p j_2}
\norm{f}_{L^p(\R^2)}.
\end{equation}
By Sobolev embedding and Van der Corput's lemma, we obtain 
\begin{equation}\label{231130e4_102zz}
\Norm{
\int_{\R} 
P_k f(x-\theta, y-\widetilde{\gamma}(v_{j_1}(x, y), \theta))
\chi_{+}(\theta)d\theta
}_{
L^p(\R^2)
}
\lesim 
\norm{f}_{L^p(\R^2)},
\end{equation}
for every $p\ge 2$ and every $k\in \N$.  We compute the cinematic curvature of $\widetilde{\gamma}$: 
\begin{equation}\label{240107e4_133jj}
\begin{split}
& \det\begin{bmatrix}
2^{\mf{q}_i j_2-2j_2}\gamma_{\theta\theta}(v, 2^{-j_2}\theta), & 2^{\mf{q}_i j_2-3j_2}\gamma_{\theta\theta\theta}(v, 2^{-j_2}\theta)\\
2^{\mf{q}_i j_2-j_2}
\gamma_{v\theta}(v, 2^{-j_2}\theta), & 2^{\mf{q}_i j_2-2j_2}
\gamma_{v\theta\theta}(v, 2^{-j_2}\theta)
\end{bmatrix}\\
& =
2^{2\mf{q}_i j_2-4j_2} 
\det\begin{bmatrix}
\gamma_{\theta\theta}(v, 2^{-j_2}\theta), & \gamma_{\theta\theta\theta}(v, 2^{-j_2}\theta)\\
\gamma_{v\theta}(v, 2^{-j_2}\theta), & 
\gamma_{v\theta\theta}(v, 2^{-j_2}\theta)
\end{bmatrix}
\simeq 
2^{2\mf{q}_i j_2-4j_2} 2^{-a_0 j_1-b_0 j_2}.
\end{split}
\end{equation}
\begin{claim}\label{231130claim4_4}
There exists a large constant $C_{\gamma}>0$ such that 
\begin{equation}\label{231130e4_104}
\Norm{
\int_{\R} 
P_k f(x-\theta, y-\widetilde{\gamma}(v_{j_1}(x, y), \theta))
\chi_{+}(\theta)d\theta
}_{
L^p(\R^2)
}
\lesim 
2^{
-\frac{k}{2p}+ C_{\gamma}(j_1+j_2)
}
\norm{f}_{L^p(\R^2)},
\end{equation}
at $p=6$, for every $k\in \N$. The exponent $1/2$ can be replaced by any number smaller than $1$; in our application, any number strictly bigger than zero works. 
\end{claim}
Claim \ref{231130claim4_4} can be proven by slightly modifying the argument in Mockenhaupt, Seeger and Sogge \cite{MSS92}, and will not be detailed here. 

We apply Claim \ref{231130claim4_4} and finish the proof of \eqref{240104e4_128}. By interpolating \eqref{231130e4_104} with a trivial $L^{\infty}$ bound and with the $L^2$ bound in \eqref{231130e4_102zz}, we obtain that 
\begin{equation}\label{231130e4_105}
\Norm{
\int_{\R} 
P_k f(x-\theta, y-\widetilde{\gamma}(v_{j_1}(x, y), \theta))
\chi_{+}(\theta)d\theta
}_{
L^p(\R^2)
}
\lesim 
2^{
-C_p k+ \frac{6C_{\gamma}}{p}(j_1+j_2)
}
\norm{f}_{L^p(\R^2)},
\end{equation}
where 
\begin{equation}
C_p:=
\begin{cases}
\frac{1}{2p}, & \text{ if } p\ge 6,\\
\frac14\pnorm{\frac12-\frac1p}, & \text{ if } 2\le p\le 6.
\end{cases}
\end{equation}
Recall that 
\begin{equation}
j_1\le j_2/\alpha_0.
\end{equation}
Let $C_{\gamma}^+$ be a constant that is much larger compared with $C_{\gamma}$ and $\alpha_0$. We bound the left hand side of \eqref{240104e4_128} by 
\begin{equation}\label{240107e4_138zz}
\begin{split}
& \Norm{
\int_{\R} 
\sum_{k\le C^+_{\gamma}j_2}
P_k f(x-\theta, y-\widetilde{\gamma}(v_{j_1}(x, y), \theta))
\chi_{+}(\theta)d\theta
}_{
L^p(\R^2)
}\\
&+ \Norm{
\int_{\R} 
\sum_{k\ge C^+_{\gamma}j_2}
P_k f(x-\theta, y-\widetilde{\gamma}(v_{j_1}(x, y), \theta))
\chi_{+}(\theta)d\theta
}_{
L^p(\R^2)
}\\
& 
\lesim 
C^+_{\gamma}j_2 \norm{f}_{L^p(\R^2)}+ 
2^{
-C_p C^+_{\gamma} j_2 
+
\frac{6C_{\gamma}}{p}(j_1+j_2)
}
\norm{f}_{L^p(\R^2)}.
\end{split}
\end{equation}
Now it is clear that once $C^+_{\gamma}$ is chosen to be large enough, we will be able to obtain \eqref{240104e4_128}. This finishes the discussion on the case $n=0$. \\

Next, we consider the case $n=N+1$. Assume that 
\begin{equation}
\Theta_N(v)= c_N v^{\alpha_N} + o(v^{\alpha_N}),
\end{equation}
where $c_N>0$. Recall that for $\theta\ge \Theta_N(v)$, it holds that 
\begin{equation}\label{231130e4_111}
|P(v, \theta)|\simeq |v|^{a_N} |\theta|^{b_N}.
\end{equation}
We choose $c_{\gamma}>0$ so small that \eqref{231130e4_111} holds for all
\begin{equation}
\theta\ge (1-c_{\gamma}) \Theta_N(v). 
\end{equation}
Recall the notation from \eqref{231130e4_86} and \eqref{231130e4_87}. For each $j_1$, we define $J^-_2(j_1)$ to be an integer such that 
\begin{equation}
\sum_{
j_2\le J_2^-(j_1)
}
\chi_{\bfj}(v, \theta)
\end{equation}
as a function of $\theta$ is supported on
\begin{equation}
\supp(a_i(v, \cdot ))\cap 
((1-c_{\gamma}) \Theta_N(v), \epsilon)
\end{equation}
 and equals one on 
 \begin{equation}
\supp(a_i(v, \cdot ))\cap 
(\Theta_N(v), \epsilon)
\end{equation}
  for all $v$ in the support of $\chi_{j_1}(v)$. What we need to prove can be written as 
\begin{equation}\label{231130e4_114}
\Norm{
\sup_{j_1}\sup_v 
\anorm{
\sum_{
j_2\le J^-_2(j_1)
}
\int_{\R}
f(x-\theta, y-\gamma(v, \theta))
\chi_{\bfj}(v, \theta)d\theta
}
}_{
L^p(\R^2)
}
\lesim 
\norm{f}_{L^p(\R^2)}.
\end{equation}
Recall the notation $v(x, y), \mathfrak{j}_1(x, y)$ and $v_j(x, y)$ from \eqref{231130e4_90} and \eqref{231130e4_91}. To prove \eqref{231130e4_114}, it is equivalent to prove 
\begin{equation}\label{231130e4_115}
\Norm{
\sum_{
j_2\le J^-_2(
\mathfrak{j}_1(x, y)
)
}
\int_{\R} 
f(x-\theta, y-\gamma(v(x, y), \theta))
\chi_{j_2}(\theta)d\theta
}_{
L^p(\R^2)
}
\lesim \norm{f}_{L^p(\R^2)}.
\end{equation}
We fix $j_2$, and aim to prove that there exists $\delta_p>0$ such that 
\begin{equation}\label{231130e4_116}
\Norm{
\int_{\R} 
f(x-\theta, y-\gamma(v(x, y), \theta))
\chi_{j_2}(\theta)d\theta
}_{
L^p(\R^2)
}
\lesim 
2^{
-\delta_p j_2
}
\norm{f}_{L^p(\R^2)},
\end{equation}
where we can without loss of generality assume that 
\begin{equation}
j_2\le J^-_2(
\mathfrak{j}_1(x, y)
), \ \ \forall (x, y).
\end{equation}
By the change of variable
\begin{equation}
\theta\mapsto 2^{-j_2}\theta,
\end{equation}
 to prove \eqref{231130e4_116}, it suffices to prove that 
\begin{equation}\label{231130e4_118}
\Norm{
\int_{\R} 
f(x-\theta, y-\widetilde{\gamma}(v(x, y), \theta))
\chi_{+}(\theta)d\theta
}_{
L^p(\R^2)
}
\lesim 
2^{j_2
-\delta_p j_2
}
\norm{f}_{L^p(\R^2)},
\end{equation}
where 
\begin{equation}
\widetilde{\gamma}(v, \theta):= 
2^{
\mf{q}_i
j_2
}\gamma(v, 2^{-j_2}\theta).
\end{equation}
By Littlewood-Paley inequalities and $L^p$ bounds for the strong maximal operators, \eqref{231130e4_118} follows from 
\begin{equation}\label{231130e4_120}
\Norm{
\int_{\R} 
\sum_{k\in \N} P_k f(x-\theta, y-\widetilde{\gamma}(v(x, y), \theta))
\chi_{+}(\theta)d\theta
}_{
L^p(\R^2)
}
\lesim 
2^{j_2
-\delta_p j_2
}
\norm{f}_{L^p(\R^2)}.
\end{equation}
Here comes the difference with the previous case. We let $\mf{p}'$ be the smallest positive integer such that 
\begin{equation}
(\mf{p}', \mf{q}')\in \mc{T}(\gamma), \text{ for some } \mf{q}'.
\end{equation}
Note that if
\begin{equation}
k\le \mf{p}' 
\mathfrak{j}_1(x, y),
\end{equation}
 then the function $P_k f$ is essentially constant in the second variable at the scale 
\begin{equation}
2^{-k}\ge 2^{-
\mf{p}' \mf{j}_1(x, y)
},
\end{equation}
and as a consequence of Minkowski's inequality, we obtain 
\begin{equation}\label{231130e4_122}
\Norm{
\int_{\R} 
\sum_{k\le \mf{p}' 
\mathfrak{j}_1(x, y)
} P_k f(x-\theta, y-\widetilde{\gamma}(v(x, y), \theta))
\chi_{+}(\theta)d\theta
}_{
L^p(\R^2)
}
\lesim 
2^{j_2
-\delta_p j_2
}
\norm{f}_{L^p(\R^2)}.
\end{equation}
Therefore it remains to prove that 
\begin{equation}\label{231130e4_123}
\Norm{
\int_{\R} 
\sum_{k\ge \mf{p}' 
\mathfrak{j}_1(x, y)
} P_k f(x-\theta, y-\widetilde{\gamma}(v(x, y), \theta))
\chi_{+}(\theta)d\theta
}_{
L^p(\R^2)
}
\lesim 
2^{j_2
-\delta_p j_2
}
\norm{f}_{L^p(\R^2)}.
\end{equation}
By the triangle inequality and Littlewood-Paley inequalities, it suffices to prove that 
\begin{equation}\label{231130e4_125}
\Norm{
\int_{\R} 
P_{
\mf{p}'
j_1+k
} f(x-\theta, y-\widetilde{\gamma}(v_{j_1}(x, y), \theta))
\chi_{+}(\theta)d\theta
}_{
L^p(\R^2)
}
\lesim 
2^{-
\delta_p k
}
2^{j_2
-\delta_p j_2
}
\norm{f}_{L^p(\R^2)},
\end{equation}
which is equivalent to 
\begin{equation}\label{231130e4_126}
\Norm{
\sup_{|v|\simeq 1}
\anorm{
\int_{\R} 
P_{
\mf{p}'
j_1+k
} f(x-\theta, y-\widetilde{\gamma}(2^{-j_1}v, \theta))
\chi_{+}(\theta)d\theta
}
}_{
L^p(\R^2)
}
\lesim 
2^{-
\delta_p k
}
2^{j_2
-\delta_p j_2
}
\norm{f}_{L^p(\R^2)}.
\end{equation}
Note that by Sobolev embedding, Van der Corput's lemma and interpolation with trivial $L^{\infty}$ bounds, we have\footnote{It is crucial here to gain the factor $2^{-j_1/p}$.}
\begin{equation}\label{240107e4_158pp}
\Norm{
\sup_{|v|\simeq 1}
\anorm{
\int_{\R} 
P_{
\mf{p}'
j_1+k
} f(x-\theta, y-\widetilde{\gamma}(2^{-j_1}v, \theta))
\chi_{+}(\theta)d\theta
}
}_{
L^p(\R^2)
}
\lesim 
2^{
\frac{k}{p}
}
2^{
-\frac{k+j_1}{p}
}
\norm{f}_{L^p(\R^2)},
\end{equation}
for every $p\ge 2$. Let $C_{\gamma}^+$ be a large constant depending on $\gamma$; it is the same as the large constant $C_{\gamma}^+$ in \eqref{240107e4_138zz}. If $k\le C_{\gamma}^+ j_1$, then the desired bound \eqref{231130e4_126} follows directly from \eqref{240107e4_158pp}. When $k\ge C_{\gamma}^+ j_1$, \eqref{231130e4_126}  follows from  Claim \ref{231130claim4_4} and interpolation (see \eqref{231130e4_105}). This finishes the proof of the case $n=N+1$.  \\

In the end, we handle the case $n=1, \dots, N$. Our goal is to prove that 
\begin{equation}\label{231202e4_123}
    \Norm{
    \sup_{0\le v\le \epsilon} 
    \anorm{
    \int_{U_n(v)} f(x-\theta, y-\gamma(v, \theta))
    a_i(v, \theta)
    d\theta
    }
    }_{
    L^p(\R^2)
    }
    \lesim_{\gamma, p} \norm{f}_{L^p(\R^2)}.
\end{equation}
It seems convenient to do the change of variable
\begin{equation}
\theta\mapsto 
\theta+ \Theta_n^{\circ}(v),
\end{equation}
and  prove 
\begin{multline}\label{231202e4_125}
    \Norm{
    \sup_{0\le v\le \epsilon} 
    \anorm{
    \int_{
    \Theta_{n, \ell_n}^-
    }^{
    \Theta_{n, \ell_n}^+
    } f(x-\theta-\Theta_n^{\circ}(v), y-\gamma_n(v, \theta
    ))
    a_i(v, 
    \theta+ \Theta_n^{\circ}(v)
    )
    d\theta
    }
    }_{
    L^p(\R^2)
    }\\
    \lesim \norm{f}_{L^p(\R^2)},
\end{multline}
where 
\begin{equation}
\gamma_n(v, \theta):=
\gamma(v, \theta+ 
    \Theta_n^{\circ}(v)
    ).
\end{equation}
Depending on signs of $\Theta_{n, \ell_n}^-$ and $\Theta_{n, \ell_n}^+$, and on the shape of the intersection 
\begin{equation}
\supp(
a_i(v, \cdot+ \Theta_n^{\circ}(v))
)
\cap 
(
\Theta^-_{n, \ell_n}(v), \Theta^+_{n, \ell_n}(v)
),
\end{equation}
 there are multiple cases to consider, and the most involved case is when $\Theta^-_{n, \ell_n}(v)$ and $\Theta^+_{n, \ell_n}(v)$ have opposite signs, and 
 \begin{equation}
 (
\Theta^-_{n, \ell_n}(v), \Theta^+_{n, \ell_n}(v)
)\subset 
\supp(
a_i(v, \cdot+ \Theta_n^{\circ}(v))
),
 \end{equation}
 which we assume from now on. By splitting the integral in \eqref{231202e4_125} into two parts, we without loss of generality will only prove 
\begin{equation}\label{231202e4_126}
\begin{split}
   &  \Norm{
    \sup_{0\le v\le \epsilon} 
    \anorm{
    \int_{
    0
    }^{
    \Theta_{n, \ell_n}^+
    } f(x-\theta-\Theta_n^{\circ}(v), y-\gamma_n(v, \theta
    ))
    d\theta
    }
    }_{
    L^p(\R^2)
    }
    \lesim \norm{f}_{L^p(\R^2)}.
    \end{split}
\end{equation}
Direct computation shows that 
\begin{equation}
|\cine(\gamma_n)|\simeq |v|^{a_n}|\theta|^{b_n}.
\end{equation}
In the algorithm of resolutions of singularities, we have that 
the new function $\gamma_n(v, \theta)$ always satisfies 
\begin{equation}
|\gamma_n(v, \theta)|\simeq 
|\Theta_n^{\circ}(v)
|^{
\mf{q}_i
},
\end{equation}
as the dominating monomial of $\gamma$ in the current case is $\theta^{\mf{q}_i}$. 
Let us write 
\begin{equation}
\begin{split}
& \Theta_n^{\circ}(v)=c_n v^{\alpha_n}+ o(v^{\alpha_n}),\\
& 
\Theta_{n, \ell_n}^+(v)=
c_n^+ v^{
\alpha_n^+
}+o(v^{
\alpha_n^+
}),
\end{split}
\end{equation}
where 
\begin{equation}
0< \alpha_n\le \alpha_n^+.
\end{equation}
Let us write \eqref{231202e4_126} as 
\begin{equation}\label{231202e4_133}
    \Norm{
    \int_{
    0
    }^{
    \Theta_{n, \ell_n}^+(v(x, y))
    } f(x-\theta-\Theta_n^{\circ}(v(x, y)), y-\gamma_n(v(x, y), \theta
    )
    )d\theta
    }_{
    L^p(\R^2)
    }
    \lesim \norm{f}_{L^p(\R^2)}.
\end{equation}
Recall the definition of $\mathfrak{j}_1(x, y)$. Denote 
\begin{equation}
J_1(x, y):= 
\alpha_n^+ \mathfrak{j}_1(x, y).
\end{equation}
Under this notation, \eqref{231202e4_133} can be written as 
\begin{equation}
    \Norm{
    \sum_{
    j_2\ge 
    J_1(x, y)
    }
    \int
     f(x-\theta-\Theta_n^{\circ}(v(x, y)), y-\gamma_n(v(x, y), \theta
    )
    )
        \chi_{j_2}(\theta)
        d\theta
    }_{
    L^p(\R^2)
    }
    \lesim \norm{f}_{L^p(\R^2)}.
\end{equation}
By the triangle inequality applied to $j_2$, it suffices to prove 
\begin{equation}\label{231203e4_135}
    \Norm{
\int     f(x-\theta-\Theta_n^{\circ}(v(x, y)), y-\gamma_n(v(x, y), \theta
    )
    )
        \chi_{j_2}(\theta)
        d\theta
    }_{
    L^p(\R^2)
    }
    \lesim
    2^{
    -\delta_p j_2
    }
    \norm{f}_{L^p(\R^2)},
\end{equation}
for every $j_2\in \N$, and we can without loss of generality assume that 
\begin{equation}
J_1(x, y)\le j_2, \ \ \forall (x, y).
\end{equation}
By the triangle inequality applied in $j_1$, it suffices to prove that 
\begin{equation}\label{231203e4_138}
    \Norm{
\int     f(x-\theta-\Theta_n^{\circ}(v_{j_1}(x, y)), y-\gamma_n(v_{j_1}(x, y), \theta
    )
    )
        \chi_{j_2}(\theta)
        d\theta
    }_{
    L^p(\R^2)
    }
    \lesim
    2^{
    -\delta_p j_2
    }
    \norm{f}_{L^p(\R^2)}.
\end{equation}
Recall that 
\begin{equation}
\gamma_n(v, \theta)=
(\theta+
\Theta_n^{\circ}(v)
)^{
\mf{q}_i
}+ \text{higher order terms.}
\end{equation}
We apply the change of variable 
\begin{equation}
\theta\mapsto 2^{-j_2}\theta,
\end{equation} 
rename the parameter $v\to 2^{-j_1}v,$ and need to prove 
\begin{equation}\label{231203e4_143}
\begin{split}
&    \Norm{
\int     f(x-2^{-j_2+ \alpha_n j_1}\theta-2^{\alpha_n j_1}\Theta_n^{\circ}(2^{-j_1}v_{0}(x, y)), y-\gamma_{n, \bfj}(v_0(x, y), \theta)
    )
        \chi_{+}(\theta)
        d\theta
    }_{
    L^p(\R^2)
    }\\
    & 
    \lesim
    2^{j_2
    -\delta_p j_2
    }
    \norm{f}_{L^p(\R^2)},
\end{split}
\end{equation}
where 
\begin{equation}
\gamma_{n, \bfj}(v, \theta):=
2^{
\mf{q}_i
\alpha_n j_1
}
\gamma_n(2^{-j_1}v, 2^{-j_2}\theta).
\end{equation}
We first apply a Littlewood-Paley decomposition in the second variable, and need to prove 
\begin{equation}\label{231203e4_145}
\begin{split}
&    \Norm{
\int    P_{\le 0} f(x-2^{-j_2+ \alpha_n j_1}\theta-2^{\alpha_n j_1}\Theta_n^{\circ}(2^{-j_1}v_{0}(x, y)), y-\gamma_{n, \bfj}(v_0(x, y), \theta)
    )
        \chi_{+}(\theta)
        d\theta
    }_{
    L^p(\R^2)
    }\\
    & 
    \lesim
    2^{j_2
    -\delta_p j_2
    }
    \norm{f}_{L^p(\R^2)},
\end{split}
\end{equation}
and 
\begin{equation}\label{231203e4_146}
\begin{split}
&    \Norm{
\int    P_{\ge 0} f(x-2^{-j_2+ \alpha_n j_1}\theta-2^{\alpha_n j_1}\Theta_n^{\circ}(2^{-j_1}v_{0}(x, y)), y-\gamma_{n, \bfj}(v_0(x, y), \theta)
    )
        \chi_{+}(\theta)
        d\theta
    }_{
    L^p(\R^2)
    }\\
    & 
    \lesim
    2^{j_2
    -\delta_p j_2
    }
    \norm{f}_{L^p(\R^2)}.
\end{split}
\end{equation}
The low frequency part \eqref{231203e4_145} is more tricky than before, because of the $x$ term. First of all, by Taylor's expansion in the second variable, it is equivalent to 
\begin{equation}\label{231203e4_147}
\begin{split}
&    \Norm{
\int    P_{\le 0} f(x-2^{-j_2+ \alpha_n j_1}\theta-2^{\alpha_n j_1}\Theta_n^{\circ}(2^{-j_1}v_{0}(x, y)), y
    )
        \chi_{+}(\theta)
        d\theta
    }_{
    L^p(\R^2)
    }\\
    & 
    \lesim
    2^{j_2
    -\delta_p j_2
    }
    \norm{f}_{L^p(\R^2)}.
\end{split}
\end{equation}
Let $P'_{k}$ be a Littlewood-Paley projection operator in the first variable. The low frequency part $P'_{\le 0}$ is trivial. To prove \eqref{231203e4_147}, it suffices to prove 
\begin{equation}\label{231203e4_148}
\begin{split}
&    \Norm{
\int   P'_{\ge 0} P_{\le 0} f(x-2^{-j_2+ \alpha_n j_1}\theta-2^{\alpha_n j_1}\Theta_n^{\circ}(2^{-j_1}v_{0}(x, y)), y
    )
        \chi_{+}(\theta)
        d\theta
    }_{
    L^p(\R^2)
    }\\
    & 
    \lesim
    2^{j_2
    -\delta_p j_2
    }
    \norm{f}_{L^p(\R^2)}.
\end{split}
\end{equation}
We consider two different frequencies: 
\begin{equation}
P'_{\ge 0}= 
P'_{\ge j_2-\alpha_n j_1}+ P'_{\le  j_2-\alpha_n j_1}.
\end{equation}
For the low frequency part $P'_{\le  j_2-\alpha_n j_1}$, observe that the term $2^{-j_2+ \alpha_n j_1}\theta$ is negligible. For the high frequency part $P'_{\ge j_2-\alpha_n j_1}$, we will explore cancellations when  integrating in the $\theta$ variable. Let us be more precise. We will prove that 
\begin{equation}\label{231203e4_149}
\begin{split}
&    \Norm{
\int   P'_{\ge j_2-\alpha_n j_1} P_{\le 0} f(x-2^{-j_2+ \alpha_n j_1}\theta-2^{\alpha_n j_1}\Theta_n^{\circ}(2^{-j_1}v_{0}(x, y)), y
    )
        \chi_{+}(\theta)
        d\theta
    }_{
    L^p(\R^2)
    }\\
    & 
    \lesim
    2^{j_2
    -\delta_p j_2
    }
    \norm{f}_{L^p(\R^2)},
\end{split}
\end{equation}
and 
\begin{equation}\label{231203e4_150}
\begin{split}
&    \Norm{
\int   P'_{\le j_2-\alpha_n j_1} P_{\le 0} f(x-2^{-j_2+ \alpha_n j_1}\theta-2^{\alpha_n j_1}\Theta_n^{\circ}(2^{-j_1}v_{0}(x, y)), y
    )
        \chi_{+}(\theta)
        d\theta
    }_{
    L^p(\R^2)
    }\\
    & 
    \lesim
    2^{j_2
    -\delta_p j_2
    }
    \norm{f}_{L^p(\R^2)}.
\end{split}
\end{equation}
Consider \eqref{231203e4_150} first. It is equivalent to 
\begin{equation}\label{231203e4_151}
\begin{split}
&    \Norm{
\int   P'_{\le j_2-\alpha_n j_1} P_{\le 0} f(x-2^{\alpha_n j_1}\Theta_n^{\circ}(2^{-j_1}v_{0}(x, y)), y
    )
        \chi_{+}(\theta)
        d\theta
    }_{
    L^p(\R^2)
    }\\
    & 
    \lesim
    2^{j_2
    -\delta_p j_2
    }
    \norm{f}_{L^p(\R^2)},
\end{split}
\end{equation}
which holds for all $p>1$, and follows directly from Sobolev embedding. Consider \eqref{231203e4_149} next. We first apply Sobolev embedding in the $v$ variable, and then integrate in the $\theta$ variable for every fixed $v$. This finishes the discussion of \eqref{231203e4_145}.\\

From now on, we focus on the high frequency part \eqref{231203e4_146}. We apply a Littlewood-Paley decomposition in the first variable, and consider the contributions from $P'_{\ge 0}$ and $P'_{\le 0}$ separately. The lower frequency part $P'_{\le 0}$ is easier, and here we will only write down details for the high frequency part. More precisely, we will prove 
\begin{equation}\label{231203e4_152}
\begin{split}
&    \Norm{
\int    \bfP_{\ge 0} f(x-2^{-j_2+ \alpha_n j_1}\theta-2^{\alpha_n j_1}\Theta_n^{\circ}(2^{-j_1}v_{0}(x, y)), y-\gamma_{n, \bfj}(v_0(x, y), \theta)
    )
        \chi_{+}(\theta)
        d\theta
    }_{
    L^p(\R^2)
    }\\
    & 
    \lesim
    2^{j_2
    -\delta_p j_2
    }
    \norm{f}_{L^p(\R^2)},
\end{split}
\end{equation}
where 
\begin{equation}
\bfP_{\ge 0}:=P'_{\ge 0} P_{\ge 0}.
\end{equation}
Write 
\begin{equation}
\bfk=(k_1, k_2)\in \N^2, \ \ |\bfk|:=\max\{|k_1|, |k_2|\}.
\end{equation}
To prove \eqref{231203e4_152}, it suffices to prove 
\begin{equation}\label{231203e4_155}
\begin{split}
&    \Norm{
\int    \bfP_{\bfk} f(x-2^{-j_2+ \alpha_n j_1}\theta-2^{\alpha_n j_1}\Theta_n^{\circ}(2^{-j_1}v_{0}(x, y)), y-\gamma_{n, \bfj}(v_0(x, y), \theta)
    )
        \chi_{+}(\theta)
        d\theta
    }_{
    L^p(\R^2)
    }\\
    & 
    \lesim
    2^{
    -\delta_p |\bfk|
    }
    2^{j_2
    -\delta_p j_2
    }
    \norm{f}_{L^p(\R^2)}.
\end{split}
\end{equation}
By Sobolev embedding, it suffices to prove 
\begin{equation}\label{231203e4_156}
\begin{split}
&    \Norm{
\int    \bfP_{\bfk} f(x-2^{-j_2+ \alpha_n j_1}\theta, y-\gamma_{n, \bfj}(v, \theta)
    )
        \chi_{+}(\theta)
        d\theta
    }_{
    L^p_{x, y, v}
    }\\
    & 
    \lesim
    2^{
    -\delta_p |\bfk|-
    \frac{|\bfk|}{p}
    }
    2^{j_2
    -\delta_p j_2
    }
    \norm{f}_{L^p(\R^2)}.
\end{split}
\end{equation}
Here when integrating in the $x, y, v$ variables, we implicitly applied a change of variables in $x$ and $v$, and removed the term
\begin{equation}
2^{\alpha_n j_1}\Theta_n^{\circ}(2^{-j_1}v_{0}(x, y)).
\end{equation}
  Recall that 
\begin{equation}
\gamma_{n, \bfj}(v, \theta):=
2^{
\mf{q}_i
\alpha_n j_1
}
\gamma_n(2^{-j_1}v, 2^{-j_2}\theta),
\end{equation}
and 
\begin{equation}
\gamma_n(v, \theta):=
\gamma(v, \theta+ 
    \Theta_n^{\circ}(v)
    ).
\end{equation}
Moreover, recall that 
\begin{equation}
|\cine(\gamma_n)|\simeq |v|^{a_n}|\theta|^{b_n}.
\end{equation}
Note that 
\begin{equation}
|
\partial_{\theta}^2\gamma_{n, \bfj}(v, \theta)
|\simeq 
2^{
-2j_2+2\alpha_n j_1
}.
\end{equation}
Therefore, by Van der Corput's lemma and interpolation with trivial $L^{\infty}$ bounds, we obtain
\begin{equation}\label{231203e4_161}
\begin{split}
&    \Norm{
\int    \bfP_{\bfk} f(x-2^{-j_2+ \alpha_n j_1}\theta, y-\gamma_{n, \bfj}(v, \theta)
    )
        \chi_{+}(\theta)
        d\theta
    }_{
    L^p_{x, y, v}
    }\\
    & 
    \lesim
    2^{-
    \frac{k_2}{p}
    }
    2^{\frac{2j_2}{p}
    }
    \norm{f}_{L^p(\R^2)}.
\end{split}
\end{equation}
By balancing $k_1$ and $k_2$, this proves \eqref{231203e4_156} for all $p>2$ and $|\bfk|\le C_{p, \gamma} j_2$, where $C_{p, \gamma}$ is a large constant depending on $p$ and $\gamma$. Therefore, we only need to consider $|\bfk|\ge C_{p, \gamma} j_2$, which follows directly from Claim \ref{231130claim4_4}. This finishes the discussion on the case \eqref{231203e4_146}.

\subsubsection{The first subcase}

In this subsection we consider the case when there exists a vertex-tangent line $\mf{L}\in 
\mf{L}_{
(\mf{p}_i, \mf{q}_i)
}
$
 passing through $(\mf{p}_0, 0)$. Recall from \eqref{240104e4_91} that we need to prove 
\begin{equation}\label{240105e4_196}
\Norm{
\sup_{|v|\le \epsilon}
\anorm{
\int_{
\R
}
f(x-\theta, y-\gamma(v, \theta)) a_i(v, \theta)d\theta
}
}_{L^p(\R^2)}
\lesim_{
\gamma, p, a
}
\norm{f}_{L^p(\R^2)},
\end{equation}
for all 
\begin{equation}
p> \max\{
2, \mf{D}_{\gamma}
\}.
\end{equation}
In particular, in this subsection we will see the exponent $\mf{D}_{\gamma}$.  It suffices to prove that 
\begin{equation}\label{231201e4_130}
    \Norm{
    \sup_{|v|\le \epsilon} 
    \anorm{
    \int_{\R} f(x-\theta, y-\gamma(v(x, y), \theta))
    a_i(v, \theta)
    \chi_{j_2}(\theta)d\theta
    }
    }_{
    L^p(\R^2)
    }
    \lesim_{\gamma, \chi, p} 
    2^{-\delta_p j_2}
    \norm{f}_{L^p(\R^2)}.
\end{equation}
We consider two cases separately: The first case is when 
\begin{equation}\label{231201e4_131}
2^{-
\mf{q}_i
 j_2}\ge 2^{-
\mf{p}_0
\mathfrak{j}_1(x, y)
},
\end{equation}
and the second case is 
\begin{equation}\label{231201e4_132}
2^{-\mf{q}_i j_2}\le 2^{-
\mf{p}_0
\mathfrak{j}_1(x, y)
}.
\end{equation}
We start with the easier case \eqref{231201e4_131}. In this case, we still have that 
\begin{equation}
|\gamma(v, \theta)|\simeq |\theta|^{\mf{q}_i},
\end{equation}
exactly the same as what we had in Subsection \ref{240106subsection4_2_1}. Indeed, one can also repeat exactly the same proof there and prove \eqref{231201e4_130}, for all $p>2$. We leave out the details. \\

We next discuss the case \eqref{231201e4_132}. This case is more interesting, and we will see the  exponent $p>\mf{D}_{\gamma}$.  By the triangle inequality, it suffices to prove that
\begin{multline}\label{231201e4_135}
    \Norm{
    \int_{\R} f(x-\theta, y-\gamma(v_{j_1}(x, y), \theta))
    a_i(v_{j_1}(x, y), \theta)
    \chi_{j_2}(\theta)d\theta
    }_{
    L^p(\R^2)
    }\\
    \lesim_{\gamma, \chi, p} 
    2^{-\delta_p j_2}
    \norm{f}_{L^p(\R^2)},
\end{multline}  
for every
\begin{equation}
j_1\le 
\mf{q}_i
j_2/\mf{p}_0.
\end{equation}
Recall the notation $\gamma_{\mf{r}\mf{e}}$ in \eqref{240106e4_3}. Moreover, recall that on the support of $a_i(v, \theta)$, we have 
\begin{equation}
|\gamma_{\mf{r}\mf{e}}(v, \theta)|\simeq |\theta|^{\mf{q}_i}.
\end{equation} 
When proving \eqref{231201e4_135}, we can always without loss of generality assume that 
\begin{equation}\label{240106e4_206}
\supp(\chi_{j_2})\subset 
\supp(
a_i(v, \cdot)
), 
\end{equation}
for all $|v|\simeq 2^{-j_1}$, as otherwise there is nothing to prove.  We therefore from now on always assume \eqref{240106e4_206}, and \eqref{231201e4_135} becomes 
\begin{equation}\label{240106e4_207}
    \Norm{
    \int_{\R} f(x-\theta, y-\gamma(v_{j_1}(x, y), \theta))
    \chi_{j_2}(\theta)d\theta
    }_{
    L^p(\R^2)
    }
    \lesim_{\gamma, \chi, p} 
    2^{-\delta_p j_2}
    \norm{f}_{L^p(\R^2)}.
\end{equation}  
 By changing variable in $\theta$ and renaming $v$ variables, it is equivalent to prove 
\begin{equation}
    \Norm{
    \sup_{|v|\simeq 1} 
    \anorm{
    \int_{\R} f(x-\theta, y-\gamma_{\bfj}(v, \theta))\chi_{+}(\theta)d\theta
    }
    }_{
    L^p(\R^2)
    }
    \lesim_{\gamma, \chi, p} 
    2^{j_2-\delta_p j_2}
    \norm{f}_{L^p(\R^2)},
\end{equation}
where 
\begin{equation}
\gamma_{\bfj}(v, \theta):=
2^{
\mf{p}_0
j_1
}
\gamma(
2^{-j_1}v, 2^{-j_2}\theta
).
\end{equation}
By applying a  Littlewood-Paley decomposition in the second variable, and by the triangle inequality, 
it suffices to prove that 
\begin{equation}\label{231201e4_139}
    \Norm{
    \sup_{|v|\simeq 1} 
    \anorm{
    \int_{\R} 
    P_k f(x-\theta, y-\gamma_{\bfj}(v, \theta))\chi_{+}(\theta)d\theta
    }
    }_{
    L^p(\R^2)
    }
    \lesim_{\gamma, \chi, p} 
    2^{
    -\delta_p k
    }
    2^{j_2-\delta_p j_2}
    \norm{f}_{L^p(\R^2)},
\end{equation}
for every $k\in \N$. By Sobolev embedding, Van der Corput's lemma and interpolation with trivial $L^{\infty}$ bounds, we can bound the left hand side by 
\begin{equation}\label{240106e4_212}
2^{
\frac{k}{p}}
2^{-
\frac{
k+
\mf{p}_0
j_1-
\mf{q}_i
j_2
}{p}
}\norm{f}_p. 
\end{equation}
Consider the neighboring vertex $(\mf{p}_{i+1}, \mf{q}_{i+1})$ of the vertex $(0, \mf{q}_i)$ in the reduced Newton diagram $\mc{R}\mc{N}_d(\gamma)$. The relation of $j_1$ and $j_2$ is given by 
\begin{equation}\label{240106e4_213}
2^{
-j_2\mf{q}_i
}
\ge 2^{
-j_1 \mf{p}_{i+1}
-j_2
\mf{q}_{i+1}
}\implies 
(\mf{q}_i-\mf{q}_{i+1}) j_2\le \mf{p}_{i+1} j_1.
\end{equation}
We compare the constant in \eqref{240106e4_212} with that on the right hand side of \eqref{231201e4_139}, and let 
\begin{equation}
\frac{k}{p}-\frac{
k+\mf{p}_0 j_1-\mf{q}_i j_2
}{p}
\le j_2.
\end{equation}
This gives us the constraint 
\begin{equation}
p\ge 
\mf{q}_i-
\mf{p}_0 \frac{j_1}{j_2},
\end{equation}
for all $j_1, j_2$ satisfying \eqref{240106e4_212}, which further leads to the constraint 
\begin{equation}\label{240106e4_216}
p\ge 
\mf{q}_i-
\mf{p}_0 
\frac{
\mf{p}_{i+1}
}{
\mf{q}_i-\mf{q}_{i+1}
}.
\end{equation}
Note that the right hand side of \eqref{240106e4_216} is exactly the vertical distance between the vertex $(\mf{p}_0, 0)$ and the line connecting the vertices $(0, \mf{p}_i)$ and $(\mf{p}_{i+1}, \mf{q}_{i+1})$. This finishes the proof of \eqref{231201e4_139} with $\delta_p=0$ for all 
\begin{equation}\label{240106e4_217}
p> \max\{2, \mf{D}_{\gamma}\}.
\end{equation}
To prove \eqref{231201e4_139} for all $p$ satisfying \eqref{240106e4_217} and some $\delta_p>0$, by interpolation, we only need to prove it for sufficiently large $p$, and this can be obtained by repeating exactly the same local smoothing estimates in Subsection \ref{240106subsection4_2_1} (for instance Claim \ref{231130claim4_4}).

\subsection{
The dominating polynomial \texorpdfstring{$v^{\mf{p}_i} \theta^{\mf{q}_i}$}{} satisfies 
\texorpdfstring{$\mf{p}_i\ge 1, \mf{q}_i=1$}{}
and \texorpdfstring{$i=1$}{}
}\label{240108subsection4_3}

In this subsection we consider the case where 
\begin{equation}\label{240108e4_219}
    \gamma(v, \theta)=
    c_{
    \mf{p}_i, 1
    } v^{\mf{p}_i} \theta+ o(|v|^{\mf{p}_i} |\theta|)+ h_1(v),\ c_{
    \mf{p}_i, 1
    }\neq 0, \ \mf{p}_i\ge 1,
\end{equation}
for all $(v, \theta)$ in the support of $a_i$, 
and $h_1$ collections all the terms in the Taylor expansion of $\gamma$ that depend purely on $v$. The extra assumption that $i=1$ carries a lot of information. Recall how we labelled vertices of the reduced Newton diagram $\rn_d(\gamma)$ in \eqref{240108e4_5}. That $i=1$, combined with the fact that $\mf{q}_i=1$, implies that $v^{\mf{p}_i}\theta$ is the only vertex in the reduced Newton diagram $\rn_d(\gamma)$. As a consequence, we see that \eqref{240108e4_219} holds for all $(v, \theta)$. 

Recall the definition of strongly degenerate functions in Definition \ref{231122defi1_5}. In particular, note that $\gamma$ satisfies the assumption of item (2) in Definition \ref{231122defi1_5}, which says that the reduced Newton diagram $\rn_d(\gamma)$ consists of only one vertex $(\mf{p}, 1)$ for some $\mf{p}\ge 1$. 
As we are assuming that $\gamma$ is not strongly degenerate, we can therefore write 
\begin{equation}\label{240125e4_222}
\gamma(v, \theta)
=
v^{\mf{p}_i} h_2(\theta)+ O(|v|^{\mf{p}_i+1} |\theta|)+ h_1(v),
\end{equation}
where 
\begin{equation}
h_2(\theta)=\theta+ o(\theta),
\end{equation}
and the cinematic curvature of $v^{\mf{p}_i} h_2(\theta)$ does not vanish constantly. Moreover, by item (3) in Definition \ref{231122defi1_5}, we obtain that 
\begin{equation}
h_1(v)= O(v^{\mf{p}_i}).
\end{equation}
Recall from \eqref{240102e4_24} that we need to prove 
\begin{equation}\label{240106e4_221}
\Norm{
\sup_{|v|\le \epsilon}
\anorm{
\int_{
\R
}
f(x-\theta, y-\gamma(v, \theta)) a_i(v, \theta)d\theta
}
}_{L^p(\R^2)}
\lesim_{
\gamma, p, a
}
\norm{f}_{L^p(\R^2)},
\end{equation}
for all 
\begin{equation}\label{240108e4_224}
p> \max\{
2, \mf{D}_{\gamma}
\}=2.
\end{equation}
In \eqref{240108e4_224}, we used the fact that $\mf{D}_{\gamma}\le 1$ in the current case. We follow the first few steps of Subsection \ref{240103subsubsection4_1_1} (or similarly Subsection \ref{240104subsubsection4_1_2}), and need to prove 
\begin{equation}\label{240108e4_242kkk}
\begin{split}
& \Norm{
\sum_{k\in \Z}
\int_{
\R
}
P_k f(x-\theta, y-\gamma(v(x, y), \theta))
\chi_{j_2}(\theta)
d\theta
}_{L^p(\R^2)}\lesim
2^{-\delta_p j_2}
\norm{f}_{L^p(\R^2)},
\end{split}
\end{equation}
for all $p>2, j_2\in \N$ and some $\delta_p>0$. In \eqref{240108e4_242kkk} we do not have the term $a_i(v(x, y), \theta)$ in \eqref{240106e4_221} because there is only one vertex in the reduced Newton diagram $\rn_d(\gamma)$. When proving \eqref{240108e4_242kkk}, the low frequency part $k\le 
\mf{p}_i \mf{j}_1(x, y)
$ can be controlled by the strong maximal operator. We only need to prove 
\begin{equation}\label{240108e4_242kkkk}
\begin{split}
& \Norm{
\sum_{k\ge \mf{p}_i
\mf{j}_1(x, y)
}
\int_{
\R
}
P_k f(x-\theta, y-\gamma(v(x, y), \theta))
\chi_{j_2}(\theta)
d\theta
}_{L^p(\R^2)}\lesim
2^{-\delta_p j_2}
\norm{f}_{L^p(\R^2)}.
\end{split}
\end{equation}
By Littlewood-Paley inequalities, it suffices to prove
\begin{equation}\label{240111e4_227jj}
\begin{split}
& \Norm{
\int_{
\R
}
P_k f(x-\theta, y-\widetilde{\gamma}(v_0(x, y), \theta))
\chi_{j_2}(\theta)
d\theta
}_{L^p(\R^2)}\lesim
2^{-\delta_p j_2
-\delta_p k
}
\norm{f}_{L^p(\R^2)},
\end{split}
\end{equation} 
for all $k\in \N$, 
where 
\begin{equation}
\widetilde{\gamma}(v, \theta):=
2^{\mf{p}_i j_1}\gamma(2^{-j_1}v, \theta).
\end{equation}
Here we find it much more convenient not to do the change of variable $\theta\mapsto 2^{-j_2}\theta$. 
\begin{claim}\label{240111claim4_4}
For every $j_2\in \N, k\in \N$, it holds that 
\begin{equation}\label{240111e4_255}
\begin{split}
& \Norm{
\int_{
\R
}
P_{
j_2+k
} f(x-\theta, y-\widetilde{\gamma}(v_{0}(x, y), \theta))
\chi_{j_2}(\theta)
d\theta
}_{L^2(\R^2)}\lesim_{\delta}
2^{\delta j_2+\delta k}
\norm{f}_{L^2(\R^2)},
\end{split}
\end{equation}
for every $\delta>0$. 
\end{claim}
\begin{proof}[Proof of Claim \ref{240111claim4_4}]
Recall that  $\chi: \R\to \R$ is a smooth bump function supported on $(-\epsilon, \epsilon)$. By the triangle inequality, we bound the left hand side of \eqref{240111e4_255} by 
\begin{equation}
\begin{split}
&
\Norm{
\int_{
\R
}
|
P_{
j_2+k
} f(x-\theta, y-\widetilde{\gamma}(v_{0}(x, y), \theta))
\chi_{j_2}(\theta)
|
d\theta
}_{L^2(\R^2)}\\
& 
\le 
\Norm{
\int_{
\R
}
|
P_{
j_2+k
} f(x-\theta, y-\widetilde{\gamma}(v_{0}(x, y), \theta))
\chi(\theta)
|
d\theta
}_{L^2(\R^2)}.
\end{split}
\end{equation}
Denote 
\begin{equation}
\Gamma(\theta; y, v):=
y-\widetilde{\gamma}(v, \theta).
\end{equation}
If we can show that 
\begin{equation}\label{240111e4_231}
\det
\begin{bmatrix}
\partial_y \Gamma, & \partial_v \Gamma\\
\partial_y\partial_{\theta} \Gamma, & \partial_v \partial_{\theta} \Gamma
\end{bmatrix}
\end{equation}
never vanishes, then Claim \ref{240111claim4_4} follows directly from the (one-parameter case of the)  main theorem in Zahl \cite{Zah23}. Direct computations show that 
\begin{equation}\label{240111e4_232}
\eqref{240111e4_231}=
|\partial_v \partial_{\theta}\Gamma|=
|\partial_v \partial_{\theta}\gamma_{\bfj}|
\simeq 1,
\end{equation}
and this finishes the proof of Claim \ref{240111claim4_4}. 
\end{proof}

Let $p>2$. Let $C_{p, \gamma}$ be a large constant depending on $p$ and $\gamma$ and its choice will become clear later.  By interpolating \eqref{240111e4_255} with trivial $L^{\infty}$ bounds, we can prove \eqref{240111e4_227jj} for all 
\begin{equation}
k\le (C_{p, \gamma})^8 j_2.
\end{equation}
Therefore it remains to prove \eqref{240111e4_227jj} for all 
\begin{equation}\label{240111e4_235ggg}
k\ge (C_{p, \gamma})^8 j_2.
\end{equation}
By the change of variable $\theta\mapsto 2^{-j_2}\theta$, it is equivalent to prove 
\begin{equation}\label{240111e4_227jjj}
\begin{split}
& \Norm{
\int_{
\R
}
P_k f(x-\theta, y-\gamma_{\bfj}(v_0(x, y), \theta))
\chi_{+}(\theta)
d\theta
}_{L^p(\R^2)}\lesim
2^{j_2-\delta_p j_2
-\delta_p k
}
\norm{f}_{L^p(\R^2)},
\end{split}
\end{equation} 
where 
\begin{equation}
\gamma_{\bfj}(v, \theta):=
2^{\mf{p}_i j_1}
\gamma(
2^{-j_1}v, 2^{-j_2}\theta
).
\end{equation}
We consider the case
\begin{equation}\label{231128e4_153}
j_1\le (C_{p, \gamma})^4 j_2,
\end{equation}
and the case 
\begin{equation}\label{231128e4_154}
j_1\ge (C_{p, \gamma})^4 j_2.
\end{equation}
Start with the case \eqref{231128e4_154}. Let us compute cinematic curvature in this case. Write 
\begin{equation}
\gamma(v, \theta)=v^{\mf{p}_i} h_2(\theta)+ R(v, \theta)+ h_1(v),
\end{equation}
where 
\begin{equation}
R(v, \theta)= O(v^{\mf{p}_i+1}\theta).
\end{equation}
Under this notation, 
\begin{equation}
\gamma_{\bfj}(v, \theta)=
v^{\mf{p}_i} h_2(2^{-j_2}\theta)+ 2^{\mf{p}_i j_1} R(2^{-j_1}v, 2^{-j_2}\theta)
+
2^{
\mf{p}_i j_1
}h_1(2^{-j_1}v)
.
\end{equation}
The cinematic curvature of $\gamma_{\bfj}$ is 
\begin{equation}
\det
\begin{bmatrix}
(\star)_{11}, & 
(\star)_{12}\\
(\star)_{21}, & 
(\star)_{22}
\end{bmatrix}
\end{equation}
where 
\begin{equation}
\begin{split}
& (\star)_{11}:= 2^{-2j_2} v^{
\mf{p}_i
} h''_2(2^{-j_2}\theta)+ 2^{-2j_2}2^{\mf{p}_i j_1} \partial^2_{\theta} R(2^{-j_1}v, 2^{-j_2}\theta),\\
& 
(\star)_{12}:=2^{-3j_2} v^{\mf{p}_i} h'''_2(2^{-j_2}\theta)+ 2^{-3j_2}2^{
\mf{p}_i
j_1} \partial^3_{\theta} R(2^{-j_1}v, 2^{-j_2}\theta),\\ 
& (\star)_{21}:= 
\mf{p}_i v^{\mf{p}_i-1} 2^{-j_2}h'_2(2^{-j_2}\theta)+ 2^{-j_1-j_2} 2^{
\mf{p}_i
j_1} \partial_v \partial_{\theta}R(2^{-j_1}v, 2^{-j_2}\theta),\\
& (\star)_{22}:=
\mf{p}_i v^{\mf{p}_i-1} 2^{-2j_2} h''_2(2^{-j_2}\theta)+ 2^{-j_1-2j_2} 2^{\mf{p}_i j_1} \partial_v \partial^2_{\theta}R(2^{-j_1}v, 2^{-j_2}\theta).
\end{split}
\end{equation}
This is equal to 
\begin{equation}
2^{-4j_2}
\det
\begin{bmatrix}
(\star)'_{11}, & 
(\star)'_{12}\\
(\star)'_{21}, & 
(\star)'_{22}
\end{bmatrix}
\end{equation}
where 
\begin{equation}
\begin{split}
& (\star)'_{11}:= v^{\mf{p}_i} h''_2(2^{-j_2}\theta)+ 2^{\mf{p}_i j_1} \partial^2_{\theta} R(2^{-j_1}v, 2^{-j_2}\theta),\\
& 
(\star)'_{12}:=v^{\mf{p}_i} h'''_2(2^{-j_2}\theta)+ 2^{\mf{p}_i j_1} \partial^3_{\theta} R(2^{-j_1}v, 2^{-j_2}\theta),\\ 
& (\star)'_{21}:= 
\mf{p}_i v^{\mf{p}_i-1} h'_2(2^{-j_2}\theta)+ 2^{-j_1} 2^{\mf{p}_i j_1} \partial_v \partial_{\theta}R(2^{-j_1}v, 2^{-j_2}\theta),\\
& (\star)'_{22}:=
\mf{p}_i v^{\mf{p}_i-1} h''_2(2^{-j_2}\theta)+ 2^{-j_1} 2^{\mf{p}_i j_1} \partial_v \partial^2_{\theta}R(2^{-j_1}v, 2^{-j_2}\theta).
\end{split}
\end{equation}
By the assumption that the cinematic curvature of $v^{\mf{p}_i} h_2(\theta)$ does not vanish constantly and by the relation \eqref{231128e4_154}, we obtain that 
\begin{equation}\label{240107e4_237}
|\cine(\gamma_{\bfj})|\gtrsim 
2^{- (C_{p, \gamma})^2 j_2}.
\end{equation}
The desired estimate \eqref{240111e4_227jjj} now follows from \eqref{240111e4_235ggg} and Claim \ref{231130claim4_4}. This finishes the discussion on the case \eqref{231128e4_154}. 

Let us remark here that it is crucial that the right hand side of \eqref{240107e4_237} does not tend to $0$ as $j_1$ tends to $\infty$ (unlike, say the cinematic curvature bound in \eqref{240107e4_133jj}). The constant on the right hand side of \eqref{240111e4_227jjj} does not involve $j_1$, and therefore Claim \ref{240111claim4_4} can only handle the case $k\lesim_{p, \gamma} j_2$. In other words, we need that the local smoothing estimates in Claim \ref{231130claim4_4} handles all ``large" $j_1$ uniformly, and this is why it is crucial that  the bound in \eqref{240107e4_237} must be independent of $j_1$. \footnote{This also explains why we need to include item (2) in the definition of strongly degenerate functions in Definition \ref{231122defi1_5}.}\\

Next we discuss the case \eqref{231128e4_153}. In this case, because of the constraint \eqref{231128e4_153} and because of the relation \eqref{240111e4_235ggg}, we have that 
\begin{equation}\label{240122e4_248}
k\ge (C_{p, \gamma})^4 \max\{j_1, j_2\}.
\end{equation}
By repeating the whole argument in Subsection \ref{240107subsection4_2}, we will be able to prove \eqref{240111e4_227jjj}. The details are left out.

\subsection{
The dominating polynomial 
\texorpdfstring{$v^{\mf{p}_i} \theta^{\mf{q}_i}$}{}
satisfies 
\texorpdfstring{$\mf{p}_i\ge 1, \mf{q}_i=1$}{}
and 
\texorpdfstring{$i\ge 2$}{}
}\label{240118subsection4_4}

Recall the labelling in \eqref{240108e4_5}. The assumption that $i\ge 2$ implies that the vertex $(\mf{p}_{i-1}, \mf{q}_{i-1})$ is also a vertex in the reduced Newton diagram $\rn_d(\gamma)$. Recall from \eqref{240102e4_24} that we need to prove 
\begin{equation}\label{240108e4_255}
\Norm{
\sup_{|v|\le \epsilon}
\anorm{
\int_{
\R
}
f(x-\theta, y-\gamma(v, \theta)) a_i(v, \theta)d\theta
}
}_{L^p(\R^2)}
\lesim_{
\gamma, p, a
}
\norm{f}_{L^p(\R^2)},
\end{equation}
for all 
\begin{equation}\label{240108e4_224hhh}
p> \max\{
2, \mf{D}_{\gamma}
\}.
\end{equation}
Because of the existence of the vertex $(\mf{p}_{i-1}, \mf{q}_{i-1})$, the equality in \eqref{240108e4_224} may not hold anymore. \\

We consider two subcases.  Let $\mfp_0$ be the smallest $\mfp$ such that $c_{\mfp, 0}\neq 0$. If such $\mfp$ does not exist, that is, $c_{\mfp, 0}=0$ for all $\mfp$, then we write $\mfp_0=\infty$. The first subcase is when $\mf{p}_0\ge \mf{p}_i$, and the second subcase is when $\mf{p}_0< \mf{p}_i$. 

The case distinction here is quite different from previous ones, say the one in Subsection \ref{240107subsection4_2}. The main reason for the differences comes from the two different methods of proving $L^2$ estimates: The method of proving $L^2$ estimates in Subsection \ref{240107subsection4_2} is via Van der Corput's lemma, see for instance \eqref{231130e4_102zz}, while the method we will use in this Subsection is from planar maximal Kakeya inequalities, similar to Claim \ref{240111claim4_4}.

\subsubsection{The first subcase}

We first consider the case $\mf{p}_0\ge \mf{p}_i$.  We will see that the proof of \eqref{240108e4_255} is very similar to the proof in Subsection \ref{240108subsection4_3}, more precisely, the proof of \eqref{240106e4_221} under the assumption \eqref{231128e4_153}.

%
First of all, let us point out that under the assumption  that $\mf{p}_0\ge \mf{p}_i$, we have
\begin{equation}
\max\{2, \mf{D}_{\gamma}\}=2,
\end{equation}
precisely the same as in \eqref{240108e4_224}. In other words,  we will prove the estimate \eqref{240108e4_255} for all $p>2$, which is exactly the same as the range \eqref{240108e4_224} for the estimate \eqref{240106e4_221}. Secondly, as pointed out at the beginning of this section, the vertex $(\mf{p}_{i-1}, \mf{q}_{i-1})$ is also a vertex in the reduced Newton diagram $\mc{R}\mc{N}_d(\gamma)$. Take $(v, \theta)\in \supp(a_i)$, and let $j_1, j_2\in \N$ be such that 
\begin{equation}\label{240122e4_252}
|v|\simeq 2^{-j_1}, \ \ |\theta|\simeq 2^{-j_2},
\end{equation}
then we must have 
\begin{equation}\label{240122e4_253}
2^{
-
\mf{p}_{i-1} j_1-
\mf{q}_{i-1} j_2
}
\le 2^{
-\mf{p}_i j_1-
\mf{q}_i j_2
},
\end{equation}
which further implies that 
\begin{equation}\label{240118e4_253}
j_1
\le 
C_{\gamma} j_2,
\end{equation}
for some constant $C_{\gamma}$ depending only on $\gamma$; the relation \eqref{240118e4_253} is the same as the relation \eqref{231128e4_153}. \footnote{The constant in \eqref{231128e4_153} is allowed to depend on $p$, but this is a minor difference. } 

Now we very briefly sketch how to prove \eqref{240108e4_255} for all $p>2$. Similarly to \eqref{240108e4_242kkk}, it suffices to prove 
\begin{multline}\label{240118e4_254}
\Norm{
\sum_{k\in \Z}
\int_{
\R
}
P_k f(x-\theta, y-\gamma(v(x, y), \theta))
a_i(v(x, y), \theta)
\chi_{j_2}(\theta)
d\theta
}_{L^p(\R^2)}\\
\lesim
2^{-\delta_p j_2}
\norm{f}_{L^p(\R^2)},
\end{multline}
for all $p>2, j_2\in \N$ and some $\delta_p>0$. The low frequency part $k\le \mf{p}_i \mf{j}_1(x, y)$ can be controlled by the strong maximal operator. The high frequency part $k\ge \mf{p}_i \mf{j}_1(x, y)$ can be proven in exactly the same way as how \eqref{240108e4_242kkkk} was proven.

\subsubsection{The second subcase}

We consider the case $\mf{p}_0< \mf{p}_i$. What we need to prove is still \eqref{240118e4_254}, that is, 
\begin{multline}\label{240118e4_256}
\Norm{
\sum_{k\in \Z}
\int_{
\R
}
P_k f(x-\theta, y-\gamma(v(x, y), \theta))
a_i(v(x, y), \theta)
\chi_{j_2}(\theta)
d\theta
}_{L^p(\R^2)}\\
\lesim
2^{-\delta_p j_2}
\norm{f}_{L^p(\R^2)},
\end{multline}
for all $j_2\in \N$, some $\delta_p>0$, and a different range 
\begin{equation}
p> \max\{2, \mf{D}_{\gamma}\}.
\end{equation}
Recall the discussions in \eqref{240122e4_252} and \eqref{240122e4_253}. Recall the notation that 
\begin{equation}
|v(x, y)|\simeq 2^{
-\mf{j}_1(x, y)
}. 
\end{equation}
Let $j_1\in \N$ be a possible value of the function $\mf{j}_1(x, y)$. Then we can without loss of generality assume that \eqref{240122e4_253} always holds, that is, 
\begin{equation}\label{240118e4_259}
2^{
-
\mf{p}_{i-1} j_1-
\mf{q}_{i-1} j_2
}
\le 2^{
-\mf{p}_i j_1-
\mf{q}_i j_2
},
\end{equation}
as otherwise the integral on the left hand side of \eqref{240118e4_256} vanishes. By the triangle inequality applied in $j_1$, to prove \eqref{240118e4_256}, it suffices to prove that 
\begin{equation}\label{240118e4_260}
\begin{split}
& \Norm{
\sum_{k\in \Z}
\int_{
\R
}
P_k f(x-\theta, y-\gamma(v_{j_1}(x, y), \theta))
a_i(v_{j_1}(x, y), \theta)
\chi_{j_2}(\theta)
d\theta
}_{L^p(\R^2)}\\
& \lesim
2^{-\delta_p j_2}
\norm{f}_{L^p(\R^2)},
\end{split}
\end{equation}
for all $j_1, j_2$ satisfying \eqref{240118e4_259}. To simplify our presentation, we without loss of generality assume that 
\begin{equation}
\supp(\chi_{j_2})\subset 
\supp(
a_i(v, \cdot)
),
\end{equation}
for every $|v|\simeq 2^{-j_1}$; this can be achieved by cutting the support of $\chi_{j_2}$ into several smaller pieces. We write \eqref{240118e4_262} in the simpler form
\begin{equation}\label{240118e4_262}
\begin{split}
& \Norm{
\sum_{k\in \Z}
\int_{
\R
}
P_k f(x-\theta, y-\gamma(v_{j_1}(x, y), \theta))
\chi_{j_2}(\theta)
d\theta
}_{L^p(\R^2)} \lesim
2^{-\delta_p j_2}
\norm{f}_{L^p(\R^2)}.
\end{split}
\end{equation}
By the change of variable 
\begin{equation}
\theta\mapsto 2^{-j_2}\theta,
\end{equation}
and renaming the variable $v\to 2^{-j_1}v$, it suffices to prove 
\begin{equation}\label{240118e4_264}
\begin{split}
& \Norm{
\sum_{k\in \Z}
\int_{
\R
}
P_k f(x-\theta, y-\gamma_{\bfj}(v_{0}(x, y), \theta))
\chi_{+}(\theta)
d\theta
}_{L^p(\R^2)} \lesim
2^{j_2-\delta_p j_2}
\norm{f}_{L^p(\R^2)},
\end{split}
\end{equation}
where 
\begin{equation}
\gamma_{\bfj}(v, \theta):=
2^{
\mf{p}_0 j_1
}
\gamma(
2^{-j_1}v, 2^{-j_2}\theta
), \ \ \bfj=(j_1, j_2).
\end{equation}
The low frequency part $k\le 0$ can be controlled by the strong maximal function, and therefore it suffices to prove that 
\begin{equation}\label{240118e4_266}
\begin{split}
& \Norm{
\int_{
\R
}
P_k f(x-\theta, y-\gamma_{\bfj}(v_{0}(x, y), \theta))
\chi_{+}(\theta)
d\theta
}_{L^p(\R^2)} \lesim
2^{j_2-\delta_p j_2
-\delta_p k
}
\norm{f}_{L^p(\R^2)},
\end{split}
\end{equation}
for all $k\in \N$. Here comes a new phenomenon. Note that 
\begin{equation}
\gamma_{\bfj}(v, \theta)=
c_{\mf{p}_0, 0} v^{\mf{p}_0}+ 
c_{
\mf{p}_i, 1
}
2^{
\mf{p}_0 j_1- \mf{p}_i j_1-
j_2
} v^{\mf{p}_i} \theta+ \text{smaller term},
\end{equation}
which suggests that we consider the case
\begin{equation}\label{240118e4_268z}
0\le k\le 
\mf{p}_i j_1+j_2-\mf{p}_0 j_1
\end{equation}
and the case 
\begin{equation}\label{240118e4_269z}
k\ge \mf{p}_i j_1+j_2-\mf{p}_0 j_1
\end{equation}
separately. \\

Consider the former case \eqref{240118e4_268z}. In this case, we apply Sobolev embedding in the $v$ variable, and bound the left hand side of \eqref{240118e4_266} by 
\begin{equation}
2^{\frac{k}{p}} \|f\|_p\lesim 
2^{\frac{\mf{p}_i j_1+j_2-\mf{p}_0 j_1}{p}}\|f\|_p.
\end{equation}
In order for \eqref{240118e4_266} to hold, we know that $p$ must satisfy 
\begin{equation}
\frac{\mf{p}_i j_1+j_2-\mf{p}_0 j_1}{p}\le j_2 \implies p\ge 
(\mf{p}_i-\mf{p}_0)\frac{j_1}{j_2}+1.
\end{equation}
Recall the relation in \eqref{240118e4_259} that 
\begin{equation}
\frac{j_1}{j_2}\le 
\frac{
\mf{q}_{i-1}-\mf{q}_i
}{
\mf{p}_i-\mf{p}_{i-1}
}.
\end{equation}
All these together imply that $p$ must satisfy 
\begin{equation}
p\ge 
(\mf{p}_i-\mf{p}_0)\frac{
\mf{q}_{i-1}-\mf{q}_i
}{
\mf{p}_i-\mf{p}_{i-1}
}
+1,
\end{equation}
which is exactly the vertical distance between $(\mf{p}_0, 0)$ and the line connecting the vertices $(\mf{p}_{i-1}, \mf{q}_{i-1})$ and $(\mf{p}_{i}, \mf{q}_{i})$. This proves \eqref{240118e4_266} for all 
\begin{equation}\label{240125e4_276}
p\ge \max\{2, \mf{D}_{\gamma}\},
\end{equation}
with $\delta_p=0$. Moreover, by picking $p$ to be strictly bigger than $\max\{2, \mf{D}_{\gamma}\}$, it is elementary to see that  \eqref{240118e4_266} holds for some $\delta_p>0$.\\

It remains to handle the latter case \eqref{240118e4_269z}. 
Note that 
\begin{equation}\label{240118e4_268}
\anorm{
\det
\begin{bmatrix}
\partial_y \gamma_{\bfj}, & \partial_v \gamma_{\bfj}\\
\partial_y\partial_{\theta} \gamma_{\bfj}, & \partial_v \partial_{\theta} \gamma_{\bfj}
\end{bmatrix}
}
\simeq 2^{\mf{p}_0 j_1- \mf{p}_i j_1-
j_2},
\end{equation}
which may get very small if $j_1, j_2$ are large, and therefore we are not able to use Zahl's result \cite{Zah23} directly. To prove \eqref{240118e4_266}, observe that $|P_k f|$ is essentially constant in the vertical direction at the scale $2^{-k}$, and therefore 
\begin{equation}\label{240122e4_276}
\begin{split}
& \Norm{
\int_{
\R
}
P_k f(x-\theta, y-\gamma_{\bfj}(v_{0}(x, y), \theta))
\chi_{+}(\theta)
d\theta
}_{L^p(\R^2)}\\
& 
\lesim
\Norm{
\sup_{w\in \N, w\simeq 2^k}
\anorm{
\int_{\R}
|
P_k f(x-\theta, 
y-\gamma_{\bfj}(w2^{-k}, \theta))\chi_+(\theta)
|d\theta
}
}_{L^p(\R^2)}.
\end{split}
\end{equation}
Recall that 
\begin{equation}
\gamma_{\bfj}(v, \theta)=
c_{\mf{p}_0, 0} v^{\mf{p}_0}+ 
c_{
\mf{p}_i, 1
}
2^{
\mf{p}_0 j_1- \mf{p}_i j_1-
j_2
} v^{\mf{p}_i} \theta+ \text{smaller term}.
\end{equation}
Because of the small factor $2^{
\mf{p}_0 j_1- \mf{p}_i j_1-
j_2
}$, for two $w_1, w_2\in \N$, the slopes between the two lines 
\begin{equation}\label{240122e4_278}
\pnorm{\theta, 2^{
\mf{p}_0 j_1- \mf{p}_i j_1-
j_2
} (w_1 2^{-k})^{\mf{p}_i} \theta}
\text{ and }
\pnorm{\theta, 
2^{
\mf{p}_0 j_1- \mf{p}_i j_1-
j_2
} (w_2 2^{-k})^{\mf{p}_i} \theta}
\end{equation}
can be as small as 
\begin{equation}
2^{
\mf{p}_0 j_1- \mf{p}_i j_1-
j_2
}
|w_1-w_2|2^{-k},
\end{equation}
which is much smaller compared with $|w_1-w_2|2^{-k}$. The idea is to cut the sup in \eqref{240122e4_276} into smaller groups. More precisely, for 
\begin{equation}
A=1, 2, \dots, 2^{\mf{p}_i j_1+j_2-\mf{p}_0 j_1},
\end{equation}
define 
\begin{equation}
W_A:=\{w: w\in \N, w\simeq 2^k, w\equiv A\ (\mathrm{mod}\ 2^{\mf{p}_i j_1+j_2-\mf{p}_0 j_1})\}.
\end{equation}
We bound \eqref{240122e4_276} by 
\begin{equation}
\Norm{
\pnorm{
\sum_{A=1}^{
2^{\mf{p}_i j_1+j_2-\mf{p}_0 j_1}
}
\anorm{
\sup_{w\in W_A}
\int_{\R}
|
P_k f(x-\theta, 
y-\gamma_{\bfj}(w2^{-k}, \theta))\chi_+(\theta)
|d\theta
}^p
}^{\frac{1}{p}}
}_{L^p(\R^2)}.
\end{equation}
For each fixed $A$, we have 
\begin{equation}\label{240122e4_283}
\Norm{
\sup_{w\in W_A}
\int_{\R}
|
P_k f(x-\theta, 
y-\gamma_{\bfj}(w2^{-k}, \theta))\chi_+(\theta)
|d\theta
}_{L^2(\R^2)}
\lesim_{\delta}
2^{\delta k}\norm{f}_{L^2(\R^2)},
\end{equation}
for every $\delta>0$, by a standard duality argument and by expanding the $L^2$ norm; one can also apply Zahl \cite{Zah23} directly. We interpolate \eqref{240122e4_283} with a trivial $L^{\infty}$ bound, and can further bound \eqref{240122e4_276} by 
\begin{equation}
2^{
\frac{
\mf{p}_i j_1+j_2-\mf{p}_0 j_1
}{p}
} 2^{\delta k}\norm{f}_{L^p(\R^2)}, \ \ p\ge 2. 
\end{equation}
This finishes the proof of the desired estimate \eqref{240118e4_266} if $k\le (C_{p, \gamma})^8  j_2$, where $C_{p, \gamma}$ is a large constant depending on $p$ and $\gamma$. For the case $k\ge (C_{p, \gamma})^8  j_2$, we first note that 
\begin{equation}
k\ge (C_{p, \gamma})^4  \max\{j_1, j_2\},
\end{equation}
and we are in exactly the same situation as in \eqref{240122e4_248}. The details are left out. \\

So far we have finished all the cases \eqref{240107e4_26}, \eqref{240107e4_27}, \eqref{240107e4_28} and \eqref{240108e4_29}. In the next section we will handle the contributions from the regions $\mf{O}'_{i}$.

\section{Sharp \texorpdfstring{$L^p$}{} bounds: Edges-dominating}\label{240220section6}

Let us first remind ourselves what we mean by $\mf{O}'_{i}$. Recall that we use $\mf{E}_i$ to denote the edge of the reduced Newton diagram that connects the vertices $(\mf{p}_i, \mf{q}_i)$ and $(\mf{p}_{i+1}, \mf{q}_{i+1})$. Moreover,  
\begin{equation}
\gamma_{\mf{E}_i}(v, \theta)=
\sum_{
(\mfp, \mfq)\in \mf{E}_i
}
c_{\mfp, \mfq} v^{\mfp} \theta^{\mfq}.
\end{equation}
Consider the curve $\theta= rv^{\mf{m}_i}$, where 
\begin{equation}
\mf{m}_{i}=
\frac{
\mf{p}_{i+1}-\mf{p}_i
}{
\mf{q}_{i}-\mf{q}_{i+1}
}.
\end{equation}
and 
\begin{equation}\label{240123e4_290zzz}
r\in [
(C_{\gamma})^{-1}, C_{\gamma}
].
\end{equation}
Here $C_{\gamma}$ is a large constant depending only on $\gamma$. The reason of considering the region \eqref{240123e4_290zzz} is that outside this region, the function $\gamma_{\mf{E}_i}(v, \theta)$ is very ``small".  On this curve, we have 
\begin{equation}
\gamma_{
\mf{E}_i
}(v, \theta)=
v^{\mf{e}_i}
\sum_{(\mfp, \mfq)\in \mf{E}_i}
c_{\mfp, \mfq}
r^{\mfq}:= 
v^{\mf{e}_i} \gamma_{\mf{E}_i}(r),
\end{equation}
where 
\begin{equation}\label{240123e4_292zzz}
\mf{e}_i= \mf{p}+\mf{q}\mf{m}_i,
\end{equation}
and in \eqref{240123e4_292zzz}, the point $(\mf{p}, \mf{q})$ can be taken to be an arbitrary point on $\mf{E}_i$.\\

We will prove that 
\begin{equation}
    \Norm{
\sup_{v\in \R}
\anorm{
\int_{
\mf{O}'_i(v)
}
f(x-\theta, y-\gamma(v, \theta)) a(v, \theta)d\theta
}
}_{L^p(\R^2)}
\lesim_{
\gamma, p, a
}
\norm{f}_{L^p(\R^2)},
\end{equation}
for all 
\begin{equation}
    p> \max\{
    2, \mf{D}_{\gamma}
    \},
\end{equation}
where 
\begin{equation}
    \mf{O}'_i(v):=
    \{
    \theta: (v, \theta)\in \mf{O}'_i
    \}.
\end{equation}
For each $|v|\le \epsilon$, let
\begin{equation}
a'_i(v, \cdot): \R\to \R
\end{equation}
 be a non-negative smooth bump function supported on $(1+c_{\gamma})\mf{O}'_i(v)$, and equal to one on $\mf{O}'_i(v)$, where $c_{\gamma}>0$ is a small constant depending only on $\gamma$. Under the new notation, we need to prove 
 \begin{equation}\label{240219e6_9mmm}
    \Norm{
\sup_{v\in \R}
\anorm{
\int_{
\R
}
f(x-\theta, y-\gamma(v, \theta)) a'_i(v, \theta)d\theta
}
}_{L^p(\R^2)}
\lesim_{
\gamma, p, a
}
\norm{f}_{L^p(\R^2)}.
\end{equation}
We consider two cases separately. The first case is when 
\begin{equation}\label{240219e6_6}
    \#\{
    (\mf{p}, \mf{q})\in \mf{E}_i\cap \mc{T}(\gamma): \mf{q}\ge 2
    \}=1,
\end{equation}
and the second case is when 
\begin{equation}\label{240219e6_7}
    \#\{
    (\mf{p}, \mf{q})\in \mf{E}_i\cap \mc{T}(\gamma): \mf{q}\ge 2
    \}\ge 2.
\end{equation}
Here $\mc{T}(\gamma)$ means the Taylor support of $\gamma$. 
Let us first consider the case \eqref{240219e6_6}, as it is much easier compared with the other case. \\

Assume \eqref{240219e6_6}. What makes this case easy is that the second order derivative of $\gamma$ in the $\theta$ variable has a simple structure. Recall that the two endpoints of the edge $\mf{E}_i$ are
\begin{equation}
    (
    \mf{p}_i, \mf{q}_i
    ), \ \ (
    \mf{p}_{i+1}, \mf{q}_{i+1}
    ),
\end{equation}
with 
\begin{equation}
    \mf{q}_i> \mf{q}_{i+1}.
\end{equation}
Note that we always have 
\begin{equation}
\mf{q}_{i+1}\ge 1, \ 
    \mf{q}_i\ge 2,
\end{equation}
and therefore under the assumption \eqref{240219e6_6}, the only vertex that belongs to the set in \eqref{240219e6_6} is $(
    \mf{p}_i, \mf{q}_i
    )$. As a consequence, we obtain 
\begin{equation}\label{240219e6_15}
    |
    \partial_{\theta}^2 \gamma(v, \theta)
    |\simeq 
    |
    \partial_{\theta}^2 \gamma_{
    \mf{E}_i
    }(v, \theta)
    |
    \simeq 
    |v|^{
    \mf{p}_i
    }
    |\theta|^{
    \mf{q}_i-2
    },
\end{equation}
for every 
\begin{equation}
    (v, \theta)\in \supp(
    a'_i
    ).
\end{equation}
This favorable bound on the second derivative in $\theta$, combined with local smoothing estimates, will lead to the desired estimate \eqref{240219e6_9mmm}. We will not write down the details here. One reason is that similar calculations have appeared multiple times before. A more important reason is that the second case \eqref{240219e6_7} will be strictly more difficult, and we will add more details when we discuss that case. \\

From now on we assume \eqref{240219e6_7}. In this case, the second derivative $\partial^2_{\theta}\gamma$ has a much more complicated behavior than \eqref{240219e6_15}. More precisely, $\partial^2_{\theta}\gamma_{
\mf{E}_i
}$ contains at least two competing monomials. Denote 
\begin{equation}
    \kappa_{
    \mf{E}_i
    }(r):= 
    \sum_{
    (\mf{p}, \mf{q})\in \mf{E}_i, \mf{q}\ge 2
    } c_{
    \mf{p}, \mf{q}
    } \mf{q}(\mf{q}-1) r^{
    \mf{q}-2
    },
\end{equation}
which, similarly to the function $\gamma_{\mf{E}_i}(r)$, is obtained by restricting the function $\partial^2_{\theta}\gamma_{
\mf{E}_i
}$ to the curve
\begin{equation}
    \theta= rv^{
    \mf{m}_i
    }.
\end{equation}
Let $\{\mf{r}_{i, j}\}_{j=1}^{J_i}$ be the collection of all the distinct non-zero roots of the polynomial $\kappa_{
    \mf{E}_i
    }(r)$. 
Denote 
\begin{equation}
    \mf{O}'_{i, j}:=
    \{
    (v, r v^{\mf{m}_i})\in \mf{O}'_i: |v|\le \epsilon,  r\in [(C_{\gamma})^{-1}, C_{\gamma}], |r-\mf{r}_{i, j}|\le (C_{\gamma})^{-2}
    \},
\end{equation}
where $C_{\gamma}$ is the same as in \eqref{240123e4_290zzz}. Moreover, denote 
\begin{equation}\label{240219e6_20}
    \mf{O}'_{
    i, \mathrm{bad}
    }:=
    \bigcup_{j=1}^{J_i}
    \mf{O}'_{i, j}, 
\end{equation}
and 
\begin{equation}\label{240219e6_21}
    \mf{O}'_{i, \mathrm{good}}:=
    \mf{O}'_{i}\setminus 
    \mf{O}'_{i, 
    \mathrm{bad}}.
\end{equation}
The contribution from $\mf{O}'_{i, \mathrm{good}}$ and the contribution from $\mf{O}'_{i, \mathrm{bad}}$ will be discussed in the following two subsections separately.

\subsection{The good region}

In this subsection we handle the contribution from the good region $\mf{O}'_{
i, \mathrm{good}
}$, as defined in \eqref{240219e6_21}. We will prove that 
\begin{equation}
    \Norm{
\sup_{v\in \R}
\anorm{
\int_{
\mf{O}'_{
i, \mathrm{good}
}(v)
}
f(x-\theta, y-\gamma(v, \theta)) a(v, \theta)d\theta
}
}_{L^p(\R^2)}
\lesim_{
\gamma, p, a
}
\norm{f}_{L^p(\R^2)},
\end{equation}
for all 
\begin{equation}
    p> \max\{
    2, \mf{D}_{\gamma}
    \},
\end{equation}
where 
\begin{equation}
    \mf{O}'_{
    i, \mathrm{good}
    }(v):=
    \{
    \theta: (v, \theta)\in \mf{O}'_{
    i, \mathrm{good}
    }
    \}.
\end{equation}
For each $|v|\le \epsilon$, let 
\begin{equation}
    a'_{
i, \mathrm{good}
}(v, \cdot): \R\to \R
\end{equation}
 be a non-negative smooth bump function supported on a  sufficiently small neighborhood of $\mf{O}'_{
 i, \mathrm{good}
 }(v)$, and equal to one on $\mf{O}'_{
 i, \mathrm{good}
 }(v)$. Under the new notation, we need to prove 
 \begin{equation}\label{240219e6_23}
    \Norm{
\sup_{|v|\le \epsilon}
\anorm{
\int_{
\R
}
f(x-\theta, y-\gamma(v, \theta)) a'_{
i, \mathrm{good}
}(v, \theta)d\theta
}
}_{L^p(\R^2)}
\lesim_{
\gamma, p, a
}
\norm{f}_{L^p(\R^2)}.
\end{equation}
What makes this case good is that 
\begin{equation}
    |\partial^2_{\theta} \gamma(v, \theta)|\simeq 
    |
    \partial^2_{\theta}\gamma_{
    \mf{E}_i
    }(v, \theta)
    |\simeq 
    |v|^{\mf{p}} |\theta|^{\mf{q}-2},
\end{equation}
for all 
\begin{equation}
    (v, \theta)\in \supp(
    a'_{i, \mathrm{good}}
    ),
\end{equation}
and all 
\begin{equation}
    (
    \mf{p}, \mf{q}
    )\in \mf{E}_i, \ \mf{q}\ge 2. 
\end{equation}
This is very similar to \eqref{240219e6_15}. \\

Let $\mf{p}_0$ be the smallest $\mf{p}$ such that 
\begin{equation}
    c_{\mf{p}, 0}\neq 0.
\end{equation}
If such $\mf{p}$ does not exist, that is, $c_{\mf{p}, 0}=0$ for every $\mf{p}$, then we write $\mf{p}_0=\infty$. 
Let $\mf{L}_{\mf{E}_i}$ be the line that contains the edge $\mf{E}_i$. Write 
\begin{equation}\label{240220e6_32}
    \mf{L}_{\mf{E}_i}=
    \{
    (x, y)\in \R^2: 
    a_{
    \mf{E}_i
    } x+ 
    b_{
    \mf{E}_i
    } y = c_{
    \mf{E}_i
    } 
    \}
\end{equation}
where 
\begin{equation}
    a_{
    \mf{E}_i
    }>0, b_{
    \mf{E}_i
    }>0, c_{
    \mf{E}_i
    }>0.
\end{equation}
We consider two subcases separately. The first subcase is when the vertex $(\mf{p}_0, 0)$ does not lie to the lower left of $\mf{L}_{\mf{E}_i}$, that is,
\begin{equation}
    a_{
    \mf{E}_i
    }
    \mf{p}_0\ge c_{
    \mf{E}_i
    }.
\end{equation}
The second subcase is when the vertex $(\mf{p}_0, 0)$ lies to the lower left of $\mf{L}_{\mf{E}_i}$.

\subsubsection{The first subcase}

Recall that we need to prove \eqref{240219e6_23}. In this case, we will actually prove something stronger: 
\begin{equation}\label{240219e6_34}
     \Norm{
\sup_{|v|\le \epsilon}
\anorm{
\int_{
\R
}
f(x-\theta, y-\gamma(v, \theta)) a'_{
i, \mathrm{good}
}(v, \theta)d\theta
}
}_{L^p(\R^2)}
\lesim_{
\gamma, p, a
}
\norm{f}_{L^p(\R^2)},
\end{equation}
for all $p>2$. Let $\chi_+: \R\to \R$ be a non-negative smooth bump function supported on $(1, 3)$. Denote 
\begin{equation}
    \chi_{j_1}(v):=\chi_+(2^{j_1}v), \ \ 
    \chi_{j_2}(\theta):=\chi_+(2^{j_2}\theta). 
\end{equation}
Note that on the support of $a'_{
i, \mathrm{good}
}$, we always have 
\begin{equation}
    |\theta|\simeq |v|^{
    \mf{m}_i
    }.
\end{equation}
In other words, if 
\begin{equation}
    |\theta|\simeq 2^{-j_2},
\end{equation}
then 
\begin{equation}
    |v|\simeq 2^{-j_1}, \ \ j_1=
    j_2/\mf{m}_i.
\end{equation}
Denote 
\begin{equation}
    \chi_{\bfj, \mathrm{good}}(v, \theta):=
    \chi_{j_1}(v)\chi_{j_2}(\theta) a'_{
    i, \mathrm{good}
    }(v, \theta),
\end{equation}
where 
\begin{equation}
    \bfj=(j_1, j_2), \ \ j_1=j_2/\mf{m}_i.
\end{equation}
By the triangle inequality, to prove \eqref{240219e6_34}, it suffices to prove that 
\begin{equation}
        \Norm{
\sup_{v\in \R}
\anorm{
\int_{
\R
}
f(x-\theta, y-\gamma(v, \theta)) 
\chi_{\bfj, \mathrm{good}}(v, \theta)d\theta
}
}_{L^p(\R^2)}
\lesim
2^{-\delta_p j_2}
\norm{f}_{L^p(\R^2)},
\end{equation}
for some $\delta_p>0$. We apply the change of variable 
\begin{equation}\label{240219e6_42}
    \theta\to 2^{-j_2}\theta
\end{equation}
and rename the parameter 
\begin{equation}\label{240219e6_43}
    v\to 2^{-j_1}v,
\end{equation}
and need to prove 
\begin{equation}\label{240219e6_44}
        \Norm{
\sup_{v\in \R}
\anorm{
\int_{
\R
}
f(x-\theta, y-\gamma_{\bfj}(v, \theta)) 
\chi_{+, \mathrm{good}}(v, \theta)d\theta
}
}_{L^p(\R^2)}
\lesim
2^{j_2-\delta_p j_2}
\norm{f}_{L^p(\R^2)},
\end{equation}
where 
\begin{equation}
    \gamma_{\bfj}(v, \theta):=
    2^{
    \mf{p}_i j_1+\mf{q}_i j_2
    }
    \gamma(
    2^{-j_1}v, 2^{-j_2}\theta
    ),
\end{equation}
and 
\begin{equation}
    \chi_{
    +, \mathrm{good}
    }(v, \theta):=\chi_+(v)\chi_+(\theta) a'_{
    i, \mathrm{good}
    }(2^{-j_1}v, 2^{-j_2}\theta).
\end{equation}
On the support of $\chi_{
    +, \mathrm{good}
    }$, we know that 
\begin{equation}
    |\partial^2_{\theta}\gamma_{\bfj}|\simeq 1.
\end{equation}
This bound, combined with the fact that the cinematic curvature of $\gamma$ does not vanish constantly, will lead to the desired bound \eqref{240219e6_44}. We just need to apply Theorem \ref{231120theorem3_3}, the theorem of resolutions of singularities, to the cinematic curvature of $\gamma$, and follow precisely the steps in 
Subsection \ref{240107subsection4_2}. We leave out the details.

\subsubsection{The second subcase}

We consider the case when the vertex $(
\mf{p}_0, 0
)$ lies to the lower left of $\mf{L}_{
\mf{E}_i
}$, that is, 
\begin{equation}
    a_{
    \mf{E}_i
    } \mf{p}_0< c_{\mf{E}_i}.
\end{equation}
Here we are still using the notation in \eqref{240220e6_32}. We will prove that 
\begin{equation}
       \Norm{
\sup_{|v|\le \epsilon}
\anorm{
\int_{
\R
}
f(x-\theta, y-\gamma(v, \theta)) a'_{
i, \mathrm{good}
}(v, \theta)d\theta
}
}_{L^p(\R^2)}
\lesim_{
\gamma, p, a
}
\norm{f}_{L^p(\R^2)},
\end{equation}
for all 
\begin{equation}\label{240220e6_51}
    p> \max\{2, \mf{D}_{\gamma}\}.
\end{equation}
Indeed, we will see 
\begin{equation}
    p> \max\{
    2, \mf{d}(
    \mf{L}_{\mf{E}_i}, 
    (
    \mf{p}_0, 0
    )
    )
    \}
\end{equation}
suffices. The key difference between this subcase and the first subcase is that, after the change of variable \eqref{240219e6_42} and renaming the parameter \eqref{240219e6_43}, the new function $\gamma_{\bfj}(v, \theta)$ contains monomials with large coefficients. More precisely, 
\begin{equation}
    \gamma_{\bfj}(v, \theta)=
    2^{
    \mf{p}_i j_1+\mf{q}_i j_2
    }
    h_1(
    2^{-j_1}v
    )
    + 
    2^{\mf{p}_i j_1+\mf{q}_i j_2}
    \gamma_{\mf{r}\mf{e}}(
    2^{-j_1}v, 2^{-j_2}\theta
    ),
\end{equation}
where 
\begin{equation}
    h_1(v):=\sum_{
    \mf{p}
    } c_{
    \mf{p}, 0
    } v^{\mf{p}}
\end{equation}
and 
\begin{equation}
    \gamma_{\mf{r}\mf{e}}(v, \theta):= \gamma(v, \theta)-h_1(v).
\end{equation}
In particular, $\gamma_{\bfj}(v, \theta)$ contains the term 
\begin{equation}
    c_{
    \mf{p}_0, 0
    }
    2^{
    -\mf{p}_0 j_1+
    \mf{p}_i j_1+\mf{q}_i j_2
    }
    v^{
    \mf{p}_0
    },
\end{equation}
which has a very large coefficient when $j_1$ and $j_2$ are large. It is exactly this term that gives rise to the exponent in \eqref{240220e6_51}.

\subsection{The bad region}\label{240220subsection6_2}

In this subsection we will take care of the contributions from the bad region \eqref{240219e6_20}. More precisely, we will prove that 
\begin{equation}\label{240219e6_23kkk}
    \Norm{
\sup_{|v|\le \epsilon}
\anorm{
\int_{
\R
}
f(x-\theta, y-\gamma(v, \theta)) a'_{
i, j
}(v, \theta)d\theta
}
}_{L^p(\R^2)}
\lesim_{
\gamma, p, a
}
\norm{f}_{L^p(\R^2)},
\end{equation} 
for all $j$, all 
\begin{equation}
    p> \max\{
    2, \mf{D}_{\gamma}
    \},
\end{equation}
where for each $|v|\le \epsilon$, we define 
\begin{equation}
    \mf{O}'_{
    i, j
    }(v):=
    \{
    \theta: (v, \theta)\in \mf{O}'_{
    i, j
    }
    \},
\end{equation}
and 
\begin{equation}
    a'_{
i, j
}(v, \cdot): \R\to \R
\end{equation}
 is a non-negative smooth bump function supported on a small neighborhood of $\mf{O}'_{
 i, j
 }(v)$, and equal to one on $\mf{O}'_{
 i, j
 }(v)$. 
We apply the change of variable 
\begin{equation}
\theta\mapsto \theta+ \mf{r}_{i, j} v^{\mf{m}_i},
\end{equation}
and write \eqref{240219e6_23kkk} as 
\begin{equation}\label{240123e5_7}
\Norm{
\sup_{v\in \R}
\anorm{
\int_{
\R
}
f(x-\theta-\mf{r}_{i, j} v^{\mf{m}_i}, y-\gamma_{\star}(v, \theta)) a_{\star}(v, \theta)d\theta
}
}_{L^p(\R^2)}
\lesim_{
\gamma, p, a
}
\norm{f}_{L^p(\R^2)},
\end{equation}
where 
\begin{equation}
\gamma_{\star}(v, \theta):=
\gamma(
v, \theta+ \mf{r}_{i, j} v^{\mf{m}_i}
)
\end{equation}
and 
\begin{equation}
a_{\star}(v, \theta):=a'_{i, j}(
v, 
\theta+ \mf{r}_{i, j} v^{\mf{m}_i}
).
\end{equation}
Note that on the support of $a_{\star}$, we always have 
\begin{equation}\label{240124e5_16}
|\theta|\le c_{\gamma} |v|^{\mf{m}_i},
\end{equation}
where $c_{\gamma}$ is a small constant depending on $\gamma$. The new function $\gamma_{\star}$ may not be analytic in $v, \theta$ anymore, but there exists a positive integer $M$ such that it is analytic in $\theta$ and $v^{\frac{1}{M}}$.

To simplify the argument, we consider the case $\mf{m}_i=M=1$. Indeed, the general case can be reduced to the case $\mf{m}_i=M=1$ by renaming the parameter $v$ to $v^M$, and this does not change the sup in \eqref{240123e5_7}. Moreover, renaming the $v$ variable does not change vertical Newton distances either because $v$ is always treated as a horizontal variable. Under this simplification, the new function 
\begin{equation}\label{240124e5_17}
    \gamma_{\star}(v, \theta)=
    \gamma(v, 
    \theta+\mf{r}_{i, j}v
    )
\end{equation}
is an analytic function in $v$ and $\theta$. Recall that $(\mf{p}_i, \mf{q}_i)$ is a vertex in the reduced Newton diagram $\mc{R}\mc{N}_d(\gamma)$. Moreover, as the edge $\mf{E}_i$ contains another vertex $(
\mf{p}_{i+1}, \mf{q}_{i+1}
)$, we therefore have 
\begin{equation}
    q_{i}> q_{i+1}\ge 1.
\end{equation}

\begin{lemma}\label{240124lemma5_1}
Let $\gamma_{\star}$ be given by \eqref{240124e5_17}.
\begin{enumerate}
\item[(1)]
The point $(\mf{p}_i, \mf{q}_i)$ is a vertex of the reduced Newton diagram $\mc{R}\mc{N}_d(
    \gamma_{\star})$. 
    \item[(2)] The point $(\mf{p}_i, \mf{q}_i-2)$ is a vertex of the Newton diagram $\mc{N}_d(
    \partial^2_{\theta}
    \gamma_{\star})$. 
\end{enumerate} 
\end{lemma}
\begin{proof}[Proof of Lemma \ref{240124lemma5_1}]
Recall that 
\begin{equation}
    \gamma_{
    \mf{E}_i
    }(v, \theta)=
    \sum_{
    (\mf{p}, \mf{q})\in \mf{E}_i
    }
    c_{\mf{p}, \mf{q}} v^{\mf{p}}\theta^{\mf{q}}.
\end{equation}
For all the vertices $(\mf{p}, \mf{q})\in \mf{E}_i$, we always have $\mf{q}\le \mf{q}_i$. When applying the change of variables 
\begin{equation}
    \theta\to \theta+\mf{r}_{i, j}v,
\end{equation}
the monomial $v^{\mf{p}_i}\theta^{\mf{q}_i}$ will never be cancelled, and therefore the point $(\mf{p}_i, \mf{q}_i)$ is  still a vertex of the reduced Newton diagram $\mc{R}\mc{N}_d(
    \gamma_{\star})$. This finishes the proof of item (1). The proof of item (2) is the same. 
\end{proof}

\begin{lemma}\label{240124lemma5_2}
    Let $\mf{L}_{\mf{E}_i}$ be the line on $\R^2$ containing the edge $\mf{E}_i$. Let $\mf{L}''_{
    \mf{E}_i
    }$ be the line parallel to $\mf{L}_{\mf{E}_i}$ that passes through the point $(\mf{p}_i, \mf{q}_i-2)$. 
    \begin{enumerate}
        \item[(1)] None of the vertices in the 
        reduced Newton diagram $\mc{R}\mc{N}_d(
    \gamma_{\star})$ lies to the lower left of $\mf{L}_{\mf{E}_i}$. More precisely, if we  write $\mf{L}_{\mf{E}_i}$ as 
    \begin{equation}\label{240124e5_18}
        \{(x, y): a_{\mf{E}_i} x+ b_{\mf{E}_i} y=c_{\mf{E}_i}\}, \ \ a_{\mf{E}_i}>0, b_{\mf{E}_i}>0, c_{\mf{E}_i}>0,
    \end{equation}
    then 
    \begin{equation}
        a_{\mf{E}_i} \mf{p}+ b_{\mf{E}_i} \mf{q}\ge c_{\mf{E}_i},
    \end{equation}
    for every 
    \begin{equation}
        (\mf{p}, \mf{q})\in 
        \mc{R}\mc{N}_d(
        \gamma_{\star}).
    \end{equation}
    \item[(2)] None of the vertices in the Newton diagram $\mc{N}_d(
    \partial^2_{\theta} \gamma_{\star}
    )$ lies to the lower left of $\mf{L}''_{
    \mf{E}_i
    }$.
    \end{enumerate}
\end{lemma}
\begin{proof}[Proof of Lemma \ref{240124lemma5_2}] The proofs of item (1) and item (2) are the same, and we will write down the details for the proof of item (1).

    Note that none of the vertices of the reduced Newton diagram $\mc{R}\mc{N}_d(\gamma)$ lies to the lower left of $\mf{L}_{\mf{E}_i}$. Take one vertex 
    \begin{equation}
        (\mf{p}, \mf{q})\in \mc{T}(\gamma). 
    \end{equation}
    When applying the change of variables 
\begin{equation}
    \theta\to \theta+\mf{r}_{i, j}v,
\end{equation}
the monomial $v^{\mf{p}}\theta^{\mf{q}}$ becomes 
\begin{equation}\label{240222e6_75}
    v^{\mf{p}}(\theta+ \mf{r}_{i, j}v)^{\mf{q}}.
\end{equation}
We expand \eqref{240222e6_75}, and every resulting monomial still never lies to the lower left of $\mf{L}_{\mf{E}_i}$. This finishes the proof of item (1). 
\end{proof}

Let 
\begin{equation}\label{240220e6_71bbb}
    (\mf{p}_{\star}, \mf{q}_{\star})\in \mc{R}\mc{N}_d(
    \gamma_{\star})
\end{equation}
be the vertex that has the smallest horizontal coordinate bigger than $\mf{p}_i$. In other words, for every vertex $(\mf{p}, \mf{q})\in \mc{R}\mc{N}_d(
\gamma_{\star})$, we either have 
\begin{equation}
    \mf{p}\le \mf{p}_i
\end{equation}
or have 
\begin{equation}
    \mf{p}\ge \mf{p}_{\star}. 
\end{equation}
Lemma \ref{240124lemma5_2} suggests that we consider three cases: The first case is that there is no such $(\mf{p}_{\star}, \mf{q}_{\star})$, the second case is when 
\begin{equation}\label{240125e5_23}
    (\mf{p}_{\star}, \mf{q}_{\star})\in \mf{L}_{\mf{E}_i},
\end{equation}
and the third case is 
\begin{equation}\label{240125e5_24}
    (\mf{p}_{\star}, \mf{q}_{\star})\not\in \mf{L}_{\mf{E}_i},
\end{equation}
In the third case, we always have 
\begin{equation}
    a_{\mf{E}_i} \mf{p}_{\star}+ b_{\mf{E}_i} \mf{q}_{\star}> c_{\mf{E}_i},
\end{equation}
where we are still using the notation in \eqref{240124e5_18}.

\subsubsection{The second case}\label{240124subsection5_1}

Let us consider the case where 
\begin{equation}\label{240124e5_26}
    (\mf{p}_{\star}, \mf{q}_{\star})\in \mf{L}_{\mf{E}_i}.
\end{equation}
Recall that $\mf{m}_i=1$ and recall from \eqref{240124e5_16} that we always have 
\begin{equation}\label{240124e5_27}
    |\theta|\le c_{\gamma}|v|
\end{equation}
on the support of $a_{\star}(v, \theta)$, 
where $c_{\gamma}$ is a sufficiently small constant that is allowed to depend on $\gamma$. This means that on the support of $a_{\star}(v, \theta)$, for every 
\begin{equation}
        (\mf{p}, \mf{q})\in 
        \mc{R}\mc{N}_d(
        \gamma_{\star}) \text{ with } \mf{p}> \mf{p}_i,
    \end{equation}
    we  know that $|v|^{\mf{p}_i}|\theta|^{\mf{q}_i}$ is always much smaller compared with $|v|^{\mf{p}}|\theta|^{\mf{q}}$. \footnote{This is the motivation for ``truncating" the reduced Newton diagrams in \eqref{240124e1_18}.} The way we proceed is very similar to what we did in Section \ref{240124section4}. Let us give a brief sketch. Similarly to \eqref{240108e4_5}, we label all the vertices $(\mf{p}, \mf{q})$ in $\mc{R}\mc{N}_d(\gamma_{\star})$ with $\mf{p}\ge \mf{p}_{\star}$ by 
    \begin{equation}
        (\mf{p}_{\star+0}, \mf{q}_{\star+0}), (\mf{p}_{\star+1}, \mf{q}_{\star+1}), (\mf{p}_{\star+2}, \mf{q}_{\star+2}), \dots,
    \end{equation}
    where 
    \begin{equation}
        (\mf{p}_{\star+0}, \mf{q}_{\star+0}):=(\mf{p}_{\star}, \mf{q}_{\star}),
    \end{equation}
    and 
    \begin{equation}
        \mf{p}_{\star+0}< \mf{p}_{\star+1}< \mf{p}_{\star+2}<\dots
    \end{equation}
The reason we only consider $\mf{p}\ge \mf{p}_{\star}$ comes from \eqref{240124e5_27}. Similarly to Section \eqref{240108e4_5}, we decompose the region \eqref{240124e5_27} into smaller regions of the forms \eqref{240124e4_8}, \eqref{240220e5_10}. Depending on the form of the dominating monomial $v^{\mf{p}_{\star+\iota}}\theta^{\mf{q}_{\star+\iota}}$ we have several cases. 
The first case is 
\begin{equation}\label{240124e5_31}
    \mf{p}_{\star+\iota}\ge 1, \mf{q}_{\star+\iota}\ge 2,
\end{equation}
and
the second case is 
\begin{equation}\label{240124e5_32}
    \mf{p}_{\star+\iota}\ge 1, \mf{q}_{\star+\iota}=1.
\end{equation}
The reason for having only two cases instead of the four cases in \eqref{240107e4_26}--\eqref{240108e4_29} is as follows. The case 
\begin{equation}
    \mf{p}_{\star+\iota}=0, \mf{q}_{\star+\iota}\ge 2
\end{equation}
does not exist because 
\begin{equation}
    \mf{p}_{\star+\iota}> \mf{p}_i\ge 0.
\end{equation}
Here $(\mf{p}_i, \mf{q}_i)$ is the left end of the edge $\mf{E}_i$ we are currently considering. The case analogous to \eqref{240107e4_28} does not exist either because we have at least two vertices in the reduced Newton diagram $\mc{R}\mc{N}_d(\gamma_{\star})$.

The case \eqref{240124e5_31} can be handled in exactly the same way as Subsection \ref{240106subsection4_1}, and the case \eqref{240124e5_32} in the same way as Subsection \ref{240118subsection4_4}. Similarly to \eqref{240104e4_72} and \eqref{240125e4_276}, when trying to prove \eqref{240123e5_7}, we naturally see the exponent 
\begin{equation}\label{240125e5_36}
    p> \sup_{
    \mf{L}\in \mf{L}_{
    (\mf{p}_{\star+\iota}, \mf{q}_{\star+\iota})
    }(\gamma_{\star})
    } \mf{d}(
    \mf{L}, (\mf{p}_{0, \star}, 0)
    ),
\end{equation}
where $\mf{p}_{0, \star}$ is the biggest  integer satisfying 
\begin{equation}\label{240125e5_37}
    h_{1, \star}(v)=O(|v|^{
    \mf{p}_{0, \star}
    }),
\end{equation}
and 
\begin{equation}
    h_{1, \star}(v):=
    \sum_{\mf{p}\in \N} \frac{1}{\mf{p}!} \partial^{\mf{p}}_{v}\gamma_{\star}(0, 0)v^{\mf{p}}.
\end{equation}
This explains why we include the exponent \eqref{240125e5_36} in the definition of the vertical Newton distance $\mf{D}_{\gamma}$. \\

After dealing with the cases \eqref{240124e5_31} and \eqref{240124e5_32}, we still need to deal with edges $\mf{E}_{\star+\iota}$ in the reduced Newton diagram $\mc{R}\mc{N}_d(\gamma_{\star})$, which is the edge that connects the vertices 
\begin{equation}
(
\mf{p}_{\star+\iota}, \mf{q}_{\star+\iota}
), \ (
\mf{p}_{\star+\iota+1}, \mf{q}_{\star+\iota+1}
),
\end{equation}
with $\iota\ge 0$. What we do is to go back to the beginning of Section \ref{240220section6}, and repeat the whole argument in this section.

\subsubsection{The third case}

Here we consider the case \eqref{240125e5_24}, that is, the vertex $(\mf{p}_{\star}, \mf{q}_{\star})$ exists and 
\begin{equation}
    (\mf{p}_{\star}, \mf{q}_{\star})\not\in 
    \mf{L}_{
    \mf{E}_i
    }.
\end{equation}
See the following picture.

 \includegraphics[scale=0.6]{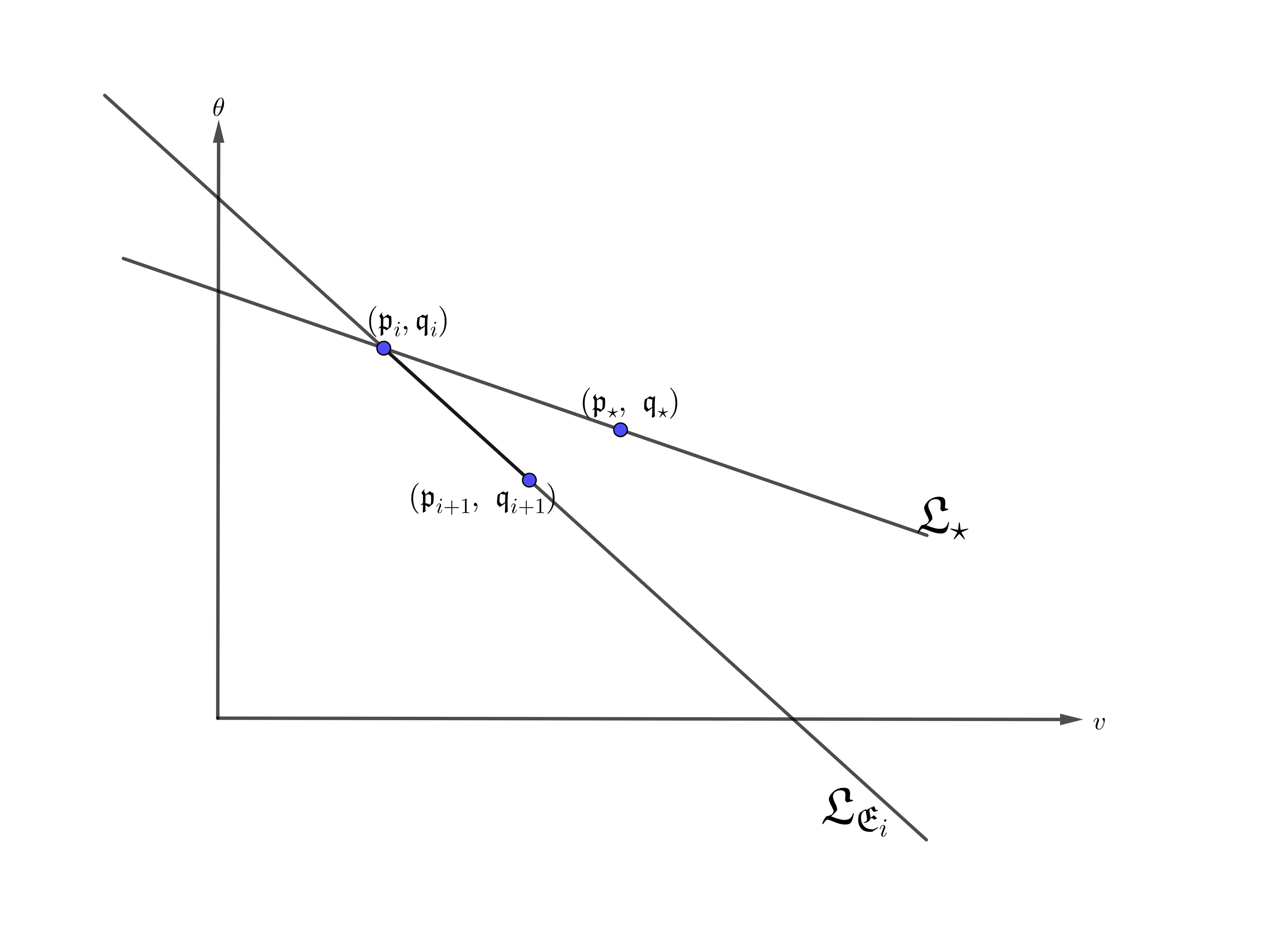}

In the above picture, the vertices $(\mf{p}_i, \mf{q}_i)$ and $(\mf{p}_{i+1}, \mf{q}_{i+1})$ are from the reduced Newton diagram $\mc{R}\mc{N}_d(\gamma)$, and the vertices $(\mf{p}_i, \mf{q}_i)$ and $(\mf{p}_{\star}, \mf{q}_{\star})$ are from the reduced Newton diagram $\mc{R}\mc{N}_d(\gamma_{\star})$. In particular, the vertex $(\mf{p}_i, \mf{q}_i)$ lies in both reduced Newton diagrams. Recall that $\mf{E}_i$ is the edge of the reduced Newton diagram $\mc{R}\mc{N}_d(\gamma)$ connecting the vertices $(\mf{p}_i, \mf{q}_i)$ and $(\mf{p}_{i+1}, \mf{q}_{i+1})$, and $\mf{L}_{\mf{E}_i}$ is the line containing $\mf{E}_i$. 

\begin{definition}
    Take two lines
    \begin{equation}
        \mf{L}_1=\{
        (x, y): a_{\mf{L}_1} x+ b_{\mf{L}_1} y=c_{\mf{L}_1}
        \}, \ \ \mf{L}_2=\{
        (x, y): a_{\mf{L}_2} x+ b_{\mf{L}_2} y=c_{\mf{L}_2}
        \},
    \end{equation}
    with 
    \begin{equation}
        a_{\mf{L}_{i}}>0, b_{\mf{L}_{i}}>0, c_{\mf{L}_{i}}>0, \ i=1, 2.
    \end{equation}
    We say that 
    \begin{equation}
        \mf{L}_1\le \mf{L}_2
    \end{equation}
    if 
    \begin{equation}
        \frac{
        a_{\mf{L}_1}
        }{b_{\mf{L}_1}} 
        \le  
        \frac{
        a_{\mf{L}_2}
        }{b_{\mf{L}_2}}.
    \end{equation}
    Similarly, we define $\mf{L}_1< \mf{L}_2$. 
\end{definition}

If we compare the current case with the case discussed in Subsection \ref{240124subsection5_1}, the biggest difference is that on the support of the amplitude function $a_{\star}(v, \theta)$, the monomial $v^{\mf{p}_i}\theta^{\mf{q}_i}$ never dominates in the case in Subsection \ref{240124subsection5_1}, but it may dominate in the current case on certain subset of the support of $a_{\star}(v, \theta)$. This can be seen more clearly from the above picture: There exists a line 
\begin{equation}
    \mf{L}\in \mf{L}_{
    (\mf{p}_i, \mf{q}_i)
    }(\gamma_{\star})
\end{equation}
satisfying 
\begin{equation}
    \mf{L}_{\star}< \mf{L}< \mf{L}_{\mf{E}_i}.
\end{equation}
As a consequence, if we repeat the argument in Subsection \ref{240106subsection4_1}  and Subsection \ref{240118subsection4_4}, similarly to what was done above the equation \eqref{240125e5_36}, then we will see the constraint 
\begin{equation}\label{240125e5_46}
    p>
    \max\Big\{
    \sup_{
    \substack{
    \mf{L}\in \mf{L}_{
    (\mf{p}_i, \mf{q}_i)
    }(\gamma_{\star})\\
    \mf{L}_{\star}< \mf{L}< \mf{L}_{\mf{E}_i}
    }
    }
    \mf{d}(
    \mf{L}, (\mf{p}_{0, \star}, 0)
    ), 
    \sup_{\iota}
    \sup_{
    \mf{L}\in \mf{L}_{
    (\mf{p}_{\star+\iota}, \mf{q}_{\star+\iota})
    }(\gamma_{\star})
    } \mf{d}(
    \mf{L}, (\mf{p}_{0, \star}, 0)
    )
    \Big\},
\end{equation}
not just the constraint in \eqref{240125e5_36}. Here
 $\mf{p}_{0, \star}$ was introduced in \eqref{240125e5_37}. The latter term in the max in \eqref{240125e5_46} appears in the definition of the vertical Newton distance $\mf{D}_{\gamma}$, while the former term does not. 
 \begin{claim}\label{240125claim5_4}
 Under the above notation, it holds that 
 \begin{equation}
    \sup_{
    \substack{
    \mf{L}\in \mf{L}_{
    (\mf{p}_i, \mf{q}_i)
    }(\gamma_{\star})\\
    \mf{L}_{\star}< \mf{L}< \mf{L}_{\mf{E}_i}
    }
    }
    \mf{d}(
    \mf{L}, (\mf{p}_{0, \star}, 0)
    )
    \le 
    \mf{D}_{\gamma}. 
\end{equation}    
 \end{claim}
\begin{proof}[Proof of Claim \ref{240125claim5_4}]
    We consider two cases separately: The case 
    \begin{equation}\label{240125e5_48}
        \mf{p}_{0, \star}\ge \mf{p}_i,
    \end{equation}
    and the case 
    \begin{equation}\label{240125e5_49}
        \mf{p}_{0, \star}< \mf{p}_i.
    \end{equation}
    For the case \eqref{240125e5_48}, we have 
    \begin{equation}
        \sup_{
    \substack{
    \mf{L}\in \mf{L}_{
    (\mf{p}_i, \mf{q}_i)
    }(\gamma_{\star})\\
    \mf{L}_{\star}< \mf{L}< \mf{L}_{\mf{E}_i}
    }
    }
    \mf{d}(
    \mf{L}, (\mf{p}_{0, \star}, 0)
    )\le 
    \mf{d}(\mf{L}_{\star}, (\mf{p}_{0, \star}, 0))\le \mf{D}_{\gamma}. 
    \end{equation}
    For the case \eqref{240125e5_49}, we have 
    \begin{equation}
        \sup_{
    \substack{
    \mf{L}\in \mf{L}_{
    (\mf{p}_i, \mf{q}_i)
    }(\gamma_{\star})\\
    \mf{L}_{\star}< \mf{L}< \mf{L}_{\mf{E}_i}
    }
    }
    \mf{d}(
    \mf{L}, (\mf{p}_{0, \star}, 0)
    )\le 
    \mf{d}(\mf{L}_{
    \mf{E}_i
    }, (\mf{p}_{0, \star}, 0)). 
    \end{equation}
    A key observation we need here is that 
    \begin{equation}
        \mf{p}_0=\mf{p}_{0, \star}, 
    \end{equation}
    which follows directly from Lemma \ref{240124lemma5_2}. This observation implies that
    \begin{equation}
        \mf{d}(\mf{L}_{
    \mf{E}_i
    }, (\mf{p}_{0, \star}, 0))
    = 
    \mf{d}(\mf{L}_{
    \mf{E}_i
    }, (\mf{p}_{0}, 0))\le \mf{D}_{\gamma}. 
    \end{equation}
    This finishes the proof of the claim. 
\end{proof}

By Claim \ref{240125claim5_4}, we see that the constraint \eqref{240125e5_46} follows from the constraint $p> \{2, \mf{D}_{\gamma}\}$.

\subsubsection{The first case}\label{240125subsection5_3}

Here we discuss the case where there is no $(\mf{p}, \mf{q})$ in the reduced Newton diagram $\mc{R}\mc{N}_d(\gamma_{\star})$ satisfying $\mf{p}> \mf{p}_i$. In other words, on the support of the amplitude function $a_{\star}(v, \theta)$, we always have 
\begin{equation}
    |\gamma_{\mf{r}\mf{e}, \star}(v, \theta)|\simeq |v|^{
    \mf{p}_i
    }
    |\theta|^{\mf{q}_i},
\end{equation}
where 
\begin{equation}
    \gamma_{\mf{r}\mf{e}, \star}(v, \theta):=
    \sum_{
    (\mf{p}, \mf{q})\in \N^2, \mf{q}\ge 1
    }
    \frac{1}{
    \mf{p}! \mf{q}!
    }\partial_{v}^{\mf{p}}
    \partial_{\theta}^{\mf{q}}\gamma_{\star}(0, 0) v^{\mf{p}}\theta^{\mf{q}}. 
\end{equation}
There are several cases to consider. The first case is when 
\begin{equation}\label{240125e5_56}
    \mf{p}_{\star}\ge 1, \mf{q}_{\star}\ge 2; 
\end{equation}
the second case is when 
\begin{equation}\label{240125e5_57}
    \mf{p}_{\star}= 0, \mf{q}_{\star}\ge 2; 
\end{equation}
the third case is when 
\begin{equation}\label{240125e5_58}
\mf{q}_i=1 \text{ and }
    (\mf{p}_i, \mf{q}_i) \text{ is the only vertex in } \mc{R}\mc{N}_d(\gamma_{\star});
\end{equation}
and the last case is when \begin{equation}\label{240125e5_59}
\mf{q}_i=1 \text{ and }
    (\mf{p}_i, \mf{q}_i) \text{ is not the only vertex in } \mc{R}\mc{N}_d(\gamma_{\star}).
\end{equation}
The case distinction here is precisely the same as \eqref{240107e4_26}--\eqref{240108e4_29} in Section \ref{240124section4}, and we refer to the paragraph right above Subsection \ref{240106subsection4_1} for the motivation behind the case distinction. \\

The proofs for the cases \eqref{240125e5_56}--\eqref{240125e5_59} are essentially the same as those presented in Subsection \ref{240106subsection4_1}--\ref{240118subsection4_4}, respectively. Here we will only explain the small differences. \\

Let us start with the case \eqref{240125e5_56}.  This case is the simplest, and we only make several remarks here. Recall the definition of $\mf{p}_{0, \star}$ in \eqref{240125e5_37}, and consider the case 
\begin{equation}
    \mf{p}_{0, \star}< \mf{p}_i.
\end{equation}
In this case, because of the item (3) in the strongly degenerate condition (Definition \ref{231122defi1_5}), we know that there must exist $(\mf{p}, \mf{q})\in \mc{R}\mc{N}_d(\gamma_{\star})$ satisfying 
\begin{equation}
    \mf{p}\le \mf{p}_{0, \star},
\end{equation}
as otherwise the function $\gamma$ itself would be strongly degenerate. Next, let us remark on the convention we made in \eqref{240125e1_19}; this convention is needed when we consider the case 
\begin{equation}
    \mf{p}_{0, \star}\ge  \mf{p}_i.
\end{equation}
Let $\mf{L}$ be the horizontal line passing through the point $(\mf{p}_i, \mf{q}_i)$.

\includegraphics[scale=0.6]{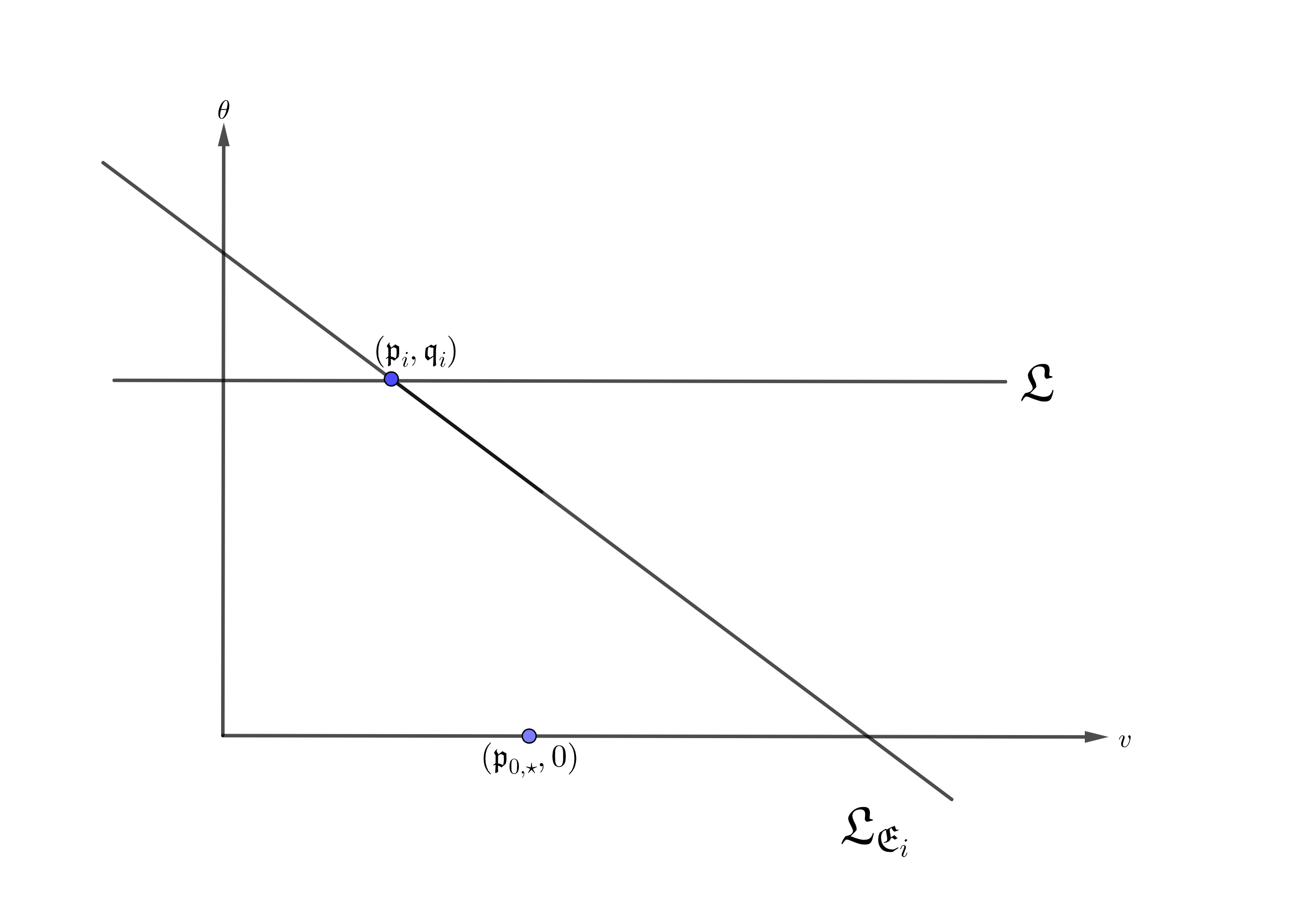}

If we run the argument in Subsection \ref{240125e5_56}, then we will see the constraint 
\begin{equation}
    p> 
    \begin{cases}
        \mf{q}_i & \text{ if } \mf{p}_{0, \star}< \infty,\\
        2 & \text{ if } \mf{p}_{0, \star}=\infty. 
    \end{cases}
\end{equation}
In particular, $\mf{q}_i$ is the vertical distance between the line $\mf{L}$ and the point $(\mf{p}_{0, \star}, 0)$. This explains the convention we made in \eqref{240125e1_19}, and finishes all the comments we would like to make on the case \eqref{240125e5_56}. \\

Next, consider the case \eqref{240125e5_57}. In this case, the reduced Newton diagram $\mc{R}\mc{N}_d(\gamma_{\star})$ consists of only one vertex $(\mf{p}_i, 0)$. The proof for this case is precisely the same as that in Subsection \ref{240107subsection4_2}. Let us only make one remark here that 
\begin{equation}
    \cine(\gamma_{\star})(v, \theta)=
    \det\begin{bmatrix}
        \partial_{\theta\theta}\gamma_{\star}, & \partial_{\theta\theta\theta}\gamma_{\star}\\
        \partial_{v\theta}\gamma_{\star}, & \partial_{v\theta\theta}\gamma_{\star}
    \end{bmatrix}
    =\cine(\gamma)(v, \theta),
\end{equation}
and therefore $\cine(\gamma_{\star})$ does not vanish constantly either, by the assumption that $\gamma$ itself is not strongly degenerate. \\

Next, consider the case \eqref{240125e5_58}. The proof for this case is the same as that in Subsection \ref{240108subsection4_3}. We only need to make sure that the new function $\gamma_{\star}$ is not strongly degenerate. Recall from
\eqref{240125e4_222} that 
\begin{equation}
    \gamma(v, \theta)=v^{
    \mf{p}_i
    } h_2(\theta)+ O(
    |v|^{\mf{p}_i+1}|\theta|
    )+ h_1(v),
\end{equation}
where 
\begin{equation}
    h_2(\theta)= \theta+ o(\theta),
\end{equation}
and the cinematic curvature of $v^{
\mf{p}_i
}h_2(v)$ does not vanish constantly. Note that when we do the change of variable 
\begin{equation}
    \theta\to \theta+ \mf{r}_{i, j} v^{\mf{m}_i},
\end{equation}
the new function $\gamma_{\star}$ can still be written in a similar form 
\begin{equation}
    \gamma_{\star}(v, \theta)=v^{
    \mf{p}_i
    } h_2(\theta)+ O(
    |v|^{\mf{p}_i+1}|\theta|
    )+ \widetilde{h}_1(v),
\end{equation}
and clearly $\gamma_{\star}$ is not strongly degenerate. \\

The proof for the last case \eqref{240125e5_59} is precisely the same as that in Subsection \ref{240118subsection4_4}.

\section{Sharp \texorpdfstring{$L^p$}{} bounds: Finishing the proof}

In Section \ref{240124section4} and Section \ref{240220section6}, we explained the general strategy of proving $L^p$ bounds for the maximal operator $\mc{M}_{\gamma}$. Let us give a quick summary. 

For the analytic function $\gamma(v, \theta)$, we first draw its reduced Newton diagram $\mc{R}\mc{N}_d(\gamma)$. According to the reduced Newton diagram $\mc{R}\mc{N}_d(\gamma)$, we decompose $\B_{\epsilon}$ into different regions, given by \eqref{240124e4_8} and \eqref{240220e5_10}, that is,
\begin{equation}
    \B_{\epsilon}=
    \pnorm{
    \bigcup_{i=1}^{I_1+1}\mf{O}_i
    }\bigcup \pnorm{
    \bigcup_{i=1}^{I_1} \mf{O}'_i
    }.
\end{equation}
 On regions $\mf{O}_i$ given by \eqref{240124e4_8}, the function $\gamma(v, \theta)$ behaves like monomials $v^{\mf{p}_i} \theta^{\mf{q}_i}$ where $(\mf{p}_i, \mf{q}_i)$ is a vertex of the reduced Newton diagram; on regions $\mf{O}'_i$ given by \eqref{240220e5_10}, the function behaves like mixed homogeneous polynomials $\gamma_{\mf{E}_i}$, where $\mf{E}_i$ is an edge of the reduced Newton diagram. 

In Section \ref{240124section4}, we handle the vertices-dominating regions $\mf{O}_i$ given by \eqref{240124e4_8}, and in Section \ref{240220section6} we handle the edges-dominating regions $\mf{O}'_i$ given by \eqref{240220e5_10}. In Section \ref{240220section6}, we further decompose the region $\mf{O}'_i$  into a good region $\mf{O}'_{i, \mathrm{good}}$ and a bad region $\mf{O}'_{i, \mathrm{bad}}$. Here by ``good" we mean points in the region are ``away" from zeros of $\partial^2_{\theta}\gamma_{\mf{E}_i}$, and by ``bad" we mean ``close". On $\mf{O}'_{i, \mathrm{good}}$, the second derivative $\partial^2_{\theta}\gamma_{\mf{E}_i}$ still behaves like a monomial, and the proof goes in the same way as that in Section \ref{240124section4}.  The region $\mf{O}'_{i, \mathrm{bad}}$ is more interesting as points in the region are close to zeros of $\partial^2_{\theta}\gamma_{\mf{E}_i}$. We use the idea of resolutions of singularities to break bad regions into smaller pieces, and then repeat the whole argument in Section \ref{240124section4} and Section \ref{240220section6}.\\

Next, we discuss how the above process terminates. In Subsection \ref{240220subsection6_2} when handling bad regions
\begin{equation}
\mf{O}'_{i, j}\subset \mf{O}'_{i, \mathrm{bad}},
\end{equation} 
we were considering the edge $\mf{E}_i$, which is given by connecting the two vertices 
\begin{equation}
(
\mf{p}_i, \mf{q}_i
), (
\mf{p}_{i+1}, \mf{q}_{i+1}
)
\end{equation}
of the reduced Newton diagram $\mc{R}\mc{N}_d(\gamma)$. 
%
%
%
%
Recall the definition of the new function $\gamma_{\star}$ in \eqref{240124e5_17}. The above process will terminate if the reduced Newton diagram $\mc{R}\mc{N}_d(\gamma_{\star})$ does not contain any edge $\mf{L}$ satisfying 
\begin{equation}\label{240220e7_4}
\mf{L}< \mf{L}_{\mf{E}_i}.
\end{equation}
It may happen that the reduced Newton diagram $\mc{R}\mc{N}_d(\gamma_{\star})$ always contains an edge $\mf{L}$ satisfying \eqref{240220e7_4}, that is, the above process does not terminate after finitely many steps. However, Xiao \cite[Theorem 4.11]{Xia17} proved that starting from certain stage we must have 
\begin{equation}\label{240129e1_42kkkk}
        \gamma_{\star+0}(v, \theta)=c_{\star} v^{b_{\star}}
        (\theta+ \mf{r}_n v^{\mf{w}_n})^{W}
        -
        c_{\star} v^{b_{\star}}
        (\mf{r}_n v^{\mf{w}_n})^{W}
        ,
    \end{equation}
    where $\gamma_{\star+0}$ was introduced in \eqref{240129e1_42kkk}, $n=1, 2, \dots,$ $b_{\star}\ge 0, c_{\star}\neq 0, \ \mf{r}_{n}\neq 0, n=1, 2, \dots,$ 
    the exponents $\mf{w}_1, \mf{w}_2, \dots$ are non-zero positive rational numbers that can be written as fractions sharing a common denominator, and $W$ is a positive integer. 
Instead of doing the change of variable
\begin{equation}
    \theta+ \mf{r}_n v^{
    \mf{w}_n
    }\to \theta, \ n=1, 2, \dots,
\end{equation}
we should do the above change of variables all at once, that is, 
\begin{equation}
    \theta+ \mf{r}_1v^{
    \mf{w}_1
    }+\mf{r}_2v^{
    \mf{w}_2
    }+\dots\to \theta,
\end{equation}
which forces the algorithm to terminate immediately.


\noindent Department of Mathematics, University of Wisconsin-Madison, Madison, WI-53706, USA \\
Email addresses:\\
mchen454@math.wisc.edu\\
shaomingguo@math.wisc.edu

\end{document}